\documentclass[11pt]{amsart}
\usepackage{amsmath, amssymb, amscd, mathrsfs, url, pinlabel,hyperref, verbatim}
\usepackage[margin=1.25in,centering,letterpaper,dvips]{geometry}
\usepackage{color,dcpic,latexsym,graphicx,epstopdf,comment}
\usepackage[all]{xy}
\usepackage{tikz}
\usepackage{enumitem,upquote}

\title{Stein fillings and $SU(2)$ representations}

\author{John A. Baldwin}
\email{john.baldwin@bc.edu}
\address{Department of Mathematics\\Boston College}

\author{Steven Sivek}
\email{sivek@math.uni-bonn.de}
\address{Mathematical Institute\\University of Bonn}

\thanks{JAB was supported by NSF Grant DMS-1406383 and NSF CAREER Grant DMS-1454865.}
\thanks{SS was supported by NSF Grant DMS-1506157.}

\newtheorem {theorem}{Theorem}
\newtheorem {lemma}[theorem]{Lemma}
\newtheorem {proposition}[theorem]{Proposition}
\newtheorem {corollary}[theorem]{Corollary}
\newtheorem {conjecture}[theorem]{Conjecture}

\newtheorem {question}[theorem]{Question}

\newtheorem {claim}[theorem]{Claim}

\numberwithin{equation}{section}
\numberwithin{theorem}{section}

\theoremstyle{definition}
\newtheorem{definition}[theorem]{Definition}

\newtheorem{remark}[theorem]{Remark}

\newlist{pcases}{enumerate}{1}
\setlist[pcases]{
  label=\bf{Case~\arabic*:}\protect\thiscase.~,
  ref=\arabic*,
  align=left,
  labelsep=0pt,
  leftmargin=0pt,
  labelwidth=0pt,
  parsep=0pt
}
\newcommand{\case}[1][]{%
  \if\relax\detokenize{#1}\relax
    \def\thiscase{}%
  \else
    \def\thiscase{~#1}%
  \fi
  \item
}

\newcommand{\ZZ}{\mathbb{Z}}
\newcommand{\Z}{\mathbb{Z}}
\newcommand{\R}{\mathbb{R}}
\newcommand{\C}{\mathbb{C}}

\newcommand{\Q}{\mathbb{Q}}

\newcommand{\ssm}{\smallsetminus}

\newcommand{\maxtb}{\overline{tb}}
\newcommand{\maxsl}{\overline{sl}}
\newcommand{\rank}{\operatorname{rank}}
\newcommand{\Hom}{\operatorname{Hom}}
\newcommand{\Aut}{\operatorname{Aut}}
\newcommand\bA{\mathbb{A}}
\newcommand{\cD}{\mathcal{D}}
\newcommand{\pbot}{p_{\mathrm{bot}}}

\newcommand\data{\mathscr{D}}

\newcommand\SHI{SHI}
\newcommand\SHIt{\underline{\SHI}}
\newcommand\KHI{KHI}
\newcommand\KHIt{\underline{\KHI}}
\newcommand{\definefunctor}[1]{\textbf{\textup{#1}}}

\newcommand\SHItfun{\definefunctor{\SHIt}}
\newcommand\KHItfun{\definefunctor{\KHIt}}
\newcommand\iinvt{\Theta} 

\DeclareFontFamily{U}{mathx}{\hyphenchar\font45}
\DeclareFontShape{U}{mathx}{m}{n}{
      <5> <6> <7> <8> <9> <10>
      <10.95> <12> <14.4> <17.28> <20.74> <24.88>
      mathx10
      }{}
\DeclareSymbolFont{mathx}{U}{mathx}{m}{n}
\DeclareFontSubstitution{U}{mathx}{m}{n}
\DeclareMathAccent{\widecheck}{0}{mathx}{"71}

\newcommand{\inr}{\operatorname{int}}
\newcommand{\img}{\operatorname{Im}}
\newcommand{\hfhat}{\widehat{HF}}

\newcommand{\ad}{\operatorname{ad}}

\newcommand{\id}{\operatorname{id}}

\newcommand{\coker}{\operatorname{coker}}

\newcommand{\knotdraw}[1]{%
  \begin{tikzpicture}[baseline=-\the\dimexpr\fontdimen22\textfont2\relax] #1 \end{tikzpicture}}
\tikzset{every picture/.style=semithick}
\newcommand{\skeinpositive}{\knotdraw{%
    \draw[->] (-0.1,0.1) -- (-0.25,0.25);
    \draw[->] (-0.25,-0.25) -- (0.25,0.25);
    \draw (0.25,-0.25) -- (0.1,-0.1); }}
\newcommand{\skeinnegative}{\knotdraw{%
    \draw[->] (0.25,-0.25) -- (-0.25,0.25);
    \draw (-0.25,-0.25) -- (-0.1,-0.1);
    \draw[->] (0.1,0.1) -- (0.25,0.25); }}
\newcommand{\skeinresolve}{\knotdraw{%
    \draw[->] (-0.25,-0.25) arc (-45:45:0.3535);
    \draw[->] (0.25,-0.25) arc (225:135:0.3535); }}
\newcommand{\unknot}{ \knotdraw{\draw circle (0.25);} }
\newcommand{\positivemeridian}{\knotdraw{%
    \draw[->] (0,0) [insert path={+ (110:0.3 and 0.1) arc (110:430:0.3 and 0.1)}];
    \draw (0,-0.25) -- (0,-0.15);
    \draw[->] (0,-0.05) -- (0,0.25); }}
\newcommand{\negativemeridian}{\knotdraw{%
    \draw[->] (0,0) [insert path={+ (430:0.3 and 0.1) arc (430:110:0.3 and 0.1)}];
    \draw (0,-0.25) -- (0,-0.15);
    \draw[->] (0,-0.05) -- (0,0.25); }}
\newcommand{\upstrand}{\knotdraw{%
    \useasboundingbox (-0.25,-0.25) rectangle (0.25,0.25);
    \draw[->] (0,-0.25) -- (0,0.25); }}

\begin{document}

\begin{abstract}
We recently defined invariants of contact 3-manifolds using a version of instanton Floer homology for sutured manifolds. In this paper, we prove that if several contact structures on a 3-manifold are induced by Stein structures on a single 4-manifold with distinct Chern classes modulo torsion then their contact invariants in sutured instanton homology are linearly independent.  As a corollary, we show that if a 3-manifold bounds a Stein domain that is not an integer homology ball then its fundamental group admits a nontrivial homomorphism to $SU(2)$. We   give several new applications of these results,  proving the existence of nontrivial and irreducible $SU(2)$ representations for a variety of $3$-manifold groups.
\end{abstract}

\maketitle

%
%
%
%

\section{Introduction} \label{sec:intro}

In \cite{bs-shi}, we defined invariants of contact 3-manifolds  using Kronheimer-Mrowka's sutured instanton Floer homology. These invariants are formally similar to those of Ozsv{\'a}th-Szab{\'o} and Honda-Kazez-Mati{\'c} in the Heegaard Floer setting. Our motivation was to take advantage of features unique to the instanton  setting to prove new results linking contact geometry to topology. The most important such feature is the relationship between instanton Floer homology and the fundamental group. In this paper, we use that relationship in combination with our contact invariants to establish the following new  connection between  Stein fillings of a closed 3-manifold $Y$ and representations from $\pi_1(Y)$ to $SU(2)$.

\begin{theorem} \label{thm:main}
If $Y$ is the boundary of a Stein domain which is not an integer homology ball, then there is a nontrivial homomorphism $\pi_1(Y) \to SU(2)$.
\end{theorem}

Beyond the intrinsic appeal of this connection, we show that Theorem \ref{thm:main} (together with our related Theorem \ref{thm:mainqhs}) is also a useful tool for establishing the existence of both nontrivial and irreducible $SU(2)$ representations for a variety of 3-manifold groups. 

For context, recall that understanding representations of 3-manifold groups into $SL_2(\C)$ and its compact real form $SU(2)$ has been critical   in low-dimensional topology, whether for studying hyperbolic 3-manifolds, incompressible surfaces \cite{culler-shalen},  Dehn surgery \cite{cgls, km-p}, or quantum invariants \cite{km-khovanov}. Even so,  some fundamental  questions about the existence of such representations  remain unanswered. The most basic, Problem 3.105 (A) on Kirby's problem list, asks whether the $SU(2)$ representation variety is nontrivial for every compact 3-manifold with nontrivial fundamental group \cite{kirby-list}. For closed 3-manifolds, we express this in the form of the conjecture below, a strengthening of the Poincar{\'e} Conjecture.

\begin{conjecture} \label{conj:su2-reps}
If $Y\not\cong S^3$ is a closed 3-manifold, then there is a nontrivial homomorphism $\pi_1(Y) \to SU(2)$.
\end{conjecture}

\begin{remark}
Zentner \cite{zentner2} has recently proven a version of this conjecture with $SL_2(\C)$ in place of $SU(2)$.  He  relies on geometrization and the  fact that every hyperbolic 3-manifold automatically admits a nontrivial $SL_2(\C)$ representation, which  is far from clear for $SU(2)$. 
\end{remark}

Note that Conjecture \ref{conj:su2-reps} is easily satisfied if $Y$ is not an integer homology sphere, since there are then nontrivial, reducible representations \[\pi_1(Y) \to H_1(Y) \to U(1)\subset SU(2).\]  If $Y$ is an integer homology sphere, then nontrivial representations were previously   known to exist in at least the following cases:
\begin{itemize}
\item when the Casson invariant $\lambda(Y)$ is nonzero \cite{akbulut-mccarthy};
\item \label{i:symplectic-filling} when $Y$ admits a co-orientable taut foliation, or a symplectic filling $(W,\omega)$ with $b^+_2(W) \geq 1$ \cite{km-p};
\item when $Y$ is Seifert fibered \cite{fs-seifert}, a graph manifold \cite{zentner}, or more, generally, the branched double cover of a nontrivial knot in $S^3$ \cite{cns,zentner} (using deep results of \cite{km-excision});
\item when $Y$ is surgery on a nontrivial knot in $S^3$ \cite{km-dehn}.
\end{itemize}

Theorem \ref{thm:main} provides a new  method of attacking Conjecture \ref{conj:su2-reps}, using Stein surfaces, and expands our repertoire of 3-manifolds known to satisfy the conjecture. To wit, in Section \ref{sec:examples}, we show that Theorem \ref{thm:main} can be used to give a new, simple proof of Fintushel-Stern's result that  Seifert fibered spaces satisfy Conjecture \ref{conj:su2-reps} (see Theorem \ref{thm:sfs-nontrivial}). We then use it to prove the existence of nontrivial $SU(2)$ representations for 3-manifolds where  previously existing methods do not appear to suffice, per the following theorem.

\begin{theorem} \label{thm:intro-hyperbolic-examples}
There are infinitely many hyperbolic integer homology spheres $Y_n$ such that
\begin{itemize}
\item $Y_n$ has Casson invariant zero;
\item $Y_n$ is not a branched double cover of a knot in $S^3$;
\item $Y_n$ bounds a negative definite Stein manifold which is not a homology ball.
\end{itemize}
By the latter property, there is a nontrivial homomorphism $\pi_1(Y_n) \to SU(2)$ for all $n$.
\end{theorem}

In a different, somewhat amusing direction, Theorem \ref{thm:main}  implies the following.

\begin{theorem}
\label{thm:stein-contractible}
Every Stein domain bounded by a homotopy 3-sphere is contractible.
\end{theorem}
We give the proof here as it is very short.
\begin{proof}
Suppose $Y$ is a homotopy $3$-sphere bounding a Stein domain $W$. Since every homomorphism $\pi_1(Y)\to SU(2)$ is  trivial, Theorem \ref{thm:main} implies that $W$ is an integer homology ball. Any Stein domain  can be obtained from the $4$-ball by attaching handles of index $1$ and $2$.  In particular, $W$ can be obtained in reverse by thickening $Y$ and  then attaching handles of index $2$, $3$, and $4$.  Thus, $\pi_1(Y)$ surjects onto $\pi_1(W)$, which implies that $\pi_1(W)$ is trivial. As $W$ is an integer homology $4$-ball with trivial $\pi_1$, the Hurewicz theorem implies that all homotopy groups of $W$ vanish. The Whitehead theorem then says that $W$ is contractible.
\end{proof}

Of course, Theorem \ref{thm:stein-contractible} also follows from the Poincar{\'e} Conjecture together with Eliashberg's foundational result \cite{eliashberg-filling} that $S^3$ bounds a unique Stein domain, of the form $(B^4,J)$, but our proof is unique in that it does not make any use of Ricci flow or  holomorphic curves.

Theorem \ref{thm:main} and most of the other results in this article follow from a new theorem on the relationship between Stein fillings and the rank of sutured instanton homology, proved using the contact invariants we defined in \cite{bs-shi}. We describe this relationship below.

Suppose $(Y,\xi)$ is a closed contact 3-manifold with base point $p$. Removing a Darboux ball around $p$, we obtain a balanced sutured manifold $Y(p) = (Y \ssm B^3, S^1)$. In \cite{bs-shi}, we defined an element 
\[ \iinvt(\xi) \in \SHItfun(-Y(p)) \]
which is an invariant of $\xi|_{Y \ssm B^3}$ up to isotopy rel boundary.  Here, $\SHItfun$ refers to our \emph{natural} refinement of Kronheimer-Mrowka's sutured instanton homology \cite{km-excision, bs-naturality}. A Stein domain $(W,J)$ with boundary $Y$ induces a natural contact structure on $Y$,
\[ \xi = TY \cap J(TY), \] consisting of the complex lines tangent to $Y$, and hence an element $\iinvt(\xi)$ of $\SHItfun(-Y(p))$.  Our main  technical result is the following, which proves \cite[Conjecture~1.6]{bs-shi}.

\begin{theorem} \label{thm:main-linear-independence}
Suppose $W$ is a compact 4-manifold with boundary $Y$, and that $W$ has Stein structures $J_1,\dots,J_n$ which induce contact structures $\xi_1,\dots,\xi_n$ on $Y$.  If the Chern classes $c_1(J_1),\dots,c_1(J_n)$ are distinct as elements of $H^2(W;\R)$, then the invariants \[\iinvt(\xi_1), \dots, \iinvt(\xi_n)\in \SHItfun(-Y(p))\] are linearly independent; in particular, the rank of $\SHItfun(-Y(p))$ is at least $n$.
\end{theorem}

Theorem \ref{thm:main-linear-independence} is an instanton Floer analogue of a theorem of Plamenevskaya \cite{plamenevskaya} regarding the contact invariant in Heegaard Floer homology (which was, in turn, inspired by work of Lisca-Mati{\'c} \cite{lisca-matic} in Seiberg-Witten theory). As such, Theorem \ref{thm:main-linear-independence} can be viewed as evidence for Kronheimer-Mrowka's conjectured isomorphism \cite[Conjecture~7.24]{km-excision} \[\SHItfun(Y(p)) \cong \hfhat(Y) \otimes \C,\] relating instanton and Heegaard Floer homologies. Our proof is inspired by Plamenevskaya's, but requires  many new ideas as  our basic geometric setup is  quite different  and because some useful structure in Heegaard Floer homology is  as yet unavailable in the instanton setting. 

The application of Theorem~\ref{thm:main-linear-independence} to $SU(2)$ representations, as expressed in Theorem \ref{thm:main}, comes from  an isomorphism
\begin{equation}\label{eqn:isoshi} \SHItfun(-Y(p)) \cong  I^\#(Y) \otimes \C, \end{equation}
where $I^\#(Y) := I^\natural(Y,U)$ is the singular instanton knot homology of the unknot $U \subset Y$, as defined in   \cite{km-khovanov}.  The latter is a form of Morse homology for a Chern-Simons functional whose space of critical points is naturally identified with the representation variety \[R(Y):=\Hom(\pi_1(Y), SU(2)).\] To see how Theorem \ref{thm:main} follows from Theorem \ref{thm:main-linear-independence},  suppose $Y$ is the boundary of a Stein domain $(W,J)$ which is not an integer homology ball. Assume, for a contradiction, that the  trivial homomorphism is the only element of $R(Y)$.  Then $\rank I^\#(Y)=1$ since the trivial homomorphism is a nondegenerate critical point of the Chern-Simons functional (see Section \ref{sec:representations}).  By an argument involving Fr{\o}yshov's $h$-invariant \cite{froyshov} (again, see Section \ref{sec:representations}), we may assume that $c_1(J)$ is nonzero in $H^2(W;\R)$. It then follows that $J$ and its conjugate Stein structure satisfy the hypotheses of Theorem \ref{thm:main-linear-independence}, which, together with the isomorphism \eqref{eqn:isoshi}, implies that $\rank I^\#(Y)\geq 2$, a contradiction.

Theorem \ref{thm:main-linear-independence} can also be used to understand $SU(2)$ representations for rational homology spheres. In this case, the natural question to ask is about the existence of \emph{irreducible} rather than nontrivial  representations (since  nontrivial representations always exist for  non integer homology spheres, as explained above). Not all rational homology spheres admit irreducible $SU(2)$ representations (e.g.,  lens spaces), so this is a more delicate question. 

To set the stage for our results on the existence of irreducible representations, recall that work of Scaduto \cite{scaduto} shows that if $Y$ is a rational homology sphere then \[\rank I^\#(Y) \geq |H_1(Y)|.\] Following the Heegaard Floer literature, we call $Y$  an \emph{instanton L-space} if equality holds. An important principle underlying  our main results and their applications (which is also at play in the proof sketch of Theorem \ref{thm:main} above) is that one can show, under favorable conditions (such as \emph{cyclical finiteness} below),   that $\pi_1(Y)$ admits an irreducible $SU(2)$ representation whenever $Y$ is not an instanton L-space. We use this principle in combination with Theorem \ref{thm:main-linear-independence} to prove the following.

\begin{theorem} \label{thm:mainqhs}
Suppose $Y$ is a rational homology sphere with $\pi_1(Y)$ cyclically finite. If $Y$ bounds a $4$-manifold $W$ with Stein structures $J_1,\dots,J_n$ whose Chern classes are distinct in $H^2(W;\R)$, and $n>|H_1(Y)|$, then there is an irreducible representation $\pi_1(Y)\to SU(2)$.
\end{theorem}

The \emph{cyclical finiteness} of $\pi_1(Y)$ amounts to  a technical condition on certain finite cyclic covers of $Y$ which is satisfied, for instance, if the universal abelian cover of $Y$ is a rational homology sphere. As we shall see (Proposition \ref{prop:morse-bott-criterion}), it ensures that reducible representations  are  Morse-Bott nondegenerate critical points of the Chern-Simons functional. This   enables us to prove  (Theorem \ref{thm:reducibles-morse-bott}) that if $\pi_1(Y)$ is cyclically finite and $Y$ is not an instanton L-space then $\pi_1(Y)$ admits an irreducible $SU(2)$ representation. So, to prove Theorem \ref{thm:mainqhs}, one need only show that any $Y$ satisfying the hypotheses of the theorem is not an instanton L-space, but that follows immediately from Theorem \ref{thm:main-linear-independence}.


To put Theorem \ref{thm:mainqhs} and   general questions regarding  the existence of irreducible $SU(2)$ representations in context, recall that little is known even about which Dehn surgeries on knots admit  irreducible $SU(2)$ representations. In \cite{km-dehn}, Kronheimer-Mrowka proved that the fundamental group of $r$-surgery on a nontrivial knot in $S^3$ admits an irreducible $SU(2)$ representation for any rational $r$ with $|r|\leq 2$. This was later strengthened by Lin  \cite{lin}.  However, the following basic question, posed by Kronheimer-Mrowka in \cite{km-dehn}, remains open.

\begin{question}
Suppose $K$ is a nontrivial knot in $S^3$. Is  there necessarily an irreducible representation $\pi_1(S^3_r(K))\to SU(2)$ for $r=3$ and $4$?\end{question}

\begin{remark}
Note that there is not always an irreducible representation for $r=5$, as $5$-surgery on the right-handed trefoil is a lens space.
\end{remark}

Lin's results provide an affirmative answer to this question  when $K$ is amphichiral. Using Theorem \ref{thm:mainqhs}, we can also provide an affirmative answer for knots whose mirrors have nonnegative maximal self-linking numbers. Such knots include  mirrors of strongly quasipositive knots. More generally,   from a combination of Theorem \ref{thm:main-linear-independence} and  a careful study of instanton L-spaces obtained by Dehn surgery (in particular, Theorem \ref{thm:l-space-slopes}), we prove the following.

\begin{theorem} \label{thm:intro-irreps-sl-positive}
Suppose $\maxsl (\overline K)\geq 0$ and fix a rational number $r = p/q > 0$. Then there is an irreducible representation \[\pi_1(S^3_{r}(K)) \to SU(2)\] if no zero of $\Delta_K(t^2)$ is a $p$th  root of unity, where $\Delta_K(t)$ is  the Alexander polynomial of $K$.
\end{theorem}

\begin{remark}The condition on $\Delta_K(t^2)$ guarantees that $\pi_1(S^3_{r}(K))$ is cyclically finite, and  is automatically satisfied when $p=3$ or $4$. 
\end{remark}

For example, the $5_2$ knot satisfies $\maxsl(\overline{5_2})=1$. Moreover, no zero of $\Delta_{5_2}(t^2) = 2t^2 - 3 + 2t^{-2}$ is a root of unity. We may therefore conclude from Theorem \ref{thm:intro-irreps-sl-positive} that there is an irreducible representation
\[ \pi_1(S^3_{r}(5_2)) \to SU(2) \]
for all rational $r > 0$.

\begin{remark} It is worth pointing out that our method of proving the existence of irreducible representations for rational homology spheres via instanton Floer means differs in an interesting way from  previous such methods. Namely, ours uses a result (Theorem \ref{thm:main-linear-independence}) about a version of instanton Floer homology  defined for any closed $3$-manifold, whereas previous methods (like that  in \cite{km-dehn}) used Floer's original construction for integer homology spheres together with clever tricks involving holonomy perturbations.
\end{remark}

We conclude by remarking that Theorems~\ref{thm:main-linear-independence} and \ref{thm:mainqhs} are often easy to apply in practice thanks to work of Gompf. Recall that a Stein domain $(W,J)$ has a handlebody decomposition specified by an oriented Legendrian link $L$ in the tight contact structure on some $\#^k (S^1 \times S^2)$, in which we attach 2-handles to $\natural^k (S^1 \times B^3)$ along each component $L_i \subset L$ with framing $tb(L_i) {-} 1$.  Gompf \cite[Proposition~2.3]{gompf} gave an explicit formula for the Chern class $c_1(J) \in H^2(W;\Z)$, which is easiest to state when there are no 1-handles, as follows.
\begin{theorem}[Gompf]
\label{thm:gompf}
Let $(W,J)$ be a Stein domain built by attaching $(tb(L_i)-1)$-framed 2-handles to $B^4$ along a Legendrian link $L = L_1 \cup \dots \cup L_k$ in $S^3$.  Then $H_2(W)$ has a basis $\Sigma_1,\dots,\Sigma_k$ built by gluing Seifert surfaces for each $L_i$ to the cores of their 2-handles, and
\[ \langle c_1(J),\Sigma_i\rangle = r(L_i), \]
where $r(L_i)$ is the rotation number of $L_i$, for all $i$.
\end{theorem}

%
In particular, if one can find $n$ Legendrian representatives of $L$ for which the tuple of rotation numbers $(r(L_1),\dots,r(L_k))$ takes $n$ different values while the tuple $(tb(L_1),\dots,tb(L_k))$ remains constant then Theorem \ref{thm:gompf} together with Theorem \ref{thm:main-linear-independence} tells us that the rank of $I^\#(Y)$ is at least $n$, where $Y$ is the result of Legendrian surgery on $L$. This principle is used in the proof of Theorem \ref{thm:intro-irreps-sl-positive}.

\subsection{Organization} In Section \ref{sec:background}, we provide background on Donaldson invariants, instanton Floer homology, open book decompositions, Stein manifolds and Lefschetz fibrations, and our contact invariant in sutured instanton homology. Section \ref{sec:stein-fillings} is the heart of this paper, where we prove our main technical result, Theorem \ref{thm:main-linear-independence}. In Section \ref{sec:representations}, we elaborate on the relationship between sutured instanton homology and $SU(2)$ representations, and we address questions of nondegeneracy. In doing so, we prove Theorems \ref{thm:main} and \ref{thm:mainqhs}. We also  develop a better understanding  in Section \ref{sec:representations} of when manifolds obtained by Dehn surgery are instanton L-spaces, proving analogues of results in the Heegaard and monopole Floer settings. As discussed above, such results are important in proving the existence of irreducible $SU(2)$ representations; they  may also be of independent interest. Finally, in Section \ref{sec:examples}, we give several applications of our apparatus, proving Theorems \ref{thm:intro-hyperbolic-examples} and \ref{thm:intro-irreps-sl-positive}.

\subsection{Acknowledgments}

We thank Lucas Culler, Stefan Friedl,  Tye Lidman, and Tom Mrowka for helpful conversations.

%
%
%
%

\section{Background} \label{sec:background}

In this section, we provide   reviews of   the background material necessary for this paper. Much of our discussion here is adapted from \cite{bs-naturality,bs-shi}.

\subsection{Donaldson invariants}
\label{ssec:donaldson}
We recall  below some basic facts about Donaldson invariants of smooth 4-manifolds and their relationship with the Seiberg-Witten invariants.

Suppose $X$ is a  closed, oriented, smooth  4-manifold with $b_1(X) = 0$ and $b^+_2(X) > 1$. Fix a homology orientation. For every  $w \in H^2(X)$, Donaldson \cite{donaldson-invariants} defines a linear map \[ D^w_X: \bA(X) \to \R, \]
where $\bA(X)$ is the  symmetric algebra on \[H_2(X;\R) \oplus H_0(X;\R).\]  This map is defined roughly as follows. Let $E \to X$ be  a $U(2)$-bundle  with $c_1(E) = w$, and set \[k=c_2(E) - \frac{1}{4}c_1(E)^2.\] For each $\lambda \in \bA(X)$, one obtains a  number $q_{k,w}(\lambda)$ by evaluating a certain cohomology class $\mu(\lambda)$ on a fundamental class of the \emph{ASD moduli space} for $E$, whose dimension varies linearly with $k$.  
The \emph{Donaldson invariant} $D^w_X$ is then obtained  by summing  the corresponding  maps \[q_{k,w}:\bA(X) \to \R\] for a fixed $w$ (as $c_2(E)$ and, hence, $k$ varies). 
The 4-manifold $X$ is said to have \emph{simple type} if \[D^w_X(x^2z) = 4D^w_X(z)\] for any $w$ and any $z \in \bA(X)$, where $x$ is the class of a point. It is known that a $4$-manifold has simple type if it contains a \emph{tight surface}, which is a closed, embedded surface $S$ of genus at least 2 and self-intersection $2g(S) - 2$ \cite[Theorem~8.1]{km-embedded}. For $X$ having simple type, Kronheimer-Mrowka defined, for any $h \in H_2(X)$, the \emph{Donaldson series}
\[ \cD^w_X(h) = D^w_X\left(\left(1+\frac{x}{2}\right)e^h\right) = \sum_{d=0}^\infty \frac{D^w_X(h^d)}{d!} + \frac{1}{2} \sum_{d=0}^\infty \frac{D^w_X(xh^d)}{d!}, \] and proved the following structure theorem \cite[Theorem~1.7]{km-embedded}.

\begin{theorem} \label{thm:km-structure}
Suppose $b_1(X) = 0$ and $b^+_2(X) > 1$ is odd, that $X$ has simple type, and that its Donaldson invariants are not all  zero. Then there are finitely many classes $K_1,\dots,K_s \in H^2(X;\Z)$ and nonzero constants $\beta_1,\dots,\beta_s \in \Q$ such that
\[ \cD^w_X(h) = e^{Q(h) / 2} \sum_{r=1}^s (-1)^{(w^2+K_r\cdot w)/2} \beta_r e^{K_r \cdot h}, \]
where $Q(h) = h\cdot h$ is the intersection form on $X$. Moreover, if $R$ is a smoothly embedded, oriented, homologically nontrivial surface in $X$ with nonnegative self-intersection, then \[|K_r \cdot R| + R\cdot R \leq 2g(R)-2\] for all $r=1,\dots,s$.
\end{theorem}
The $K_r$ above are called the (\emph{Donaldson}) \emph{basic classes} of $X$. We recall that $K$ is a basic class if and only if $-K$ is.   Witten's conjecture \cite{witten} asserts the following relationship  between the Donaldson and Seiberg-Witten invariants.
\begin{conjecture} \label{conj:witten-conjecture}
Suppose  that $b_1(X)=0$, $b_2^+(X) > 1$ is odd, and  $X$ has \emph{Seiberg-Witten simple type}.  Then $X$ has simple type; the basic classes of $X$ are precisely the Seiberg-Witten basic classes of $X$ (those $K$ for which $SW_X(K)$ is nonzero); and there is a nonzero constant $c(X)$ depending only on $X$ such that \[\beta_r = c(X) \cdot SW_X(K_r)\] for all $r=1,\dots,s$. 
\end{conjecture}
Building on work of Feehan-Leness \cite{feehan-leness}, Kronheimer-Mrowka established the following special case of Witten's conjecture \cite[Corollary~7]{km-p}.
\begin{theorem} \label{thm:witten-conjecture}
Suppose  that: 
\begin{itemize}
\item $X$ is a symplectic 4-manifold with $H_1(X) = 0$ and $b^+_2(X) > 1$, 
\item $X$ has the  same Betti numbers $b^{\pm}_2(X)$ as a smooth hypersurface in $\mathbb{CP}^3$ of even degree   at least 6, and 
\item $X$ contains  a tight surface of genus at least 2 and a sphere of self-intersection $-1$. 
\end{itemize} Then $X$ satisfies the conclusions of Conjecture~\ref{conj:witten-conjecture}.
\end{theorem}

\begin{remark}
Readers familiar with \cite{km-p} should observe that the hypotheses of Theorem~\ref{thm:witten-conjecture} imply the more general ones used in that paper. Indeed, since $X$ is symplectic with $H_1(X)=0$, it is automatically true that $X$ has Seiberg-Witten simple type; that $b_2^+(X)$ is odd (since $b_2^+-b_1+1$ is even for almost complex manifolds); that $b^{\pm}_2(X)$ determines the Euler characteristic and signature of $X$; and that $H^2(X)$ has no 2-torsion.
\end{remark}

Lastly, we record the following observation; see e.g.\ \cite[Proposition~2.9]{sivek-donaldson}.

\begin{remark} \label{rmk:blowup-donaldson} If $X \# \overline{\mathbb{CP}^2}$ satisfies Conjecture~\ref{conj:witten-conjecture}, then so does $X$.
\end{remark}

\subsection{Instanton Floer homology}
\label{ssec:floer}
We recall below the construction of instanton Floer homology, which is used to define sutured instanton homology.

Let $(Y,\alpha)$ be a pair consisting of a closed 3-manifold $Y$ and an oriented, smoothly embedded 1-cycle $\alpha \subset Y$ intersecting some smoothly embedded surface in an odd number of points.  We   associate the following data to this pair:
\begin{itemize}
\item A Hermitian line bundle $w \to Y$ with $c_1(w)$ Poincar\'e dual to $\alpha$;
\item A $U(2)$ bundle $E \to Y$ equipped with an isomorphism $\theta: \wedge^2 E \to w$.
\end{itemize}
Let $\mathcal{C}$ be the space of $SO(3)$ connections on $\operatorname{ad}(E)$. The group $\mathcal{G}$  of ``determinant-1" gauge transformations of $E$  (the automorphisms of $E$ that respect $\theta$) acts naturally on $\mathcal{C}$, and the  instanton Floer homology group $I_*(Y)_\alpha$ is the $\ZZ/8\ZZ$-graded $\C$-module arising as the Morse homology for the Chern-Simons functional on $\mathcal{B} = \mathcal{C}/\mathcal{G}$, as in \cite{donaldson-book}. 

The empty manifold is considered to have Floer homology equal to $\C$.

Given a closed, embedded surface $R \subset Y$ of genus $g(R) \geq 1$ such that $\alpha \cdot R$ is odd, the cap product with a certain associated class $\mu(R) \in H^2(\mathcal{B})$ defines an endomorphism of degree 2, which we also denote by
\[ \mu(R): I_*(Y)_\alpha \to I_{*-2}(Y)_\alpha. \]
Any two such operators $\mu(R)$ and $\mu(S)$ commute.  Moreover, the eigenvalues of $\mu(R)$ all have the form $\pm 2k$ or $\pm 2ki$, where $k$ is an integer satisfying \[0 \leq k \leq g(R)-1,\] by \cite[Corollary~7.2]{km-excision}.
\begin{definition} \label{def:instanton-floer-top}
Given $(Y,\alpha)$ and a surface $R\subset Y$ of genus at least 2 as above, we define
\[ I_*(Y|R)_\alpha \subset I_*(Y)_\alpha \]
as the generalized eigenspace of $\mu(R)$ with eigenvalue $2-2g(R)$.
\end{definition}

\begin{remark}The $(2g(R)-2)$-generalized eigenspace is used instead in \cite{km-excision, bs-naturality}, but Definition~\ref{def:instanton-floer-top} is consistent with our choice in \cite{bs-shi} and produces isomorphic invariants.
\end{remark}

\begin{remark}For a given triple $(Y,\alpha,R)$, the constructions above technically depend on the choice of triple $(w,E,\theta)$, but any two  such choices result in $\C$-modules that are related by canonical isomorphisms which are well-defined up to sign (see \cite[Section~4]{km-khovanov}). So, technically, these groups form what we call a \emph{$\{\pm 1\}$-transitive system of $\C$-modules} in \cite{bs-naturality}. It is really this kind of system that we are  referring to when we write $I_*(Y|R)_\alpha$ or $I_*(Y)_\alpha$. However, we will generally gloss over that subtlety in this paper, and think of these systems as honest $\C$-modules.
\end{remark}

Suppose  $R_0$ and $R_1$ are embedded surfaces in $Y_0$ and $Y_1$ as above. A cobordism $(W,\nu)$ from $(Y_0,\alpha_0)$ to $(Y_1,\alpha_1)$ together with an embedded surface $R_W\subset W$ containing $R_0$ and $R_1$ as components gives rise to a map \[I_*(W|R_W)_{\nu}:I_*(Y_0|R_0)_{\alpha_0}\to I_*(Y_1|R_1)_{\alpha_1}.\] This map depends only on the  homology class $[\nu]\subset H_2(W,\partial W;\mathbb{Z})$ and the isomorphism class of $(W,\nu)$, where two pairs  are isomorphic if they are diffeomorphic by a map which intertwines the (generally implicit)  boundary identifications. 

Finally, we describe  the relationship between   Donaldson invariants and Floer homology. Suppose   first that  $X$ is a  smooth 4-manifold with nonempty boundary $\partial X = Y$, and $\alpha$ is a $1$-cycle in $Y$ as above. Then each class $w \in H^2(X)$ whose restriction to $Y$ is Poincar\'e dual to $\alpha$ determines a relative Donaldson invariant
\[ \Psi_{w,X}: \bA(X) \to I_*(Y)_{\alpha}. \] 
Given a   surface $R\subset Y$ and a polynomial $p\in \mathbb{Q}[t]$, we may   view $p(R)$ as an element of $\mathbb{A}(X)$, in which case we have the  relation \begin{equation}\label{eqn:relfloer}\Psi_{w,X}(p(R)) = p(\mu(R))\cdot \Psi_{w,X}(1)\in I_*(Y)_\alpha.\end{equation}
Moreover, if $\nu$ is a $2$-cycle in $X$ with $\partial \nu = \alpha$, and $\nu$ is Poincar{\'e} dual to $w$, then we have a cobordism map \[I_*(X|R)_{\nu}:\C\to I_*(Y|R)_{\alpha},\] and $I_*(X|R)_{\nu}(1)$ is simply the projection of $\Psi_{w,X}(1)$ to the generalized eigenspace $I_*(Y|R)_{\alpha}.$

Now, suppose that $X = X_1 \cup_{Y} X_2$ is obtained by gluing together two  4-manifolds along a common boundary component,  $Y \subset \partial X_1$ and $-Y \subset \partial X_2$. Suppose further  that $w \in H^2(X)$ restricts to $Y$ as the Poincar{\'e} dual of $\alpha$.  Then there is a   pairing \[\langle\cdot,\cdot\rangle: I_*(Y)_\alpha \otimes I_*(-Y)_\alpha \to \C\] as well as a natural product
\[ \bA(X_1) \otimes \bA(X_2) \to \bA(X) \]
such that if we let $w_i = w|_{X_i}$ and take $\lambda_i \in \bA(X_i)$, then
\begin{equation}\label{eqn:pairingrelinvts} D^w_X(\lambda_1 \lambda_2) = \langle \Psi_{w_1,X_1}(\lambda_1), \Psi_{w_2,X_2}(\lambda_2) \rangle, \end{equation}
as described in \cite[Section~6.4]{donaldson-book}. (The reader might  also refer to \cite{fukaya,bd-gluing}, though we will not need a pairing theorem for relative classes in $H_*(X_i,Y)$ as given there.)

\subsection{Open book decompositions}
\label{ssec:obd}
The discussion below is adapted from \cite{bs-shi}, though  simplified somewhat by only considering open book decompositions of 3-manifolds which are the complement of a Darboux ball in a closed contact 3-manifold.

\begin{definition} \label{def:open-book}
An \emph{open book} is a triple $(S,h,{\bf c})$, where
\begin{itemize}
\item $S$ is a compact, oriented surface with nonempty boundary, called the \emph{page};
\item $h: S \to S$ is a diffeomorphism which restricts to the identity on $\partial S$;
\item ${\bf c} = \{c_1,\dots,c_{b_1(S)}\}$ is a set of disjoint, properly embedded arcs such that $S \ssm {\bf c}$ deformation retracts onto a point; these are often known as \emph{basis} arcs.
\end{itemize}
\end{definition}

The product manifold $S \times [-1,1]$ admits an $[-1,1]$-invariant contact structure in which each $S \times \{t\}$ is a convex surface with collared Legendrian boundary and dividing set consisting of one boundary-parallel arc on each component of $\partial S$, oriented in the same direction as $\partial S$.  Upon rounding corners, we obtain a \emph{product sutured contact manifold} in the terminology of \cite{bs-shi}, denoted by $H(S)$, which is topologically a handlebody with boundary the double of $S$ and dividing set $\partial S \subset \partial H(S)$. For notational convenience, we will  often simply equate  $H(S)$ and $S\times[-1,1]$, as in the definition below.

\begin{definition}
\label{def:gammas}
Given an open book $(S,h,{\bf c})$, let $\gamma_1,\dots,\gamma_{b_1(S)}$  be the curves  given by
\[ \gamma_i = \big(c_i \times \{1\}\big) \cup \big(\partial c_i \times [-1,1]\big) \cup \big(h(c_i) \times \{-1\}\big)\subset \partial (S\times[-1,1])=\partial H(S). \] Each $\gamma_i$ intersects the dividing set on $\partial H(S)$  in two points.
We define $M(S,h,{\bf c})$ to be the contact manifold built from $H(S)$ by attaching contact 2-handles along these curves.
\end{definition}

\begin{definition} \label{def:ob-decomposition}
An \emph{open book decomposition} of the based contact 3-manifold $(Y,\xi,p)$ is an open book $(S,h,{\bf c})$ together with a contactomorphism
\[ f: M(S,h,{\bf c}) \to (Y(p),\xi|_{Y(p)}), \]
where $Y(p)$ is the complement of a Darboux ball around the point $p$.
\end{definition}

Suppose $(S,h,{\bf c})$ and $(S',h',{\bf c}')$ are two open books. A diffeomorphism of pairs \begin{equation}\label{eqn:gdiff}g:(S,{\bf c})\to (S',{\bf c}')\end{equation} which intertwines $h$ and $h'$  gives rise to a canonical isotopy class of contactomorphisms \begin{equation}\label{eqn:gcont}\bar g: M(S,h,{\bf c})\to M(S',h',{\bf c}').\end{equation}

\begin{definition} \label{def:ob-isomorphic}
We say that open book decompositions $(S,h,{\bf c},f)$ and $(S',h',{\bf c}',f')$ are \emph{isomorphic} if there exists a diffeomorphism $g$ as in \eqref{eqn:gdiff} such that $f=f'\circ\bar g$.
\end{definition}

The \emph{existence} part of the Giroux correspondence \cite{giroux-correspondence} asserts the following.

\begin{theorem}[\cite{giroux-correspondence}]
Every closed contact 3-manifold has an open book decomposition.
\end{theorem}



\begin{remark} The harder \emph{uniqueness} part of the Giroux correspondence  asserts that any two open book decompositions of a  contact 3-manifold are related, up to isomorphism, by a sequence of \emph{positive stabilizations}; see \cite[Definition 2.24]{bs-shi}. Our invariance proof for the contact element introduced in \cite{bs-shi} and described in Subsection \ref{ssec:def-invt} makes use of this. However, we will not need this invariance  for our results relating Stein fillings to $SU(2)$ representations. See Remark \ref{rmk:noinvceta}.
\end{remark}

\subsection{Stein manifolds and Lefschetz fibrations}
\label{ssec:steinlef}

Recall that if $(W,J)$ is a compact Stein 4-manifold and $Y = \partial W$ is a regular level set of a proper, strictly plurisubharmonic function on $W$, then \[\xi = TY \cap J(TY)\] is a contact structure on $Y$. Moreover, $(W,J)$ is called a \emph{Stein domain} and is said to be a \emph{Stein filling} of $(Y,\xi)$.
A smooth, surjective map $\pi: W \to D^2$ is called a \emph{positive allowable Lefschetz fibration} (PALF) if:
\begin{itemize}
\item its critical points are isolated and in distinct fibers over the interior of $D^2$;
\item it has the form $\pi(z_1,z_2) = z_1^2+z_2^2$ on a neighborhood $D^4 \subset \C^2$ of each critical point;  
\item the vanishing cycles in each singular fiber are nonseparating.
\end{itemize}  

Loi-Piergallini \cite{loi-piergallini} and Akbulut-\"Ozba\u{g}c{\i} \cite{akbulut-ozbagci} proved that every Stein domain admits a PALF, and Plamenevskaya \cite[Appendix~A]{plamenevskaya} showed that said PALF is compatible with the induced contact structure on the boundary. 

More precisely, suppose $(W,J)$ is a Stein filling of $(Y,\xi)$ and let $p$ be a point of $Y$. Then $(W,J)$ admits a PALF $\pi:W\to D^2$ whose regular fibers $\Sigma_x = \pi^{-1}(x)$ are surfaces with boundary. Define $S = \Sigma_{\pi(p)}$ and fix any collection \[{\bf c} = \{c_1,c_2,\dots,c_{b_1(S)}\}\] of basis arcs on $S$ which avoids  $p$. Let $h$ be the diffeomorphism of $S$ given by a product $\tau_{v_1}\tau_{v_2} \dots \tau_{v_m}$ of positive Dehn twists along  vanishing cycles $v_1,\dots,v_m$ of $\pi$ in the usual manner.  The natural identification of the mapping torus of $h$ with the union of fibers $\bigcup_{\theta\in\partial D^2} \Sigma_{\theta}$ gives rise to a contactomorphism
\[ f: M(S,h,{\bf c}) \to (Y(p),\xi|_{Y(p)}), \]
which is canonical up to isotopy. The corresponding open book decomposition $(S,h,{\bf c}, f)$
of $(Y,\xi,p)$ is uniquely determined, up to isomorphism, by $\pi$ together with the basis of arcs ${\bf c}$. 


\begin{remark} As noted in \cite[Appendix~A]{plamenevskaya}, positively stabilizing the above open book decomposition  corresponds to taking a boundary connected sum of $W$ with the standard Stein 4-ball.  In particular, given Stein structures $J_1,\dots,J_k$ on $W$, we can find corresponding PALFs $\pi_1,\dots,\pi_k: W \to D^2$ so that the fibers of the various $\pi_i$ (and, hence, the pages of the induced open book decompositions for  the   contact structures $\xi_i = TY \cap J_i(TY)$ on $Y$) are all diffeomorphic to a common surface.  We can, moreover, take this common surface to have arbitrarily large genus and exactly one boundary component.
\end{remark}

\subsection{An instanton contact invariant} \label{ssec:def-invt} We recall below the construction of our contact invariant in sutured instanton homology \cite{bs-shi}. Our review is tailored to  the needs of this paper.

Suppose  $(Y,\xi,p)$ is a based contact manifold with  open book decomposition \[(S,h,{\bf c}=\{c_1,\dots, c_{b_1(S)}\}, f).\]  Let $\gamma_1,\dots,\gamma_{b_1(S)}$ be the curves in the boundary of $S\times [-1,1]$ given by 
\[ \gamma_i = \big(c_i \times \{1\}\big) \cup \big(\partial c_i \times [-1,1]\big) \cup \big(h(c_i) \times \{-1\}\big). \] Recall  from Subsection \ref{ssec:obd} that $M(S,h,{\bf c})$ is obtained by attaching  2-handles to   $S\times[-1,1]$ along these curves. Fix a compact, connected, oriented surface $T$ of genus at least $8$, with an identification $\partial T\cong -\partial S$, and let $R=S\cup T$ be the closed surface obtained by gluing $S$ to $T$. In a slight abuse of notation, we will also use $h$ to denote the extension of $h$ to $R$ by the identity on $T$. We  think of the mapping torus of $h:R\to R$ as given by \[R\times_h S^1:= R \times [-1,3] / \big((x,3) \sim (h(x),-1)\big).\] Likewise, in what follows, we will think of the product $R\times S^1$ as given by \[R\times S^1:= R \times [-1,3] / \big((x,3) \sim (x,-1)\big).\] In order to eventually define instanton Floer homology, we  also fix an embedded, nonseparating curve \[\eta\subset \inr(T)\times\{2\}\subset R\times\{2\}\subset R\times_h S^1 \]  and a closed curve  \[\alpha = \{q\}\times[-1,3]\subset R\times_h S^1,\] where $q$ is a point of $ \inr(T)\ssm \eta\subset R$.   See Figure \ref{fig:hs-closure} for a schematic of  $R\times_h S^1$, $\eta$,  and $\alpha$.
 Since $S\times[-1,1]$ is naturally a codimension 0 submanifold of $R\times_h S^1$, we can and will view  the $\gamma_i$  as curves in $R\times_h S^1$, as shown in the figure. As we shall see, the  topological result  below enables the definition of our contact invariant.
\begin{figure}
\labellist
\small
\pinlabel $S$ at 120 340
\pinlabel $T$ at 300 340
\pinlabel $h$ at -10 176
\pinlabel $\alpha$ at 270 207
\pinlabel $\eta$ at 353 193
\pinlabel $3$ [l] at 398 284
\pinlabel $1$ [l] at 398 180
\pinlabel ${-}1$ [l] at 387 73
\tiny
\pinlabel $\gamma_1$ at 76 176
\pinlabel $\gamma_2$ at 139 165
\endlabellist
\centering
\includegraphics[width=6.5cm]{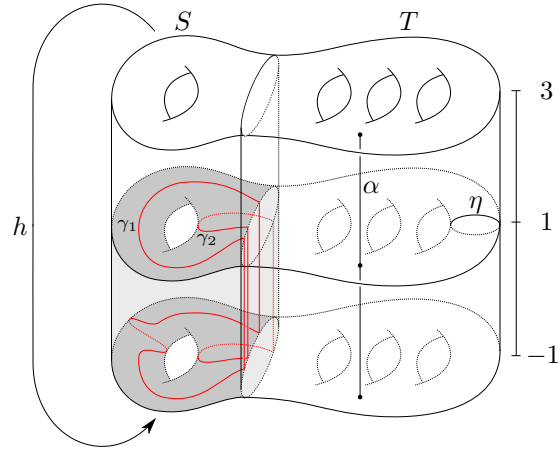}
\caption{A schematic of $R\times_h S^1$, $\eta$,  $\alpha$, and the $\gamma_i$ where $h$ is a positive Dehn twist and $S\times[-1,1]$ is shaded. For ease of drawing, we have shown the curve $\eta$ as lying on the fiber $R\times\{1\}$ rather than $R\times\{2\}$. Moreover, $g(T)=3$ in this illustration, whereas we actually require that $g(T)\geq 8$.}
\label{fig:hs-closure}
\end{figure}

\begin{proposition}
\label{prop:surgery-sum}
The 3-manifold obtained via   surgery on the $\partial (S\times[-1,1])$-framed link \[\mathbb{L}=\gamma_1\cup\dots\cup\gamma_{b_1(S)}\subset R\times_h S^1\] is canonically diffeomorphic (up to isotopy) to the union \[M(S,h,{\bf c})\cup_\partial \big((R\times S^1)\ssm B^3\big),\] where $B^3$ is a small ball around a point in $S\times\{0\}\subset R\times S^1$.
\end{proposition}

\begin{proof}
It is an easy exercise to check that surgering along the framed link $\mathbb{L}$ is topologically equivalent to first cutting $R\times_h S^1$ open along $\partial (S\times[-1,1])$ to obtain the two pieces \[P=S\times[-1,1]\,\,\,\,\textrm{ and }\,\,\,\,Q=(R\times_h S^1)\ssm \inr(S\times[-1,1]),\] then gluing $2$-handles to each  piece along the components of $\mathbb{L}$, and then gluing the resulting $3$-manifolds back together along their $2$-sphere boundaries. Note that the result of attaching these $2$-handles to $P$ is precisely the $3$-manifold $M(S,h,{\bf c})$. It thus suffices to show that attaching these $2$-handles to $Q$ yields something which is  homeomorphic to $(R\times S^1)\ssm B^3$. We can think of these $2$-handles as  \[H_i=D^2_i\times I := (c_i\times[-1,1])\times I,\]  for $i=1,\dots, b_1(S)$, where $H_i$  is attached  to $Q$ along $\gamma_i$ according to the map \[\partial D_i^2\to \partial Q=-\partial (S\times[-1,1])\] which is the identity on $c_i\times\{1\}$ and on $\partial c_i\times[-1,1]$, and $h$ on $c_i\times\{-1\}$. We will write \begin{equation}\label{eqn:Q}Q\cup_h H_1\cup \dots \cup H_{b_1(S)}\end{equation} for the result of this attachment. On the other hand, note that if we attach these $H_i$ to \[Q'= (R\times S^1)\ssm \inr(S\times[-1,1])\] according to the map \[\partial D_i^2\to \partial Q'=-\partial (S\times[-1,1])\] which is the identity on all of $\partial D_i^2=\partial (c_i\times[-1,1])$, then we recover precisely the complement $(R\times S^1)\ssm B^3$. Let us write \begin{equation}\label{eqn:Q'}Q'\cup_{\id} H_1\cup \dots \cup H_{b_1(S)}\end{equation} for the result of this attachment. The canonical homeomorphism from \eqref{eqn:Q} to \eqref{eqn:Q'} is then given by the map which is the identity on each $H_i$ and sends $(x,t)$ to $(x,t)$ for each \[(x,t)\in Q=(R\times_h S^1)\ssm \inr(S\times[-1,1]). \qedhere \]
\end{proof}

Let $\bar{Y}$ denote the connected sum \[\bar{Y} = Y\#(R\times S^1),\] formed as the union of $Y(p)$ with the complement of a small ball around a point in $S\times\{0\}\subset R\times S^1$. Note that the diffeomorphism $f$ extends in a canonical way to a diffeomorphism \begin{equation}\label{eqn:identMY}M(S,h,{\bf c})\cup_\partial \big((R\times S^1)\ssm B^3\big)\xrightarrow{\sim}\bar{Y}, \end{equation} which identifies the fiber $R\times\{2\}$ and the curves $\eta,\alpha$ with the corresponding objects in $\bar{Y}$. Let $V$ denote the cobordism obtained from $(R\times_h S^1)\times[0,1]$ by attaching $2$-handles to $(R\times_h S^1)\times\{1\}$ along the framed link $\mathbb{L}$  in Proposition \ref{prop:surgery-sum}. The identification in \eqref{eqn:identMY} enables us to view $V$ as a cobordism \[V:R\times_h S^1\to \bar{Y}.\] This cobordism then induces a map on instanton Floer homology, \begin{equation}\label{eqn:cobYHY} I_*(-V|{-}R)_{\nu}: I_*(-R\times_h S^1|{-}R)_{\alpha\sqcup\eta} \to I_*(-\bar{Y}|{-}R)_{\alpha\sqcup\eta},\end{equation} where $\nu\subset V$ is the properly embedded cylinder given by \[\nu=(\alpha\sqcup \eta)\times[0,1],\] and the ``$|{-}R$" in the notation for these Floer groups and map is shorthand for ``$|{-}(R\times \{2\})$". 
The codomain of this cobordism map is isomorphic to the sutured instanton homology of the manifold $-Y(p)$ (with a single equatorial suture on its boundary), as defined by Kronheimer-Mrowka in \cite{km-excision}. We will refer to the data $\data=(\bar{Y},R,\eta,\alpha)$ as a \emph{marked odd closure} of $Y(p)$, and define the \emph{sutured instanton homology of $-\data$} to be this $\C$-module, \[\SHIt(-\data):=I_*(-\bar{Y}|{-}R)_{\alpha\sqcup\eta}.\] While this module depends on the choice of closure, we proved in \cite{bs-naturality} that the modules associated to different closures are related by \emph{canonical} isomorphisms which are well-defined up to multiplication by units in $\C$.\footnote{Kronheimer-Mrowka showed that different choices give isomorphic modules.} These modules therefore form what we call a \emph{projectively transitive system of $\C$-modules} in \cite{bs-naturality}, and denote by $\SHItfun(-Y(p))$.

It is shown in  \cite{km-excision} that \[I_*(-R\times_h S^1|{-}R)_{\alpha\sqcup\eta}\cong \C.\] Let $1$ refer to any nonzero generator of this module, and define \[\iinvt(\xi,\data):= I_*(-V|{-}R)_{\nu}(1)\in I_*(-\bar{Y}|{-}R)_{\alpha\sqcup\eta}=\SHIt(-\data).\] This class is well-defined up to multiplication by units in $\C$. Moreover, we show in \cite{bs-shi} that for two different choices of marked odd closure $\data,\data'$ the elements $\Theta(\xi,\data)$ and $\Theta(\xi,\data')$ are related by the canonical isomorphism relating the modules associated to these closures.  In the language of \cite{bs-shi}, the collection $\{\Theta(\xi,\data)\}_\data$ thus defines an element  \[\Theta(\xi)\in\SHItfun(-Y(p)).\] As the notation suggests, this is an invariant of the contact manifold $(Y,\xi,p)$. We summarize that and other important properties in the theorem below, from \cite{bs-shi}.

\begin{theorem}\label{thm:theta-is-invariant}
The class $\iinvt(\xi)$ is invariant under positive stabilization of the open book decomposition $(S,h,{\bf c},f)$, and is therefore an invariant of the based contact manifold $(Y,\xi,p)$.
The invariant $\iinvt(\xi)$ is zero if $\xi$ is overtwisted and nonzero if $\xi$ is Stein fillable.
\end{theorem}

In particular, the second statement is equivalent to the assertion that $\Theta(\xi,\data)$ is zero if $\xi$ is overtwisted and nonzero if $\xi$ is Stein fillable, for any marked odd closure $\data$ of $Y(p)$. We will think exclusively in terms of closures for the remainder of this article.

\begin{remark}
\label{rmk:noinvceta}
For the applications to $SU(2)$ representations in this paper, we will not actually need invariance of the class $\iinvt(\xi)$ under positive stabilization. In a related vein, the curve $\eta$ in the definition of a marked odd closure is used to define the canonical isomorphisms relating various closures and the corresponding elements $\Theta(\xi,\data)$. It turns out that we will not need these  isomorphisms either for our results about $SU(2)$ representations, though we do need them for precise statements about the linear independence of our contact elements. So, if one is only interested in our results about $SU(2)$ representations, she may ignore the curve $\eta$.
\end{remark}

%
%
%

\section{Stein fillings and the contact class} \label{sec:stein-fillings}

Our goal in this section is to prove Theorem~\ref{thm:main-linear-independence}, restated here in the language of closures.
{
\renewcommand{\thetheorem}{\ref{thm:main-linear-independence}}
\begin{theorem} Suppose $W$ is a compact 4-manifold with boundary $Y$,  that $W$ admits Stein structures $J_1,\dots,J_n$  with induced contact structures $\xi_1,\dots,\xi_n$ on $Y$, and  that  the Chern classes $c_1(J_1),\dots,c_1(J_n)$ are distinct as elements of $H^2(W;\R)$. Then there is a  marked odd closure $\data = (\bar{Y}, R, \eta,\alpha)$ of $Y(p)$ such that the elements \[\iinvt(\xi_1,\data), \dots, \iinvt(\xi_n,\data)\in \SHIt(-\data) = I_*(-\bar{Y}|{-}R)_{\alpha\sqcup\eta}\] are linearly independent, for any $p\in Y$.

\end{theorem}
\addtocounter{theorem}{-1}
}

Our proof occupies the next three subsections, so we sketch it here. 

First, given a Lefschetz fibration on a Stein filling $(W,J)$ of $(Y,\xi)$ and a point $p\in Y$, we construct in Subsection~\ref{ssec:lf-closures} a $4$-manifold $X^\circ$ admitting a Lefschetz fibration over the annulus with closed fibers $R$, and a marked odd closure $\data =(\bar{Y}, R, \eta,\alpha) $ of $Y(p)$ such that $\bar{Y}$ separates $X^\circ$ into two pieces, \[X^\circ = -V\cup_{-\bar{Y}} W^\dagger,\] where $W^\dagger$ depends only on the smooth topology of $W$. See Figure \ref{fig:lf-over-annulus} for a schematic of $X^\circ$. 
 
Then, we describe  in Subsection~\ref{ssec:lf-capping} how to cap off $X^\circ$ with $4$-manifolds $Z$ and $C$, to obtain a closed $4$-manifold \[\tilde{X} = X\cup_{-\bar{Y}}\tilde{C}:=(Z\cup -V)\cup_{-\bar{Y}} (W^\dagger\cup C),\] still separated by $\bar{Y}$, admitting a Lefschetz fibration over the sphere which naturally extends that on $X^\circ$, and where $\tilde{C}$ depends only on $W$. Moreover, we perform this capping off in a manner that enables us  to understand the Donaldson invariants (in particular, the basic classes) of $\tilde{X}$. See Figure \ref{fig:lf-over-sphere} for a schematic of $\tilde{X}$. 

Given $n$  different Stein structures $J_1,\dots,J_n$ on $W$, inducing contact structures $\xi_1,\dots,\xi_n$ on $Y$ as in Theorem \ref{thm:main-linear-independence}, we thus obtain closed $4$-manifolds $\tilde{X}_1,\dots,\tilde{X}_n$ of the form described above. Further, we can arrange that the same marked odd closure $\data$ of $Y(p)$ separates each of these $\tilde{X}_i$ into two pieces, \[\tilde{X}_i =  X_i\cup_{-\bar{Y}}\tilde{C},\] with $\tilde{C}$ as above, such that $X_i$ has relative invariant the contact class \[\Psi_{\tilde{w_i},X_i}(\pbot(-R))=\iinvt(\xi_i,\data)\in\SHIt(-\data)=I_*(-\bar{Y}|{-}R)_{\alpha\sqcup\eta}\subset I_*(-\bar{Y})_{\alpha\sqcup\eta},\] for some polynomial $\pbot(t) \in \Q[t]$.

Finally, supposing the Chern classes of these $J_i$ are distinct as in the hypotheses of Theorem \ref{thm:main-linear-independence}, we describe in Subsection~\ref{ssec:detect-chern} how to distinguish the elements $\iinvt(\xi_1,\data), \dots, \iinvt(\xi_n,\data)$  by computing certain Donaldson invariants of the various $\tilde{X}_i$.  Specifically, we first identify a class $\Sigma \in H_2(\tilde{C})$ on which the canonical classes of the Lefschetz fibrations $\tilde{X}_i \to S^2$ take different values, $k_1,\dots,k_n$. We then  define, for each $i$, a polynomial $g_i(t)\in\Q[t]$ such that pairing the class \[\iinvt(\xi_j,\data) \in I_*(-\bar{Y})_{\alpha\sqcup\eta}=I_*(-\partial\tilde{C})_{\alpha\sqcup\eta}\] with the  relative invariant \[\Psi_{\tilde{w},\tilde C}\left(\left(1+\frac{x}{2}\right)g_i(\Sigma)\right)\in I_*(\partial \tilde{C})_{\alpha\sqcup\eta}\] gives a certain Donaldson invariant of $\tilde{X}_j$ which vanishes unless $k_j=k_i$, or equivalently unless $j=i$. This guarantees that $\iinvt(\xi_i,\data)$ could not have been a linear combination of the other $\iinvt(\xi_j,\data)$, proving Theorem~\ref{thm:main-linear-independence}.


\subsection{Lefschetz fibrations and marked odd closures} \label{ssec:lf-closures}

Suppose $(W,J)$ is a Stein filling of $(Y,\xi)$ and fix a point $p\in Y$. As described in Subsection \ref{ssec:steinlef}, $W$ admits a corresponding Lefschetz fibration \begin{equation}\label{eqn:lf}\pi:W\to D^2.\end{equation} Let $(S,h,{\bf c},f)$ be an open book decomposition of $(Y,\xi,p)$ constructed from this  Lefschetz fibration, also described therein. In particular, $h$ is a product, \[h=\tau_{v_1}\tau_{v_2} \dots \tau_{v_m}\] of positive Dehn twists along  vanishing cycles for  the  fibration.  Fix a surface $T$ of genus at least $8$, with an identification $\partial T \cong -\partial S$, and let $R = S \cup T$. In a slight abuse of notation, we will also use $h$ to denote the extension of $h$ to $R$ by the identity on $T$. By construction, $p$ is also a point of $S$. Let \[\bar{Y} = Y\#(R\times S^1),\] formed as the union of $Y(p)$ with the complement of a small ball around \[p\times\{0\}\in S\times\{0\}\subset R\times S^1,\] and fix curves $\eta,\alpha$ as described in Subsection \ref{ssec:def-invt}, so that $\data = (\bar{Y},R,\eta,\alpha)$ is a marked odd closure of $Y(p)$.

As outlined at the beginning of this section, we will define a 4-manifold $X^\circ$, admitting a Lefschetz fibration over the annulus with generic fiber $R$, which is separated by $\bar{Y}$ into two 
 cobordisms, \[-V: -R\times_h S^1 \to -\bar{Y}\qquad\textrm{and}\qquad W^\dagger: -\bar{Y} \to -R\times S^1.\] The cobordism $V$ is simply  the cobordism  used  in Subsection \ref{ssec:def-invt} to define the contact invariant $\Theta(\xi)$. In particular, letting $\nu\subset V$ denote the cylindrical cobordism $(\eta\sqcup\alpha)\times[0,1]$, as in that subsection, we have that the induced cobordism map  sends a nonzero generator to the appropriate contact class, by definition. We record this in the form of a lemma for easy reference. 
 
 \begin{lemma}
 \label{lem:Vmap} The cobordism map \[I_*(-V|{-}R)_{\nu}: \mathbb{C}\cong I_*(-R\times_h S^1|{-}R)_{\alpha\sqcup\eta} \to I_*(-\bar{Y}|{-}R)_{\alpha\sqcup\eta}=\SHIt(-\data)\] satisfies
\[I_*(-V|{-}R)_{\nu}(1) = \Theta(\xi,\data),\]  where $1$ represents any nonzero element of the domain.\qed
 \end{lemma}

We now describe the cobordism $W^\dagger$. The Lefschetz fibration $\pi$ in \eqref{eqn:lf} specifies a handle decomposition of $W$ in which we first attach $b_1(S)$ $1$-handles to the 4-ball to form $S\times D^2$, and then attach  $-1$-framed $2$-handles along copies of the vanishing cycles $v_i$ in  fibers $S \times \theta_i$, $\theta_i \in \partial D^2$, to obtain $W$. 
 
If we turn the above handle decomposition of $W$ upside-down and omit the 4-ball, we obtain a cobordism \[W\ssm B^4:-Y\to-S^3.\] Furthermore, we may  describe this cobordism as formed from $-Y\times[0,1]$ by first attaching $+1$-framed $2$-handles along corresponding copies of the vanishing cycles in $\{1\}\times -Y$, producing  a cobordism from $-Y$ to $-\#^{b_1(S)}(S^1\times S^2)$, and then attaching $3$-handles to the latter boundary along $2$-spheres to cancel the $S^1\times S^2$ summands.
 
Let $N(p)$ be a small ball in $-Y$ around $p$ and let $B^3$ denote a small ball around \[p\times\{0\}\subset -(R\times S^1).\] Let $W^\circ$ denote the cobordism with corners, \[W^\circ:-Y\ssm N(p)\to-S^3\ssm N(p),\] obtained by removing from the above $W\ssm B^4$ the product $N(p)\times[0,1].$  We then define the cobordism $W^\dagger$ as the union
\[ W^\dagger = W^\circ \bigcup_{\partial N(p)\times[0,1]= \partial B^3\times[0,1]} \big(-(R\times S^1) \ssm B^3\big)\times[0,1]. \]
In a slight abuse of notation, we will also use $\nu$ to denote the cylinder $\nu\subset W^\dagger$ given by $(\eta\sqcup\alpha)\times[0,1]$ as well as the cylinder \[\nu\subset X^\circ := -V\cup_{-\bar{Y}}W^\dagger\] given by the union of the cylinders $\nu\subset -V$ and $\nu\subset W^\dagger$. See Figure \ref{fig:lf-over-annulus} for a schematic diagram of $(X^\circ,\nu)$.

\begin{figure}
\labellist
\small

\pinlabel $-V$ [B] at 27 93
\pinlabel $W^\dagger$ [B] at 78 93

\tiny
\pinlabel $-R{\times_h}S^1$ [B] at 3 141
\pinlabel $-\bar{Y}$ [B] at 51 140
\pinlabel $-R{\times}S^1$ [B] at 99 141
\pinlabel $\nu$ [B] at 12 67
\pinlabel $-Y$ [r] at 54 124
\pinlabel $-R{\times}S^1$ [r] at 55 78
\pinlabel $-R{\times}S^1$ [r] at 103 78
\pinlabel $-S^3$ [r] at 103 124
\pinlabel $S^1{\times}I$ [B] at 52 12
\endlabellist
\centering
\includegraphics[width=4.5cm]{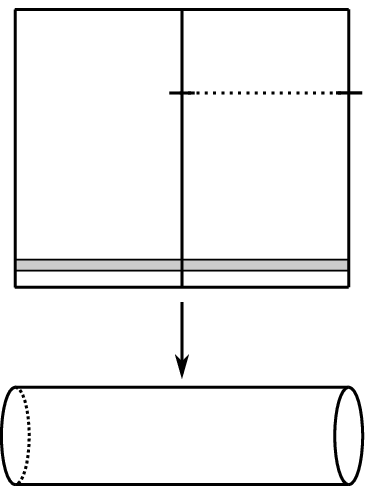}
\caption{The composition  $X^\circ = -V \cup_{-\bar{Y}} W^\dagger$ and its Lefschetz fibration over $S^1\times[0,1]$. The region of $W^\dagger$ above the dotted line consists of the handles coming from the given Lefschetz fibration on the Stein filling $W$; the region below is just a product. The cylindrical cobordism $\nu$ is represented in gray. }
\label{fig:lf-over-annulus}
\end{figure}
   
While $V$ certainly depends on the chosen Lefschetz fibration on $W$, we show that $W^\dagger$ does not. More precisely, we have the following.

\begin{lemma} \label{lem:w-dagger-smooth}
The pair $(W^\dagger, \nu )$ depends only on the smooth topology of $W$ and on the tuple $(S,R,\eta,\alpha)$, and there is a natural map \[i_\dagger: H_2(W) \to H_2(W^\dagger)\] which also depends only on this data.
\end{lemma} 

\begin{proof}
Suppose we have fixed a tuple $(S,R,\eta,\alpha)$. The lemma is then saying that $(W^\dagger,\nu)$ is independent of the chosen Lefschetz fibration on $W$. As explained above, a Lefschetz fibration on $W$ provides a description of $W\ssm B^4$ as obtained from $-Y\times[0,1]$ by attaching $2$- and $3$-handles. But the handle attachments associated to different Lefschetz fibrations on $W$ can be related by a sequence of handle moves avoiding $p$. Such a sequence then gives rise to a diffeomorphism between the two corresponding versions of $(W^\dagger,\nu)$ which restricts the identity on the boundary $\bar{Y}\sqcup-(R\times S^1)$.

To describe the map \[i_\dagger: H_2(W) \to H_2(W^\dagger),\]   note that $W^\circ$ can be identified as the complement in $W$ of \[D = B^4 \cup \big(N(p) \times [0,1]\big).\]  But  $D$ retracts into $\partial W$, so any class in $H_2(W)$ can be realized by an embedded surface which avoids it and hence lies in $W^\circ \subset W^\dagger$; this defines the desired map.  To check that $i_\dagger$ is well-defined, we observe that any two surfaces in $W^\circ$ which are homologous in $W$ must cobound a 3-chain in $W$, and this 3-chain can likewise be made to avoid $D$, so they remain homologous in $W^\dagger$ as well.
\end{proof}

We next show how the Lefschetz fibration on $W$ defines a Lefschetz fibration on $X^\circ$ over the annulus, per the following lemma.  Although $X^\circ$ is constructed as a cobordism from $-R\times_h S^1$ to $-R\times S^1$, we will find it convenient to view $X^\circ$ upside-down as a cobordism to $R\times S^1$ to $R\times_h S^1$.

\begin{lemma} \label{lem:lf-cobordism}
The composite cobordism 
\[ X^\circ = -V \cup_{-\bar{Y}} W^\dagger: R \times S^1 \to R \times_h S^1 \]
admits a relatively minimal Lefschetz fibration over the annulus  with generic fiber $R$.
\end{lemma}

\begin{proof}
From the definitions of $-V$ and $W^\dagger$, we have that $X^\circ$ is formed by thickening $-R\times_h S^1$; then attaching $\partial (S\times[-1,1])$-framed $2$-handles along the curves \[\gamma_i\subset \partial (S\times[-1,1])\subset -R\times_h S^1;\] then attaching $+1$-framed $2$-handles (with respect to the $S$-framing) along copies   \[v_j\times \{\theta_j\}\subset S\times\{\theta_j\}\] of the vanishing cycles, where \[-1<\theta_1<\dots<\theta_m<1;\] and then attaching $3$-handles along  the $S^2$'s in the resulting \[-\big(\#^{b_1(S)}(S^1\times S^2)\big)\#(R\times S^1)\] boundary. The left side  of Figure~\ref{fig:closure-slides} shows a Kirby diagram in $-R \times_h S^1$ for the $2$-handle attachments along the $\gamma_i$ and the copies of the $v_j$.

Observe that if we slide the bottom arc $h(c_i) \times \{-1\}$ of each $\gamma_i$ over the $2$-handles attached along the $v_j\times\{\theta_j\}$, the effect on  $\gamma_i$ is simply to replace this arc  with a copy of $h^{-1}(h(c_i))=c_i$, as shown in the middle  of Figure~\ref{fig:closure-slides}.  This is because performing $+1$-surgeries on these copies of the $v_j$ is topologically the same as cutting along an $S\times\{\theta\}$ fiber and regluing by the corresponding composition of left-handed Dehn twists around the $v_j$, and this composition is precisely $h^{-1}$. 

\begin{figure}
\labellist

\pinlabel $T$ at 300 455
\pinlabel $S$ at 125 455
\pinlabel $h$ at 0 370
\tiny

\pinlabel ${-}1$ at 403 76
\pinlabel $0$ at 410 176
\pinlabel $1$ at 410 283
\pinlabel $3$ at 410 387
\pinlabel slide at 432 245
\pinlabel 3-hdles at 891 245
\pinlabel $c_i$ at 30 276
\pinlabel $h(c_i)$ at -29 67
\pinlabel $v_j$ at 29 173
\pinlabel $c_i$ at 505 276
\pinlabel $c_i$ at 505 170
\pinlabel $v_j$ at 504 68
\pinlabel $v_j$ at 960 68
\pinlabel $\alpha$ at 269 311
\pinlabel $\eta$ at 347 301
\endlabellist
\centering
\includegraphics[width=15.2cm]{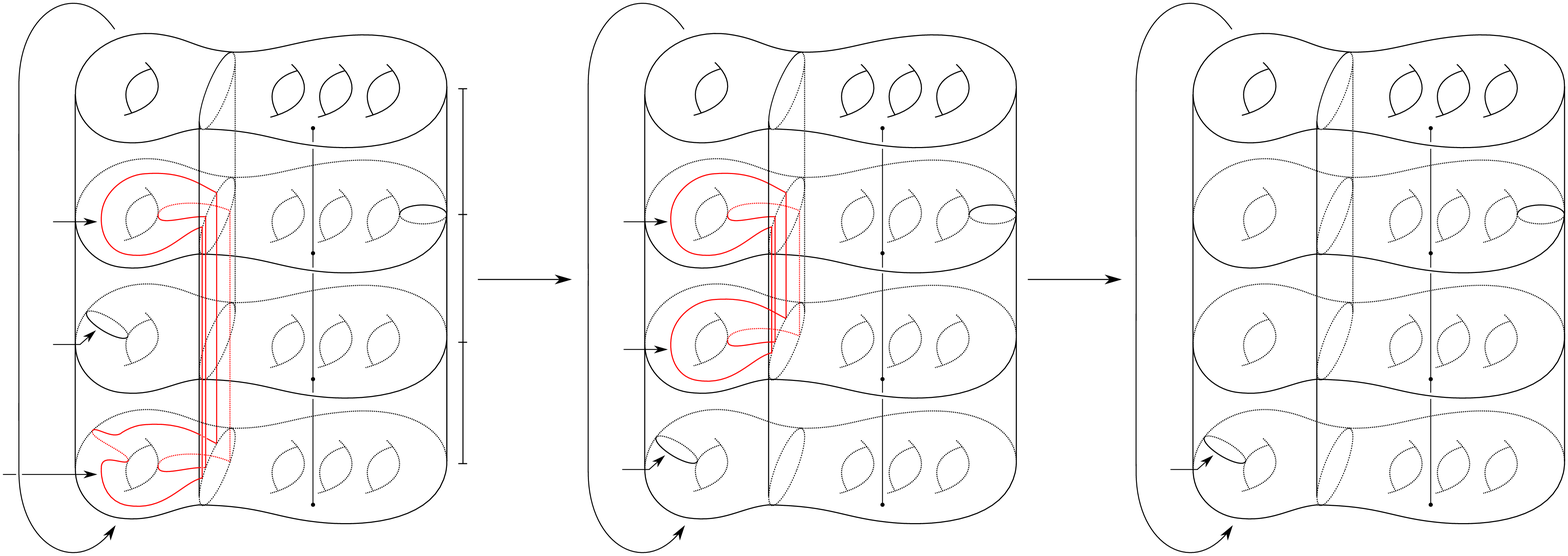}
\caption{Left, attaching $2$-handles along the $\gamma_i$ and copies of $v_j$ in $-R\times_hS^1$; we have labeled the top and bottom arcs of $\gamma_i$ by $c_i$ and $h(c_i)$. Middle, the result of sliding the $\gamma_i$ over the $v_j$ handles. Right, after attaching $3$-handles to cancel the $\gamma_i'$ handles, we are left only with $+1$-framed attachments along the copies of $v_j$.}
\label{fig:closure-slides}
\end{figure}

Let us denote the  curves obtained from $\gamma_i$ via these slides by $\gamma_i'$. After an isotopy of the $\gamma_i'$ and the $v_j$, we may view $\gamma_i'$ as given by \[\gamma_i' = \partial (c_i\times[0,1]),\] and the copies of $v_j$ as \[v_j\times \{\theta_j'\}\subset S\times\{\theta_j'\},\] where \[-1\leq \theta_1'<\dots<\theta_m'<0,\] as shown in the middle of Figure~\ref{fig:closure-slides}. 

The cobordism $X^\circ$ is thus formed by first attaching $+1$-framed $2$-handles along the $v_j\times \{\theta_j'\}$; then attaching $\partial (S\times[0,1])$-framed $2$-handles along the $\gamma_i'$, which has the effect of connect-summing the outgoing boundary with $b_1(S)$ copies of $S^1\times S^2$, as these $\gamma_i'$ bound disjoint disks in $S\times[0,1]$; and then attaching $3$-handles along the spheres $\{\ast\} \times S^2$ in these $S^1 \times S^2$ summands, canceling the $2$-handles  attached along the $\gamma_i'$. That is, $X^\circ$ is  obtained by simply thickening $-R\times_h S^1$ and then attaching $+1$-framed $2$-handles along the $v_j\times \{\theta_j'\}$, as depicted on the right of Figure~\ref{fig:closure-slides}. Therefore, as a cobordism  \[X^\circ:R \times S^1\to R\times_h S^1,\] $X^\circ$  is obtained from  $(R\times S^1)\times [0,1]$ by   attaching $-1$-framed $2$-handles along the $v_j\times \{\theta_j'\}\times\{1\}.$ This shows that $X^\circ$  admits a relatively minimal Lefschetz fibration over the annulus $S^1\times[0,1]$, with generic fiber $R$ and  vanishing cycles given by the $v_j$, as claimed.  
\end{proof}

The Lefschetz fibrations on $W$ and $X^\circ$ equip these $4$-manifolds with canonical classes, $K_W$ and $K_{X^\circ}$. The remainder of this subsection is devoted to understanding their relationship, per Lemma \ref{lem:canonical-class} below. In this lemma,  \[i: H_2(W) \to H_2(X^\circ)\] refers to the composition of the map \[i_\dagger: H_2(W) \to H_2(W^\dagger)\] of Lemma~\ref{lem:w-dagger-smooth}  with the map  $H_2(W^\dagger)\to H_2(X^\circ)$ induced by inclusion.

\begin{lemma} \label{lem:canonical-class}
The canonical classes of $W\to D^2$ and $X^\circ \to S^1 \times [0,1]$ satisfy
\[ K_W \cdot \Sigma = K_{X^\circ}\cdot i(\Sigma) \]
for all $\Sigma \in H_2(W)$.
\end{lemma}

\begin{proof}

In order to understand $K_W$, recall once more that $W$ is built from $S \times D^2$ by attaching $-1$-framed 2-handles along the vanishing cycles $v_1,\dots,v_m$, which we will view as living in a single fiber $S$. Since $S\times D^2$ has the homotopy type of a 1-complex, we have
\[ H_2(W) = \ker(\Z^m \to H_1(S)), \]
where the map sends $(c_1,\dots,c_m)$ to $\sum c_i[v_i]$.  In other words, every class $\Sigma \in H_2(W)$ is represented by a 2-chain $\sigma$ in $S$ with boundary $\partial \sigma$ equal to some linear combination of the $v_i$, together with Lefschetz thimbles capping off each $v_i$ component of $\partial\sigma$ in the core of the corresponding 2-handle.  

The canonical class $K_W$ is, by definition, \[c_1(\textrm{det}_{\mathbb{C}}(T^*W)) = -c_1(W,J).\]  Since the fiber $S$ can be made symplectic, $K_W$ restricts to $S$ as $-e(TS)$, and so its evaluation on the chain $\sigma$ is determined entirely by the topology of $\sigma$. Similarly, its evaluation on the Lefschetz thimbles depends only on their number, since each is identified with the core of a $-1$-framed handle in a standard way.  Thus the evaluation $K_W \cdot \Sigma$ depends only on the chain $\sigma \subset S$.  (See \cite[Section~2]{osz-symplectic} for an explicit description of this evaluation in the case where $\sigma$ is a subsurface of $S$, noting that their $c_1(k)$ is our $-K_W$.) 


To complete the proof, we simply note that if $\Sigma \in H_2(W)$ is represented by the $2$-chain $\sigma\subset S$ together with some Lefschetz thimbles, then $i(\Sigma)\in H_2(X^{\circ})$ is represented by the same $\sigma\subset S\subset R$, viewed as a $2$-chain in a fiber $R$ of the Lefschetz fibration on $X^{\circ}$, together with the corresponding Lefschetz thimbles. The evaluation $K_{X^\circ} \cdot i(\Sigma)$ is therefore computed in exactly the same way as $K_W \cdot \Sigma$.
\end{proof}


\subsection{Capping off Lefschetz fibrations} \label{ssec:lf-capping}

Below, we describe how to cap off the cobordism $X^{\circ}$ defined in  the previous subsection with 4-manifolds $C$ and $Z$ to form a closed $4$-manifold $\tilde{X}$, as outlined at the beginning of this section.  We  construct these caps  carefully, per the lemma below,  in a manner that enables us   to understand the basic classes of $\tilde{X}$. 

\begin{lemma} \label{lem:lf-cap}
There exist smooth 4-manifolds $C$ and $Z$, with boundaries \[\partial C \cong R \times S^1\,\,\,\textrm{ and }\,\,\,\partial Z \cong -R \times_h S^1\] such that: 
\begin{itemize}
\item $H_1(C) = H_1(Z) = 0$;
\item $C$ and $Z$ admit relatively minimal Lefschetz fibrations over $D^2$ with generic fiber $R$, compatible with the given fibrations of their boundaries; and
\item $C$ and $Z$ each contain two closed, smoothly embedded surfaces of genus $2$ and self-intersection $2$ which are disjoint from each other and from a generic fiber; in particular, $b_2^+(C) > 1$ and $b_2^+(Z) > 1$.
\end{itemize}
Moreover, the construction of $C$ depends only on $R$.
\end{lemma}

\begin{proof}
We proceed along the lines of \cite[Section~4]{sivek-donaldson}.  To define  $C$, we first choose a relatively minimal Lefschetz fibration $C_0 \to S^2$ with generic fiber $R$ whose vanishing cycles contain a set of generators for $H_1(R)$. We then define $C$ to be the fiber sum \[C = (R \times D^2) \#_{R} \big(C_0 \#_{R} C_0\big) \#_{R} \big(C_0 \#_{R} C_0\big).\]  We have $H_1(C) = 0$ because the vanishing cycles generate the homology of the fiber.

To  verify the claim about   surfaces of positive intersection in $C$, it is enough to check that each $C_0 \#_{R} C_0$ summand contains a  surface of the form described in the lemma. We prove this exactly as in the proof of \cite[Lemma~4.6]{sivek-donaldson}.  Namely,  take a matching path between one critical value in each copy of $C_0$, each with the same vanishing cycle $\gamma \subset R$, and let $\mathbb{S}$ be a sphere of self-intersection $-2$ lying above this path.  Next,  take a circle in the base which separates the bases of the two $C_0$ summands and intersects the matching path once, and let $\mathbb{T}$ be a torus of self-intersection zero lying above this circle whose restriction to each fiber is a curve dual to $\gamma$.  Then $\mathbb{S} \cdot \mathbb{T} = 1$, and by surgering two parallel copies of $\mathbb{T}$ to $\mathbb{S}$ we obtain a genus $2$ surface of self-intersection \[([\mathbb{S}]+2[\mathbb{T}])^2 = 2\] in the summand $C_0 \#_{R} C_0$.

The construction of $Z$ proceeds identically, except for the first step: we first factor  \[h^{-1}:R\to R\] into a product of positive Dehn twists around nonseparating curves and construct a  relatively minimal Lefschetz fibration $Z_0 \to D^2$  with  corresponding vanishing cycles and generic fiber $R$, so that \[\partial Z_0 = R\times_{h^{-1}}S^1 \cong -R\times_hS^1.\]   We then define $Z$ to be the fiber sum \[Z=Z_0 \#_{R} \big(C_0 \#_{R} C_0\big) \#_{R} \big(C_0 \#_{R} C_0\big),\] with $C_0$ as above.
\end{proof}

We now define the closed $4$-manifold $\tilde{X}$ by gluing  these caps to $X^\circ$,  \[\tilde{X} =Z\,\cup_{-R\times_h S^1} \,-V\,\cup_{-\bar{Y}}\, W^\dagger\,\cup_{-R\times S^1}\, C,\] as indicated in Figure \ref{fig:lf-over-sphere}. Note that $\bar{Y}$ separates $\tilde{X}$ into the $4$-manifolds  \[X:=Z\cup -V\,\,\,\textrm{ and }\,\,\,\tilde{C}:= W^\dagger\cup C\] with boundaries \[\partial X = -\bar{Y}=-\partial \tilde{C}.\] Since $H_1(Z)=H_1(C)=0$, we may extend the cylinder $\nu \subset X^\circ$ to a closed closed surface \[\tilde{\nu} \subset \tilde{X}\] by capping off the cycles $\nu \cap \partial Z$ and $\nu \cap \partial C$ inside $Z$ and $C$. Let us fix some such $\tilde{\nu}$.

We record some of the important properties of $(\tilde{X},\tilde{\nu})$ in the following lemma.

\begin{figure}
\labellist
\small

\pinlabel $-V$ [B] at 88 95
\pinlabel $Z$ [B] at 35 95
\pinlabel $C$ [B] at 190 95
\pinlabel $W^\dagger$ [B] at 138 95

\tiny
\pinlabel $-R{\times_h}S^1$ [B] at 63 144
\pinlabel $-\bar{Y}$ [B] at 111 143
\pinlabel $-R{\times}S^1$ [B] at 159 143
\pinlabel $\tilde{\nu}$ [B] at 47 70
\pinlabel $-Y$ [r] at 114 128
\pinlabel $-R{\times}S^1$ [r] at 163 81
\pinlabel $-R{\times}S^1$ [r] at 115 81
\pinlabel $-S^3$ [r] at 163 128
\pinlabel $S^1{\times}I$ [B] at 112 12
\pinlabel $D^2$ [B] at 36 12
\pinlabel $D^2$ [B] at 190 12
\endlabellist

\centering
\includegraphics[width=9cm]{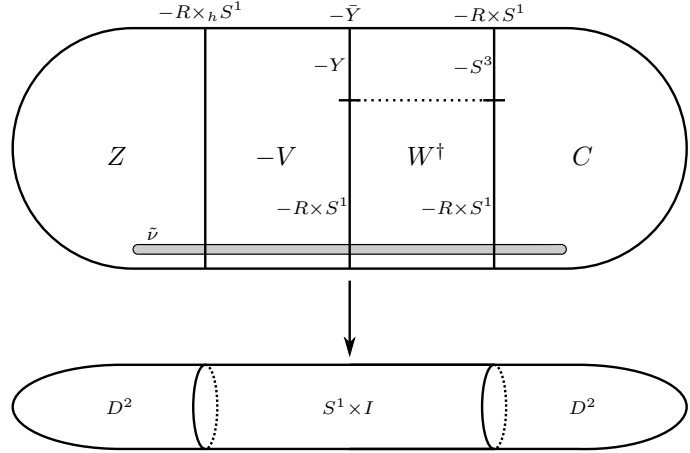}
\caption{The Lefschetz fibration $\tilde{X} = \big(-Z \cup-V \cup_{-\bar{Y}} W^\dagger \cup C\big) \to S^2$.}
\label{fig:lf-over-sphere}
\end{figure}

\begin{lemma} The closed $4$-manifold $\tilde{X}$  above satisfies the following:
\label{lem:tildex}
\begin{enumerate}
\item \label{i:tilde-x-lf} there is a relatively minimal Lefschetz fibration $\tilde{X} \to S^2$ with generic fiber $R$ extending the Lefschetz fibration $X^\circ \to S^1 \times [0,1]$;
\item \label{i:tilde-x-tight} the 3-manifold $-\bar{Y}$ separates $\tilde{X}$ into two pieces, $X \cup_{-\bar{Y}} \tilde{C}$, each of which contains two disjoint genus-2 surfaces of self-intersection 2 and thus has $b^+_2 > 1$;
\item \label{i:tilde-x-simple} $H_1(\tilde{X}) = 0$ and $\tilde{X}$ has simple type; 
\item \label{i:tilde-x-iinvt} the relative invariant of $X$ satisfies
\[ I_*(X|{-}R)_{\tilde{\nu}}(1) = \iinvt(\xi,\data). \]
\end{enumerate}
Moreover, the pair $(\tilde{C}, \tilde{\nu} \cap \tilde{C})$ is independent of the Lefschetz fibration on $W$.
\end{lemma}

\begin{remark}
\label{rmk:odd}
Note that since $\tilde X$ admits an almost complex structure, we have that \[b_2^+(\tilde X)-b_1(\tilde X)+1\] is even. The fact, from Lemma \ref{lem:tildex}, that $b_1(\tilde X)=0$ then implies that $b_2^+(\tilde X)$ is odd. We will make use of this later.
\end{remark}

\begin{proof}[Proof of Lemma \ref{lem:tildex}]
Property \eqref{i:tilde-x-lf} follows immediately from the construction of $\tilde{X}$. See Figure \ref{fig:lf-over-sphere} for a schematic of the Lefschetz fibration on $\tilde{X}$. Property \eqref{i:tilde-x-tight} follows from the fact that the same is true of $Z$ and $C$. For Property  \eqref{i:tilde-x-simple}, note that $H_1(\tilde{X}) = 0$ since  the vanishing cycles of the Lefschetz fibration on $\tilde{X}$ generate the homology of the fiber $R$, as this is true for $Z$ and $C$. Moreover, $\tilde{X}$ has simple type because, as recorded in Property \eqref{i:tilde-x-tight},  $\tilde{X}$ contains a tight surface (see the discussion in Subsection \ref{ssec:donaldson}). For Property \eqref{i:tilde-x-iinvt}, note that since $X = Z\circ -V$, we have  \begin{equation}\label{eqn:compZX} I_*(X|{-}R)_{\tilde{\nu}}(1) = I_*(-V|{-}R)_{\nu}\big(I_*(Z|{-}R)_{\tilde{\nu}}(1)\big). \end{equation} Since there is a relatively minimal Lefschetz fibration $Z \to D^2$ with $b_2^+(Z) \geq 1$ whose fiber $R$ has genus at least 2 and is dual to $\tilde{\nu}$, we know from \cite[Proposition~8.2]{sivek-donaldson} that the relative invariant 
\[ I_*(Z|{-}R)_{\tilde{\nu}}(1) \in I_*(-R \times_h S^1|{-}R)_{\alpha\sqcup\eta} \cong \C \]
is nonzero, so the element in \eqref{eqn:compZX} is simply $I_*(-V|{-}R)_{\nu}(1)$, which is equal to $\Theta(\xi,\data)$ by construction/definition, as recorded in Lemma \ref{lem:Vmap}. The last claim, that $(\tilde{C}, \tilde{\nu} \cap \tilde{C})$ is independent of the Lefschetz fibration on $W$, is self-evident.
\end{proof}

Since $\tilde{X}$ admits a relatively minimal Lefschetz fibration over the sphere with  fiber $R$ of genus at least 2, its canonical class $K_{\tilde{X}}$ is the only Seiberg-Witten basic class $K$  such that \[K\cdot R = 2g(R)-2,\] as shown by the second author in  \cite[Theorem~1.3]{sivek-donaldson}.\footnote{The proofs of the main results of \cite{sivek-donaldson} require some minor modification, as explained in the \hyperref[sec:appendix]{Appendix} to this paper, but the results themselves are still correct.}  We would like  the same to be true of the (Donaldson) basic classes of $\tilde{X}$. We    prove  below that by modifying $Z$ slightly, in a way that does not change the properties described in Lemma \ref{lem:lf-cap}, we can arrange for $\tilde{X}$ to satisfy Conjecture~\ref{conj:witten-conjecture}, in which case the two types of basic classes  coincide, as desired.
\begin{proposition}
\label{prop:modifyZ}
We can arrange that that $\tilde{X}$ satisfies Conjecture \ref{conj:witten-conjecture} by replacing $Z$ with a  suitable fiber sum $Z \#_R Z'$.
\end{proposition}

\begin{proof}
By construction, $\tilde{X}$ is symplectic with $H_1(\tilde{X})=0$ and $b_2^+(\tilde{X})>1$, and contains a genus 2 surface with self-intersection 2. It therefore  suffices, by Theorem \ref{thm:witten-conjecture} and Remark \ref{rmk:blowup-donaldson},  to show that after replacing $Z$ with a suitable fiber sum $Z \#_R Z'$ and then blowing up some number of times, the resulting manifold $\tilde{X}'$ has the same Betti numbers $b_2^{\pm}(\tilde{X}')$ as a smooth hypersurface in $\mathbb{CP}^3$ of even degree at least 6. In order to show this, we first argue the following.

\begin{claim}
\label{claim:fiber}
We can arrange that \[\frac{b_2^-(\tilde{X})}{b_2^+(\tilde{X})} < 1.9,\] and that $b_2^+(\tilde{X})$ is an arbitrary sufficiently large odd number, by replacing $Z$ with a suitable fiber sum $Z \#_{R} Z'$.
\end{claim}

\begin{proof}[Proof of Claim \ref{claim:fiber}]For this claim, we borrow from an argument in \cite{sivek-donaldson}, which is in turn adapted from \cite[Lemma~13]{km-p}. According to \cite[Lemma~4.4]{sivek-donaldson}, there is a pair of relatively minimal Lefschetz fibrations $V_1, V_2 \to S^2$ with generic fibers of genus $g = g(R)$ such that taking fiber sums of $\tilde{X}$ with either $V_i$ increases $b_2^\pm(\tilde{X})$ by $n^\pm(V_i)$, where
\begin{equation}\label{eqn:v1v2} (n^+(V_1), n^-(V_1), n^+(V_2)) = \begin{cases}
(g,g+4, g+2), & g\mathrm{\ even} \\
(g+1,g+9, g+3), & g\mathrm{\ odd}.
\end{cases} \end{equation}
Since $\gcd(n^+(V_1),n^+(V_2)) = 2$, every sufficiently large even integer can be written as \[m_1 n^+(V_1) + m_2 n^+(V_2)\] for some nonnegative integers $m_1$ and $m_2$, where $m_2 < n^+(V_1)$.  In particular, by replacing $Z$ with the fiber sum
\[ Z \#_{R} (m_1V_1) \#_{R} (m_2 V_2), \]
we can increase $b_2^+(\tilde{X})$ by any sufficiently large even integer  while preserving the properties of $Z$ and $\tilde{X}$ in  Lemmas~\ref{lem:lf-cap} and \ref{lem:tildex}.  Recall from Remark \ref{rmk:odd} that $b_2^+(\tilde{X})$ is always odd. It then follows  that $b_2^+(\tilde{X})$ can be made to equal any sufficiently large odd integer via the above fiber sum.
In order to simultaneously arrange that \[\frac{b_2^-(\tilde{X})}{b_2^+(\tilde{X})} < 1.9,\] we simply note that as we make  $b_2^+(\tilde{X})$  large in the manner above, the coefficient $m_1$ must tend to infinity while $m_2$ remains bounded. It follows that the above ratio approaches \[\frac{n^-(V_1)}{n^+(V_1)}.\]  Since $g \geq 8$, this  ratio is at most $1.8 < 1.9$, according to \eqref{eqn:v1v2}. This proves the claim.
\end{proof}

To finish the proof of the proposition, recall that the degree $d$ hypersurfaces $W_d \subset \mathbb{CP}^3$ satisfy
\[\frac{b_2^-(W_d)}{b_2^+(W_d)} \to 2 \qquad \mathrm{and} \qquad b_2^+(W_d)\to \infty\]
as $d \to \infty$. Moreover, $b_2^+(W_d)$ is odd whenever $d$ is even. Claim \ref{claim:fiber} then implies  that for a    a sufficiently large, even value of $d$, we can arrange that  $b^+_2(\tilde{X}) = b^+_2(W_d)$ and
\[ \frac{b^-_2(\tilde{X})}{b^+_2(\tilde{X})} < 1.9 < \frac{b^-_2(W_d)}{b^+_2(W_d)}, \] after replacing $Z$ with a fiber sum as above. Let us choose such a $d$.
Then for  \[n =b_2^{-}(W_d)-b_2^-(\tilde{X})> 0,\] the blown-up manifold \[\tilde{X}' = \tilde{X} \# n \overline{\mathbb{CP}^2}\] satisfies \[b^\pm_2(\tilde{X}') = b^\pm_2(W_d),\] as desired. As mentioned above, it then follows from Theorem \ref{thm:witten-conjecture} that $\tilde{X}' $ satisfies Conjecture \ref{conj:witten-conjecture}, and then, from Remark \ref{rmk:blowup-donaldson}, that $\tilde{X}$ does as well.
\end{proof}

We will henceforth assume that $Z$ has been modified as in Proposition \ref{prop:modifyZ}, so that $\tilde{X}$ satisfies Conjecture \ref{conj:witten-conjecture}. As mentioned previously, the lemma below follows immediately. 

\begin{lemma}
\label{lem:can-class}
The canonical class $K_{\tilde{X}}$ is the unique (Donaldson) basic class $K$ on $\tilde{X}$ satisfying $K\cdot R = 2g(R) - 2$. \qed
\end{lemma}





\subsection{Detecting Chern classes} \label{ssec:detect-chern}

Now suppose $(W,J_1),\dots,(W,J_n)$ are Stein domains with boundary $Y=\partial W$, inducing contact structures \[\xi_i = TY \cap J_iTY\] on $Y$. Suppose the Chern classes $c_1(J_1),\dots, c_1(J_n)$ are all distinct as classes in $H^2(W;\R)$, as in the hypothesis of Theorem \ref{thm:main-linear-independence}. Let \[\pi_i: W \to D^2\] be Lefschetz fibrations for these Stein domains with common generic fiber $S$. Let $(S,h_i,{\bf c}_i,f_i)$ be open book decompositions of the based contact manifolds $(Y,\xi_i,p)$ constructed from these Lefschetz fibrations as described in Subsection \ref{ssec:steinlef}. Let \[\data=(\bar{Y}=Y\#(R\times S^1),R,\eta,\alpha)\] be an odd marked closure of $Y(p)$ as defined in Subsection \ref{ssec:lf-closures}. Let \[\tilde{X_i}=X_i\cup_{-\bar{Y}} \tilde {C}\] be the closed $4$-manifold constructed from $(W,J_i)$ as in Subsections \ref{ssec:lf-closures} and \ref{ssec:lf-capping}, and let \[\tilde{\nu}_i\subset \tilde{X}_i\] denote the corresponding closed $2$-cycle. As described in Subsection \ref{ssec:lf-capping}, we may assume that the  pair \[(\tilde{C},\tilde{\nu}):=(\tilde{C},\tilde{C}\cap \tilde{\nu}_i)\] is independent of $i$. In what follows we will let \[\tilde{w}_i \in H^2(\tilde{X}_i)\] denote the Poincar\'e dual of $\tilde{\nu}_i$. We will denote its restrictions to $X_i$ and $\tilde {C}$ by $\tilde{w_i}$ and $\tilde{w}$. For convenience, we will also simply write $I_*(-\bar{Y})$ and $I_*(-\bar{Y}|{-}R)$ in this subsection -- that is, without the subscript $\alpha \sqcup \eta$ -- and let $g = g(R)$.

We will prove Theorem \ref{thm:main-linear-independence} through a careful analysis of the Donaldson invariants of the $4$-manifolds $\tilde{X_i}$ and their relationships with the contact elements $\Theta(\xi_i,\data)$, in the manner described at the beginning of this section. We start by showing that these contact elements can be interpreted as relative invariants of the $4$-manifolds $X_i$, per the following.

\begin{proposition} \label{prop:pbot}
There is a polynomial $\pbot(t) \in \Q[t]$, which does not depend on $i$, such that each relative invariant
\[ \psi_{\tilde{w_i},X_i}(\pbot(-R)) \in I_*(-\bar{Y}) \]
lies in the generalized $(2-2g)$-eigenspace $I_*(-\bar{Y}|{-}R)$ of $\mu(-R)$ and is equal to 
\[\Theta(\xi_i,\data)=I_*(X_i|{-R})_{\tilde\nu_i}(1) \in I_*(-\bar{Y}|{-}R), \]
and such that $\pbot(2-2g)=1$ and $\pbot(m) = 0$ for all $m = 3-2g,4-2g,\dots,2g-2$.
\end{proposition}

\begin{proof}
In this proposition, we are interpreting $\pbot(-R)$ as an element of $\bA(X_i)$, just as in Subsection \ref{ssec:floer}. Recall that $\mu(-R)$ acts on $I_*(-\bar{Y}|{-}R)$ with eigenvalues given by $\pm 2k, \pm 2ki$ for $k=0,1,\dots,g-1$. Recall from Property (\ref{i:tilde-x-iinvt}) of Lemma \ref{lem:tildex} that \[\Theta(\xi_i,\data)=I_*(X_i|{-R})_{\tilde\nu_i}(1).\] As discussed in  Subsection \ref{ssec:floer}, the element $I_*(X_i|{-R})_{\tilde\nu_i}(1)$ is equal to the projection of the relative invariant $\psi_{\tilde{w_i},X_i}(1)$ to the generalized $(2-2g)$-eigenspace of $\mu(-R)$. Moreover, recall from that subsection that \[\psi_{\tilde{w_i},X_i}(\pbot(-R)) = \pbot(\mu(-R))\cdot\psi_{\tilde{w_i},X_i}(1).\] So, to show that there is a polynomial $\pbot(t)$ such that  \[ \psi_{\tilde{w_i},X_i}(\pbot(-R)) =\Theta(\xi_i,\data),\] it suffices to prove that there exists some $\pbot(t)$ for which $\pbot(\mu(-R))$ is projection onto this generalized $(2-2g)$-eigenspace of $\mu(-R)$. Such a polynomial will not depend on $i$. The existence of such a polynomial, which also meets the other requirements of the proposition, follows immediately from the linear algebra lemma below.
\begin{lemma}
Suppose $A: \C^N \to \C^N $ is a linear map with distinct eigenvalues $\lambda_1,\dots,\lambda_s \in \Q[i]$ associated to generalized eigenspaces $E_1,\dots,E_s$.  Suppose $\lambda_1$ is real and $d_1,\dots,d_n \in \Q[i]$ are distinct from $\lambda_1$. Then there exists a polynomial $p(t) \in \Q[t]$ such that $p(\lambda_1) = 1$, $p(d_j) = 0$ for all $j$, and $p(A): \C^N \to \C^N$ is the projection onto $E_1$.
\end{lemma}

\begin{proof}
If each $\lambda_j$ has multiplicity $m_j$, then $(A-\lambda_j)^{m_j}|_{E_j} = 0$ for each $j$, so if we let
\[ f(t) = \prod_{j=1}^n (t-d_j) \cdot \prod_{j=2}^s (t-\lambda_j)^{m_j} \in (\Q[i])[t]\]
then $f(d_j)=0$ for all $j$ and $f(A)$ acts as zero on $E_2,\dots,E_s$.  The polynomial $g(t) = f(t)\overline{f(t)}$ then has rational coefficients, and $g(A)$ also annihilates $E_j$ for each $j \geq 2$. Now, $g(t)$ is invertible in $\Q[[t-\lambda_1]]$ since $g(\lambda_1) \neq 0$, so by truncating the inverse power series it is also invertible in \[\Q[t-\lambda_1]/(t-\lambda_1)^{m_1} = \Q[t]/(t-\lambda_1)^{m_1}.\]  Letting $h(t)$ be a polynomial representative of such an inverse and defining $p(t) = h(t)g(t)$, we have
\[ p(t) = h(t) g(t) = q(t)(t-\lambda_1)^{m_1} + 1 \]
for some $q(t) \in \Q[t]$.  The operator $p(A)$ annihilates $E_2,\dots,E_s$ since $g(A)$ does, it acts as $1$ on $E_1$ since $p(A)-1$ annihilates $E_1$, and $p(d_j) = 0$ for all $j$, as desired.
\end{proof}
This completes the proof of Proposition \ref{prop:pbot}.
\end{proof}

In order to distinguish the contact classes $\Theta(\xi_i,\data)$, we will first identify a class $\Sigma \in H_2(\tilde{C})$ on which the canonical classes of the Lefschetz fibrations $\tilde{X}_i \to S^2$  take different values, as mentioned at the beginning of this section.

\begin{lemma}
\label{lem:integral-class}
There exists an integral class $\Sigma \in H_2(\tilde{C};\Z)$ which is disjoint from a generic fiber of each fibration $\tilde{X}_i\to S^2$ such that $K_{\tilde{X}_i}\cdot\Sigma \neq K_{\tilde{X}_j}\cdot\Sigma$ for $i\neq j$.
\end{lemma} 

\begin{proof}
First, we show that there is an integral class $\Sigma\in H_2(W;\Z)$ such that the pairings \[K_{(W,J_i)}\cdot\Sigma=-c_1(J_i)\cdot\Sigma\] are all distinct. For this, note that, since the $c_1(J_i)$ are all distinct as elements of $H^2(W;\R)$, we have that for each $i\neq j$ the linear function
\[ x \mapsto (c_1(J_i) - c_1(J_{j}))\cdot x : H_2(W;\R) \to \R \]
is not identically zero.  It therefore vanishes on a proper subspace of $H_2(W;\R)$.  The subset of  $H_2(W;\R)$ on which at least one of these functions vanishes is thus at most a union of ${n \choose 2}$ hyperplanes. It follows that the set of points where none of these functions vanish is nonempty and open.  We can then take any rational class in this open set and rescale it to get the desired integral class $\Sigma$.

Letting \[i_\dagger:H_2(W)\to H_2(W^\dagger)\] be the map  in Lemma \ref{lem:w-dagger-smooth}, we have that the class $i_\dagger(\Sigma) \in H_2(W^\dagger)$ defines (via the inclusion $W^\dagger \hookrightarrow \tilde{C}$) a class in $H_2(\tilde{C})$, which is disjoint from a generic fiber $R$. From now on, we will use $\Sigma$ to refer to this class. It then follows from    Lemma~
\ref{lem:canonical-class} that the canonical classes $K_{X_i^\circ}$ all pair differently with $\Sigma$. 
But since $K_{X_i^\circ}$ is just the restriction of $K_{\tilde{X}_i}$, we have that the canonical classes $K_{\tilde{X}_i}$ all evaluate differently on $\Sigma$ as well.
\end{proof}


We will hereafter denote by $\Sigma$ a class in $H_2(\tilde{C})$ which is disjoint from a generic fiber of each $\tilde{X}_i$ and on which the various canonical classes $K_{\tilde{X}_i}$ pair differently, as in Lemma \ref{lem:integral-class}.

By  construction, $b_1(\tilde{X}_i) = 0$, $b_2^+(\tilde{X}_i)>1$ is odd, and   $\tilde{X}_i$ has simple type. Moreover,  the canonical class $K_{\tilde{X}_i}$ is the only basic class $K$ of the Lefschetz fibration $\tilde{X}_i\to S^2$ for which \[K \cdot (-R) = 2 - 2g(R),\] as recorded  in Lemma \ref{lem:can-class}.  We may thus define an analytic function $F_i(s,t)$ in terms of the Donaldson series of $\tilde{X}_i$ by \[ F_i(s,t) = \cD^{\tilde{w}_i}_{\tilde{X}_i}\big(s(-R) + t\Sigma\big) = e^{Q(-sR+t\Sigma)/2} \sum_{r=1}^s \alpha_{\tilde{w}_i,r}e^{K_r \cdot (-sR + t\Sigma)} \]
where the $K_r$ are the basic classes of $\tilde{X}_i$ and each \[\alpha_{\tilde{w}_i,r} = (-1)^{(\tilde{w}_i^2 + K_r\cdot \tilde{w}_i)/2} \beta_r \in \Q\] is nonzero (the second equality follows from Theorem \ref{thm:km-structure}). Since $R$ has self-intersection zero and does not intersect $\Sigma$, we can write $F_i(s,t)$ more simply as
\[ F_i(s,t) = e^{Q(\Sigma)t^2/2} \sum_{r=1}^s \alpha_{\tilde{w}_i,r}e^{K_r \cdot (s(-R) + t\Sigma)}. \]
We will ultimately use this function and its derivatives to specify a certain Donaldson invariant which can be used to distinguish the $\Theta(\xi_i,\data)$, in the manner outlined at the beginning of this section. We first prove some preliminary results regarding these derivatives.

\begin{lemma} \label{lem:pbot-d-ds}
For all $i=1,\dots,n$, we have
\[ \pbot\left(\frac{\partial}{\partial s}\right) F_i(s,t)  = \alpha_i e^{(2-2g)s} \cdot \exp\left( \frac{Q(\Sigma)}{2}t^2 + \left(K_{\tilde{X}_i} \cdot \Sigma\right)t\right), \]
where $K_{\tilde{X}_i}$ is the canonical class of the Lefschetz fibration $\tilde{X}_i \to S^2$ and $\alpha_i\neq 0$ is defined to be $\alpha_{\tilde{w}_i,r}$ for the unique value of $r$ such that $K_r = K_{\tilde{X}_i}$.
\end{lemma}

\begin{proof}
Notice that $F_i(s,t)$ is a sum of terms of the form $c_r e^{(K_r\cdot(-R))s}$ where $c_r$ is a function of $t$ which does not depend on $s$.  Each operator $\frac{\partial}{\partial s} - m$ acts on such a term as
\[ \left(\frac{\partial}{\partial s} - m\right) c_r e^{(K_r \cdot (-R))s} = \left((K_r \cdot (-R)) - m\right) \cdot c_r e^{K_r \cdot (-R)s}, \]
so if $K_r \cdot (-R)$ is a root of $\pbot(t)$ then $\left(\frac{\partial}{\partial s} - (K_r \cdot (-R)\right)$ is a factor of $\pbot\left(\frac{\partial}{\partial s}\right)$ and hence $\pbot\left(\frac{\partial}{\partial s}\right)$ annihilates this term.  By Theorem \ref{thm:km-structure}, \[|K_r \cdot (-R)|\leq 2g-2\] for all $r$. Since $3-2g,4-2g,\dots,2g-2$ are all roots of $\pbot(t)$, it follows that $\pbot\left(\frac{\partial}{\partial s}\right)$ annihilates all terms in the sum except those for which $K_r \cdot (-R) = 2-2g$, and the only basic class satisfying this constraint is $K_r = K_{\tilde{X}_i}$.

For the term involving $K_r = K_{\tilde{X}_i}$, we observe that by construction $\pbot(2-2g) = 1$, so we can write \[\pbot(t) = q(t)(t-(2-2g)) + 1\] for some $q(t) \in \Q[t]$.  Since $\left(\frac{\partial}{\partial s} - (2-2g)\right)$ annihilates the $K_{\tilde{X}_i}$ term, the operator \[\pbot\left(\frac{\partial}{\partial s}\right) = q\left(\frac{\partial}{\partial s}\right)\left(\frac{\partial}{\partial s} - (2-2g)\right) + 1\] acts on it as multiplication by 1, so we are left with
\[ \pbot\left(\frac{\partial}{\partial s}\right) F_i(s,t) = e^{Q(\Sigma)t^2/2} \cdot \alpha_i e^{K_{\tilde{X}_i}\cdot(s(-R) + t\Sigma)}. \]
We can now simplify using $K_{\tilde{X}_i} \cdot s(-R) = (2-2g)s$ to obtain the desired expression.
\end{proof}

\begin{lemma} \label{lem:pbot-di} Let $k_i =  K_{\tilde{X}_i}\cdot\Sigma$ for each $i$. By Lemma \ref{lem:integral-class}, these $k_1,\dots,k_n$ are distinct. 
Then the operator
\[ d_i = \prod_{j \neq i} \frac{1}{k_i-k_j} \left(\frac{\partial}{\partial t} - Q(\Sigma)t - k_j\right) \]
satisfies \[\pbot\left(\frac{\partial}{\partial s}\right) \left(d_i F_j(s,t)\right) = 0\] for all $i\neq j$, and
\[ \pbot\left(\frac{\partial}{\partial s}\right) \left(d_i F_i(s,t)\right) = \alpha_i e^{(2-2g)s} \cdot \exp\left(\frac{Q(\Sigma)}{2}t^2 + k_i t\right). \]
\end{lemma}

\begin{proof}
The operators $\pbot\left(\frac{\partial}{\partial s}\right)$ and $d_i$ commute since they are differential operators in $s$ and $t$ respectively, so for fixed $j$ it suffices to consider
\begin{align*}
d_i \left(\pbot\left(\frac{\partial}{\partial s}\right) F_j(s,t)\right) &= d_i \left(\alpha_j e^{(2-2g)s} \cdot \exp\left(\frac{Q(\Sigma)}{2}t^2 + k_j t\right)\right) \\
&= \alpha_j e^{(2-2g)s} \cdot d_i \left(\exp\left(\frac{Q(\Sigma)}{2}t^2 + k_j t\right)\right).
\end{align*}
For each fixed $l\neq i$, we compute that 
\[ \frac{1}{k_i - k_l}\left(\frac{\partial}{\partial t} - Q(\Sigma)t - k_l\right) \exp\left(\frac{Q(\Sigma)}{2}t^2 + k_jt\right) = \frac{k_j - k_l}{k_i - k_l}\exp\left(\frac{Q(\Sigma)}{2}t^2 + k_jt\right), \]
and so $d_i$ acts on $\exp\left(\frac{Q(\Sigma)}{2}t^2 + k_jt\right)$ as multiplication by
\[ \prod_{l \neq i} \frac{k_j-k_l}{k_i-k_l} = \begin{cases} 1, & j=i \\ 0, & j \neq i, \end{cases} \]
which proves the lemma.
\end{proof}
The  proposition below explains how the derivatives of $F_j(s,t)$ studied in Lemma \ref{lem:pbot-di} are related to Donaldson invariants of  $\tilde{X}_j$.

\begin{proposition} \label{prop:compute-constants}
There are polynomials $g_1(t),g_2(t),\dots,g_n(t) \in \Q[t]$ such that
\[ \left. D^{\tilde{w}_j}_{\tilde{X}_j}\left(\pbot(-R) \cdot \left(1+\frac{x}{2}\right)g_i(\Sigma)\right)=\pbot\left(\frac{\partial}{\partial s}\right) d_i F_j(s,t)\right|_{s=t=0} = \begin{cases} \alpha_i, & i=j \\ 0, & i \neq j \end{cases}\]
where $x \in H_0(\tilde{X}_j;\Z)$ is the class of a point, for all $i$ and $j$.
\end{proposition}
\begin{proof}
The second equality follows from Lemma \ref{lem:pbot-di}, so we will focus on the first equality. We first observe that we can write
\[ d_i = \sum_{k+l \leq n-1} c_{kl}t^k\left(\frac{\partial}{\partial t}\right)^l \]
for some rational constants $c_{kl}$, since $d_i$ is defined as a product of $n-1$ terms which are linear in the noncommuting operators $\frac{\partial}{\partial t}$ and $t$, and since when we expand this product we can use the identity $[\frac{\partial}{\partial t}, t] = 1$ to replace each product $\frac{\partial}{\partial t} \cdot t$ in each monomial with the expression $t \frac{\partial}{\partial t} + 1$.  We expand the power series \[F_j(s,t) = \cD^{\tilde{w}_j}_{\tilde{X}_j}(s(-R)+t\Sigma)\] as
\begin{align*}
\sum_{d=0}^\infty \frac{D^{\tilde{w}_j}_{\tilde{X}_j}\left(\left(1+\frac{x}{2}\right)(s(-R)+t\Sigma)^d\right)}{d!} 
&= \sum_{d=0}^\infty \sum_{p+q=d} {d\choose p} \frac{D^{\tilde{w}_j}_{\tilde{X}_j}\left(\left(1+\frac{x}{2}\right)(-R)^p\Sigma^q\right)}{d!} s^pt^q \\
&= \sum_{p,q \geq 0} D^{\tilde{w}_j}_{\tilde{X}_j}\left( \left(1+\frac{x}{2}\right) (-R)^p \Sigma^q\right) \frac{s^p}{p!} \frac{t^q}{q!}.
\end{align*}
For any integers $k,l \geq 0$ we then have
\begin{align*}
\left.t^k \left(\frac{\partial}{\partial t}\right)^l F_j(s,t)\right|_{t=0}
&= \left.\sum_{p=0}^\infty \sum_{q=l}^\infty D^{\tilde{w}_j}_{\tilde{X}_j}\left( \left(1+\frac{x}{2}\right) (-R)^p \Sigma^q\right) \frac{s^p}{p!} \frac{t^{q-l+k}}{(q-l)!}\right|_{t=0} \\
&= \delta_{k,0} \cdot \sum_{p=0}^\infty D^{\tilde{w}_j}_{\tilde{X}_j} \left( \left(1+\frac{x}{2}\right) (-R)^p \Sigma^l\right) \frac{s^p}{p!}.
\end{align*}
So if we define $g_i(\Sigma) = \sum c_{0,l}\Sigma^l$, it follows that
\[ \left.d_i F_j(s,t)\right|_{t=0} = \sum_{p=0}^\infty D^{\tilde{w}_j}_{\tilde{X}_j}\left((-R)^p \cdot \left(1+\frac{x}{2}\right)g_i(\Sigma)\right) \frac{s^p}{p!}. \]

We argue similarly for the polynomial $\pbot(t) = \sum_k e_k t^k$, with $e_k \in \Q$. Namely, we compute that 
\begin{align*}
\left.\pbot\left(\frac{\partial}{\partial s}\right)\left(\left.d_iF_j(s,t)\right|_{t=0}\right)\right|_{s=0}
&= \sum_k e_k\sum_{p=0}^\infty D^{\tilde{w}_j}_{\tilde{X}_j}\left((-R)^p \cdot \left(1+\frac{x}{2}\right)g_i(\Sigma)\right) \left.\left(\frac{\partial}{\partial s}\right)^k \frac{s^p}{p!}\right|_{s=0} \\
&= \sum_k D^{\tilde{w}_j}_{\tilde{X}_j}\left(e_k(-R)^k \cdot \left(1+\frac{x}{2}\right)g_i(\Sigma)\right).
\end{align*} 
On the left side, the operations of specializing to $t=0$ and applying $\pbot\left(\frac{\partial}{\partial s}\right)$ commute, and we have $\sum_k e_k(-R)^k = \pbot(-R)$, so this completes the proof.
\end{proof}

We are now ready to prove Theorem \ref{thm:main-linear-independence}. 
\begin{proof}[Proof of Theorem~\ref{thm:main-linear-independence}]
According to Proposition \ref{prop:compute-constants}, we have that \begin{equation*}\label{eqn:donaldson}D^{\tilde{w}_j}_{\tilde{X}_j}\left(\pbot(-R) \cdot \left(1+\frac{x}{2}\right)g_i(\Sigma)\right)=\begin{cases} \alpha_i, & i=j \\ 0, & i \neq j. \end{cases}\end{equation*} Viewing $\pbot(-R)$ as an element of $\bA(X_j)$ and $\left(1+\frac{x}{2}\right)g_i(\Sigma)$ as an element of $\bA(\tilde{C})$, this Donaldson invariant  is given by the pairing 
\[ \left\langle \Psi_{\tilde{w}_j,X_j}(\pbot(-R)), \Psi_{\tilde{w},\tilde{C}}\left(\left(1+\frac{x}{2}\right)g_i(\Sigma)\right) \right\rangle = \begin{cases} \alpha_i, & i=j \\ 0, & i \neq j \end{cases} \]
on $I_*(-\bar{Y})$.  Proposition~\ref{prop:pbot} reduces this identity to
\[ \left\langle \Theta(\xi_j,\data), \Psi_{\tilde{w},\tilde{C}}\left(\left(1+\frac{x}{2}\right)g_i(\Sigma)\right) \right\rangle = \begin{cases} \alpha_i, & i=j \\ 0, & i \neq j. \end{cases} \]
So, for any fixed $i$, the kernel of the linear map
\[ \left\langle \,\cdot\,, \Psi_{\tilde{w},\tilde{C}}\left(\left(1+\frac{x}{2}\right)g_i(\Sigma)\right) \right\rangle: I_*(-\bar{Y}) \to \C \]
contains all of the elements $\iinvt(\xi_j,\data)$ for $j \neq i$, but does not contain $\iinvt(\xi_i,\data)$ since $\alpha_i \neq 0$, so $\iinvt(\xi_i,\data)$ cannot be a linear combination of the other $\iinvt(\xi_j,\data)$.  We conclude that the elements $\iinvt(\xi_1,\data),\dots,\iinvt(\xi_n,\data)$ are linearly independent.
\end{proof}

%
%
%
%

\section{Representations of fundamental groups and L-spaces} \label{sec:representations}
In Subsection \ref{ssec:reps}, we describe some connections between the rank of the sutured instanton homology of $Y(p)$ and the existence of irreducible $SU(2)$ representations of $\pi_1(Y)$. We then use these connections, in combination with Theorem \ref{thm:main-linear-independence}, to prove Theorems \ref{thm:main} and \ref{thm:mainqhs}. As we shall see, we will be able to deduce, under favorable circumstances, the existence of an irreducible representation when \[\rank\SHItfun(Y(p))>|H_1(Y)|.\] In Subsection \ref{ssec:lspaces}, we prove several results describing when manifolds obtained via Dehn surgery on knots in the 3-sphere satisfy the above inequality. We will use these results in Section \ref{sec:examples}, in combination with Theorem \ref{thm:main-linear-independence} and the results from Subsection \ref{ssec:reps}, to prove the existence of irreducible representations for various manifolds obtained via Dehn surgery.

\subsection{Representations and the rank of instanton homology}
\label{ssec:reps}
Below, we make concrete some relationships between  instanton Floer homology and irreducible $SU(2)$ representations of $\pi_1$.
For the reader not used to thinking about $SU(2)$, we recall that reducible representations $\pi_1(Y) \to SU(2)$ are precisely those with abelian image; or, equivalently, those which factor through  $H_1(Y)$.  Thus, if $Y$ is an integer homology sphere then a representation is reducible iff it is trivial. More generally, if $Y$ is a rational homology sphere then one can show that the image of any reducible representation is a finite cyclic subgroup of $SU(2)$.

As mentioned in the introduction, the connections between sutured instanton homology and $SU(2)$ representations come from a relationship between $\SHItfun$ and another version of instanton Floer homology, as follows.

Let $I^\natural$ denote the singular instanton knot homology of \cite{km-khovanov}. Following Scaduto \cite{scaduto},  we  define the closed 3-manifold invariant
\[ I^\#(Y) := I^\natural(Y,U), \] where $U$ is an unknot in $Y$. We have the following sequence of isomorphisms
\[ \SHItfun(Y(p)) \cong \SHItfun(Y(U)) = \KHItfun(Y,U) \cong I^\natural(Y,U)\otimes \C = I^\#(Y) \otimes \C, \]
where $Y(U)$ refers to the complement of $U$ with meridional sutures.  The first of these  comes from identifying $Y(U)$ as the result of attaching a contact 1-handle to $Y(p)$, as in the proof of \cite[Lemma~4.14]{bs-shi}; the second is the definition of the instanton knot homology  $\KHItfun$; and the third is \cite[Proposition~1.4]{km-khovanov}.  Since 
\[ I^\#(Y; \Q)\cong I^\#(-Y; \Q) \]
\cite[Section~7.4]{scaduto}, Theorem~\ref{thm:main-linear-independence} implies the following.

\begin{corollary} \label{cor:rank-i-sharp}
Suppose $W$ is a  compact 4-manifold with boundary $Y$ which admits $n$ Stein structures  whose Chern classes  are distinct in $H^2(W;\R)$.  Then $\rank I^\#(Y)\geq n$.\qed
\end{corollary}

We can now use Corollary~\ref{cor:rank-i-sharp} to prove Theorem~\ref{thm:main}, restated for convenience below.

{
\renewcommand{\thetheorem}{\ref{thm:main}}
\begin{theorem}
If $Y$ is the boundary of a Stein domain which is not an integer homology ball, then there is a nontrivial homomorphism $\pi_1(Y) \to SU(2)$.
\end{theorem}
\addtocounter{theorem}{-1}
}

\begin{proof}Suppose $Y$ is the boundary of a Stein domain $(W,J)$ which is not an integer homology ball. We will assume that $Y$ is an integer homology sphere, since otherwise there is clearly a nontrivial representation \[\pi_1(Y) \to H_1(Y) \to U(1)\subset SU(2)\] as mentioned in the introduction. The fact that $W$ can be obtained from the $4$-ball by attaching handles of index $1$ and $2$ implies that $H_3(W)=0$ and that $H_2(W)$ is torsion-free. In reverse, $W$ can be obtained by thickening $Y$ and attaching handles of index $2$, $3$, and $4$, which implies that $H_1(Y)$ surjects onto $H_1(W)$ and, therefore, that $H_1(W) = 0$.  The universal coefficient theorem then tells us that $H^2(W) \cong H_2(W)$. The assumption that $W$ is not a homology ball, together with this isomorphism and the fact that $H_2(W)$ is torsion-free, then implies that $b_2(W) > 0$.  
With this established, we prove below that the instanton Floer homology $HF(Y)$ defined in \cite{floer-invariant} is nontrivial, which implies the existence of a nontrivial homomorphism $\pi_1(Y)\to SU(2)$.

By work of Lisca and Mati\'c \cite[Theorem~3.2]{lisca-matic}, we can embed $(W,J)$ holomorphically into a minimal K\"ahler surface $X$ of general type, such that $Y$ separates $X$ into pieces $W$ and $S = \overline{X \smallsetminus W}$ with $b_2^+(S) > 1$.  The Donaldson invariants of $X$ are nontrivial \cite[Theorem~C]{donaldson-invariants}, and if $b_2^+(W) \geq 1$ then they can be recovered from a pairing on $HF(Y)$, which implies that $HF(Y)$ is nontrivial, as desired. Let us therefore assume that $b_2^+(W)=0$. Note that $Q_W$  is unimodular since $Y$ is a homology sphere; thus, $b_2^+(W)=0$ implies that $Q_W$ is  negative definite. That is, $b_2^-(W) = b_2(W)>0.$
We  consider two  cases below.

First, suppose $c_1(J)\neq 0$. Then the conjugate Stein structure $\bar{J}$ satisfies $c_1(\bar{J}) = -c_1(J)$. These are distinct as real cohomology classes since $H_2(W)$ is torsion-free. Thus $I^\#(Y)$ has rank at least 2 by Corollary~\ref{cor:rank-i-sharp}. It then follows from \cite[Theorem~1.3]{scaduto} that Fr{\o}yshov's reduced Floer homology $\hfhat(Y)$ \cite{froyshov} is nontrivial (taking coefficients in $\C$).  But $\hfhat(Y)$ is defined as a subquotient of $HF(Y)$, so  $HF(Y)$ must be nontrivial as well.

Next, suppose $c_1(J)=0$.  Since the first Chern class of $(X,J)$ is characteristic for the intersection form on $X$ and restricts to $c_1(J)$ on $W$, we have that
\[ 0=c_1(J) \cdot \Sigma \equiv \Sigma \cdot \Sigma \pmod{2} \]
for any smoothly embedded surface $\Sigma \subset W$. Thus, $Q_W$ is an even unimodular form.  Observe that $Q_W$ cannot be diagonalizable over $\Z$, since if it were, there would have to be a surface in $W$ of self-intersection $-1$. Since $Q_W$ is not diagonalizable, the Fr{\o}yshov invariant
\[ h(Y) = \frac{1}{2}\left(\chi(HF(Y)) - \chi(\hfhat(Y))\right) \]
must be positive by \cite[Theorem~3]{froyshov}.  This means that at least one of $HF(Y)$ and $\hfhat(Y)$ must be nontrivial. In either case, $HF(Y)$ is nontrivial.
\end{proof}

In the proof of Theorem \ref{thm:main}, we used a relationship between $I^\#(Y)$ and $\hfhat(Y)$ due to Scaduto \cite[Theorem~1.3]{scaduto} to argue for the existence of nontrivial $SU(2)$ representations when $Y$ is a homology sphere with $\rank(I^\#(Y)) > 1$. But that relationship is much stronger than was  needed.  Indeed, one can deduce the existence of nontrivial $SU(2)$ representations  more directly from the construction of $I^\#(Y)$, by an approach which generalizes much more readily to proving the existence of \emph{irreducible}  representations when $Y$ is  merely a rational homology sphere (in which case $HF(Y)$ is generally not defined). More concretely, Scaduto \cite[Corollary~1.4]{scaduto} showed that $I^\#(Y)$ has Euler characteristic $|H_1(Y)|$, which implies that \[\rank I^\#(Y) \geq |H_1(Y)|.\]  If this inequality is strict then we can often show that there exists an irreducible representation.  The simplest version of this, for homology spheres, is as below (recall that irreducible is equivalent to nontrivial for homology spheres); we will generalize it to rational homology spheres in Theorem~\ref{thm:reducibles-morse-bott}.

\begin{theorem} \label{thm:su2-reps}
Let $Y$ be a homology sphere with $\rank I^\#(Y) >|H_1(Y)|=1$.  Then there is a nontrivial representation $\pi_1(Y) \to SU(2)$.
\end{theorem}

Before proving Theorem \ref{thm:su2-reps}, we recall   the connection between $I^\#(Y)$ and $SU(2)$ representations. This will be important later as well.

Let $m \subset Y$ denote a meridian of $U$, let $U^\natural$ be the Hopf link $U \cup m$, and let $\omega \subset Y$ be an arc from $U$ to $m$ inside of a meridional disk bounded by $m$.  Then $I^\#(Y):=I^\natural(Y,U)$ is obtained as the Morse homology of a Chern-Simons functional whose critical points are conjugacy classes of representations
\[ \rho: \pi_1(Y \ssm (U^\natural \cup \omega)) \to SU(2) \]
with holonomy conjugate to $i$ (identifying $SU(2)$ with the unit quaternions) around meridians $\mu_U, \mu_m$ of the Hopf link $U^\natural$ and equal to $-1$ around a meridian $\mu_\omega = [\mu_U,\mu_m]$ of $\omega$.  We will write \[(Y, U^\natural \cup \omega) = Y \# (S^3,U^\natural \cup \omega),\] so that each such representation is a homomorphism
\[ \rho: \pi_1(Y) \ast \pi_1(S^3 \ssm (U^\natural \cup \omega)) \to SU(2). \]
The restriction of any such $\rho$ to $\pi_1(S^3 \ssm (U^\natural \cup \omega))$ is conjugate to the unique homomorphism sending $\mu_U$ and $\mu_m$ to $i$ and $j$, respectively; see the proof of Lemma \ref{lem:natident}.

\begin{proof}[Proof of Theorem \ref{thm:su2-reps}]
Suppose all homomorphisms $\pi_1(Y) \to SU(2)$ are trivial. Then it follows from the discussion above that there is only one homomorphism \[ \rho: \pi_1(Y) \ast \pi_1(S^3 \ssm (U^\natural \cup \omega)) \to SU(2), \] up to conjugation: that which is trivial on $\pi_1(Y)$ and sends meridians of $U$ and $m$ to $i$ and $j$, respectively. Its conjugacy class $[\rho]$ is therefore the unique critical point of  the unperturbed Chern-Simons functional defining $I^\natural(Y,U)$ as above. We claim, moreover, that this critical point is nondegenerate, though we postpone the proof to Proposition~\ref{prop:cs-morse-bott}. 

It follows that $I^\#(Y)=I^\natural(Y,U)$ is the homology of a chain complex whose only generator is $[\rho]$; we may need a small holonomy perturbation to ensure the regularity of the moduli spaces which define the differential but, according to \cite[Proposition~3.5]{km-khovanov}, we can do this without changing the critical point set. We have thus shown that if there are no nontrivial representations $\pi_1(Y) \to SU(2)$ then  $\rank I^\#(Y)=1$, proving the theorem.\end{proof}

We now discuss the question of whether the conjugacy class $[\rho]$ in the proof of Theorem \ref{thm:su2-reps} is a nondegenerate critical point, or, more generally, whether that of a reducible representation is a Morse-Bott critical point.  For that, we consider three different representation varieties:
\begin{align*}
R(Y) &= \Hom(\pi_1(Y), SU(2)), \\
R(Y,U) &= \{ \rho: \pi_1(Y) \ast \langle \mu_U,\mu_m \rangle \to SU(2) \mid \rho([\mu_U,\mu_m]) = -1 \},   \\
R^\natural(Y,U) &= R(Y,U) / SU(2),
\end{align*}
where $SU(2)$ acts on $R(Y,U)$ by conjugation.

\begin{lemma}
\label{lem:natident}
There is a natural identification
\[ R(Y) \cong R^\natural(Y,U) \]
sending $\rho\in R(Y)$ to the equivalence class of the unique $\tilde{\rho}\in R(Y,U)$ such that \[\tilde{\rho}|_{\pi_1(Y)} = \rho\,\,\,\textrm{ and }\,\,\,\tilde{\rho}(\mu_U) = i\,\,\,\textrm{ and }\,\,\,\tilde{\rho}(\mu_m) = j.\] 
\end{lemma}

\begin{proof}
The map is clearly injective. To see that it is surjective, we recall that the adjoint representation $SU(2) \to SO(3)$ can be realized as the action by conjugation on the space $\R^3$ of purely imaginary quaternions.  If unit quaternions $x,y \in SU(2)$ satisfy $xyx^{-1}y^{-1} = -1$, then $x$ and $y$ are  conjugate to $-x$ and $-y$, respectively. Therefore, $x$ and $y$ are traceless and thus conjugate to $i$. In particular, $x$ and $y$ are purely imaginary, meaning that we may view them as unit vectors of this $\R^3$. The fact that $xy=-yx$ then implies that these unit vectors are orthogonal, so there is a unique element of $SO(3)$ taking $x,y$ to $i,j$.  Thus, any element of $ R(Y,U)$ is conjugate to a unique $\tilde{\rho}$ such that $\tilde{\rho}(\mu_U) = i$ and $\tilde{\rho}(\mu_m) = j$.
\end{proof}

Note that $R^\natural(Y,U) \cong R(Y)$ is the critical point set of the unperturbed Chern-Simons functional used to construct $I^\#(Y)=I^\natural(Y,U)$, since, as above, $\pi_1(Y \ssm (U^\natural \cup \omega))$ is the free product of $\pi_1(Y)$ with the group 
\[ \pi_1(S^3 \ssm (U^\natural \cup \omega)) = \langle \mu_U, \mu_m, \mu_\omega \mid \mu_\omega = [\mu_U,\mu_m] \rangle. \]
We remark that every element of $R(Y,U)$ has stabilizer $\{\pm 1\} \subset SU(2)$, since these are the only elements which commute with both $i$ and $j$.

We may now study the issue of nondegeneracy of critical points of the Chern-Simons functional computing $I^\#(Y) = I^\natural(Y,U).$ Recall that a critical point $x$ is nondegenerate if the Hessian at $x$ is nondegenerate. The critical point $x$ is Morse-Bott nondegenerate if the kernel of the Hessian at $x$ is equal to the tangent space to the component of the critical manifold containing $x$. The functional is said to be Morse (resp.\ Morse-Bott) if all of its critical points are nondegenerate (resp.\ Morse-Bott nondegenerate). Our goal below is to describe nondegeneracy (or Morse-Bott nondegeneracy) for points of $R^\natural(Y,U)$ more naturally in terms of the corresponding points of the conceptually simpler $R(Y)$. 

To carry this out, we first compute the Zariski tangent spaces to these varieties at various representations, following \cite[Section~7.2]{hhk}.  
Given a finite presentation
\[ \pi_1(Y) = \langle g_1,\dots,g_m \mid w_1, \dots, w_n \rangle, \]
let
\[F: SU(2)^m \to SU(2)^n\,\,\,\textrm{ and }\,\,\,\tilde{F}:SU(2)^{m+2} \to SU(2)^{n+1}\] be the maps defined by
\begin{align*}
F(x_1,\dots,x_m) &= (w_1(x_1,\dots,x_m), w_2(x_1,\dots,x_m), \dots, w_n(x_1,\dots,x_m)) \\
\tilde{F}(x_1,\dots,x_m,y,z) &= (F(x_1,\dots,x_m), yzy^{-1}z^{-1}),
\end{align*}
so that \[R(Y) = F^{-1}(1,1,\dots,1)\,\,\,\textrm{ and } \,\,\,R(Y,U) = \tilde{F}^{-1}(1,1,\dots,1,-1).\]
The Zariski tangent space to $R(Y)$ at a point $\rho$ is therefore given by $T_\rho R(Y)=\ker(dF_\rho)$. Similarly, for each $\tilde{\rho} \in R(Y,U)$, we have that $T_{\tilde{\rho}}R(Y,U) = \ker(d\tilde{F}_{\tilde{\rho}})$.  The kernel of the Hessian at the corresponding class $[\tilde\rho] \in R^\natural(Y,U)$ is identified with the quotient of $T_{\tilde\rho}R(Y,U)$ by the tangent space to the $SU(2)$ orbit of $\tilde\rho$; the orbit is 3-dimensional since $SU(2)$ acts with stabilizer $\{\pm 1\}$, so this kernel has dimension $\ker(d\tilde{F}_{\tilde\rho}) - 3$.

Suppose $\rho \in R(Y)$ corresponds to the class $[\tilde{\rho}] \in R^\natural(Y,U)$ via the bijection in Lemma \ref{lem:natident}, where $\tilde{\rho}(\mu_U)=i$ and $\tilde{\rho}(\mu_m)=j$. Then we have that $d\tilde{F}_{\tilde{\rho}}$ splits in block form as
\[ d\tilde{F}_{\tilde{\rho}} = \left(\begin{array}{c|c} dF_{\rho} & 0 \\ \hline 0 & dF'_{(i,j)} \end{array}\right) \]
where $F': SU(2) \times SU(2) \to SU(2)$ is the commutator map $F'(y,z) = yzy^{-1}z^{-1}$.  
It was observed in \cite[Section~7.2]{hhk} that $\ker dF'_{(i,j)}$ is 3-dimensional, so the kernel of the Hessian of the Chern-Simons functional at $[\tilde\rho]$ has dimension
\[ \dim \ker(d\tilde{F}_{\tilde\rho}) - 3 = \dim \ker(dF_\rho) = \dim T_\rho R(Y). \]
This allows us to relate questions of nondegeneracy to the dimension of $T_\rho R(Y)$. We use these observations to prove the following.

\begin{proposition} \label{prop:cs-morse-bott}
Suppose $Y$ is a rational homology sphere.  Then all representations in $R(Y)$ with image contained in $\{\pm 1\}$ correspond to isolated, nondegenerate critical points  of the Chern-Simons functional defining $I^\#(Y)=I^\natural(Y,U)$.  More generally, if all representations in $R(Y)$ are reducible then this  functional is Morse-Bott if and only if $H^1(Y;\ad(\rho)) = 0$ for all $\rho \in R(Y)$.
\end{proposition}

\begin{proof}
Weil \cite{weil} showed  that if $H^1(Y;\ad(\rho)) = 0$ then $\rho \in R(Y)$ has some neighborhood consisting only of points in the $SU(2)$-orbit of $\rho$.  This occurs, for instance, when $\img(\rho)\subset\{\pm 1\}$, since $\ad(\rho)$ is then a trivial representation and $H^1(Y;\R) = 0$. It follows that such $\rho$ are isolated points of $R(Y)$.

More generally, recall that a representation $\rho \in R(Y)$ is reducible precisely when it factors through a homomorphism $H_1(Y) \to SU(2)$, where $H_1(Y)$ is finite by assumption.  It is then easy to check that if $R(Y)$ consists only of reducibles then once again some neighborhood of $\rho$ consists only of conjugates of $\rho$, so its connected component is precisely this $SU(2)$-orbit, and therefore has dimension $3-\dim Z(\rho)$ where $Z(\rho)$ is the centralizer of $\img(\rho)$. (In fact, these components are either points or 2-spheres, as elaborated in the proof of Theorem \ref{thm:reducibles-morse-bott}.)

As in \cite{weil}, the tangent space $T_\rho R(Y)$ is naturally identified with the space of 1-cocycles $Z^1(Y;\ad(\rho))$ on $Y$ valued in $\ad(\rho)$, where the tangent space to the $SU(2)$-orbit of $\rho$ is the space of coboundaries $B^1(Y;\ad(\rho))$. It follows that
\begin{equation}\label{eqn:dimzad}
\dim T_\rho R(Y) = 3-\dim Z(\rho) + \dim H^1(Y;\ad(\rho)).
\end{equation}

Suppose that $\img(\rho)\subset \{\pm 1\}$, so that $\dim Z(\rho)=3$ and  $\dim H^1(Y;\ad(\rho))=0$. Then  \eqref{eqn:dimzad} implies that $\dim T_\rho R(Y) = 0$.  According to the discussion above, the kernel of the Hessian of the Chern-Simons functional at the corresponding $[\tilde\rho]\in R^\natural(Y,U)$ is therefore trivial, so $[\tilde\rho]$ is an isolated, nondegenerate critical point, as claimed.

Next, suppose that  $R(Y)$ consists only of reducibles.  Then for each $\rho \in R(Y)$, the critical manifold containing the corresponding $[\tilde\rho]\in R^\natural(Y,U)$ has dimension $3-\dim Z(\rho)$, as argued above. The point $[\tilde\rho]$ is then a Morse-Bott critical point if and only if the kernel of the Hessian of the Chern-Simons functional also has dimension $3-\dim Z(\rho)$. But the latter is equivalent to $\dim T_\rho R(Y)=3 - \dim Z(\rho)$, which, from \eqref{eqn:dimzad}, holds if and only if $H^1(Y;\ad(\rho)) = 0$, as claimed.
\end{proof}

The dimension of $T_\rho R(Y)$ at a reducible representation $\rho$ was previously studied by Boyer and Nicas \cite{boyer-nicas}.  The adjoint representation \[\ad(\rho): G \to \Aut(\mathfrak{su}(2)) = SO(3)\] sends $x \in G$ to the derivative at $1 \in SU(2)$ of conjugation by $\rho(x)$. If $b_1(Y) = 0$ then $\img(\rho)$ is a finite cyclic subgroup of $SU(2)$, as mentioned at the beginning of this subsection, and it follows that $\ad(\rho)$ has finite cyclic image as well.  Suppose that the image of $\ad(\rho)$ is a cyclic group $C_n$, realized as the $n$th roots of unity in a $U(1)$ subgroup of $SO(3)$, and assume that $n > 1$ since we have $\img(\rho) \subset \{\pm 1\}$ otherwise.  For all $d \mid n$ we define 
\[ \pi_d = \ker\left( \pi_1(Y) \xrightarrow{\ad(\rho)} C_n \xrightarrow{x \mapsto x^{n/d}} C_d \right) = \ker\left(\pi_1(Y) \xrightarrow{\ad(\rho^{n/d})} C_d\right), \]
noting that $\rho^{n/d}$ is in this case a well-defined representation of $\pi_1(Y)$ onto $C_d$, and that $\pi_d$ is the fundamental group of a normal, $d$-fold cyclic cover $Y_d$ of $Y=Y_1$.  Then \cite[Theorem~1.1 and Remark~1.2]{boyer-nicas} asserts that 
\begin{equation}
\label{eqn:dimad} \dim H_1(Y; \ad(\rho)) = \frac{2}{\varphi(n)} \sum_{d \mid n} \mu\left(\frac{n}{d}\right) b_1(Y_d), \end{equation}
where $\varphi$ and $\mu$ are the Euler totient function and M\"obius function, respectively.  We  use this formula to prove the following.

\begin{proposition} \label{prop:morse-bott-criterion}
Suppose $Y$ is a rational homology sphere and  every representation $\rho:\pi_1(Y)\to SU(2)$ is reducible. Then the unperturbed Chern-Simons functional defining $I^\#(Y)$ is Morse-Bott if and only if the finite cyclic cover of $Y$ corresponding to each subgroup $\ker(\ad(\rho)) \subset \pi_1(Y)$ is  a rational homology sphere. 
\end{proposition}

\begin{proof}Suppose $Y$ satisfies the hypotheses of Proposition \ref{prop:morse-bott-criterion}. Borrowing notation from above, it suffices by Proposition \ref{prop:cs-morse-bott} to show that the finite cyclic cover $Y_n$  corresponding to the subgroup $\ker(\ad(\rho)) \subset \pi_1(Y)$ is a rational homology sphere iff $\dim H_1(Y;\ad(\rho)) = 0$. 

Note that  $Y_n$ is a normal, $n/d$-fold cyclic cover of each  $Y_d$, so a transfer argument shows that $b_1(Y_d) \leq b_1(Y_n)$ for all $d \mid n$. Hence, if $b_1(Y_n)=0$ then  $b_1(Y_d)=0$ for each $Y_d$, which implies that $\dim H_1(Y;\ad(\rho)) = 0$ by the formula in \eqref{eqn:dimad}.

Conversely, if $b_1(Y_n) > 0$, let $d \geq 1$ be the smallest divisor of $n$ for which $b_1(Y_d) > 0$, and note that in fact $d \geq 2$ since $Y_1 = Y$.  Replacing $\rho$ with the representation $\rho^{n/d}$ (satisfying $\img(\ad(\rho^{n/d})) = C_d$), we have
\[ \dim H_1(Y; \ad(\rho^{n/d})) = \frac{2}{\varphi(d)} \sum_{e \mid d} \mu\left(\frac{d}{e}\right) b_1(Y_{e}) = \frac{2 b_1(Y_d)}{\varphi(d)} > 0. \qedhere \]
\end{proof}

If all of the finite cyclic covers of $Y$ corresponding to  subgroups $\ker(\ad(\rho)) \subset \pi_1(Y)$ as in Proposition \ref{prop:morse-bott-criterion} are rational homology spheres, then $\pi_1(Y)$ is said to be \emph{cyclically finite} in the language of \cite{boyer-nicas}.  For example, if the universal abelian cover $\tilde{Y}$ is a rational homology sphere, then $\pi_1(Y)$ is cyclically finite since $\tilde{Y}$ covers all finite cyclic covers of $Y$.  The notion of cyclical finiteness allows us to generalize Theorem~\ref{thm:su2-reps} as follows.

\begin{theorem} \label{thm:reducibles-morse-bott}
Suppose $Y$ is a rational homology sphere with $\pi_1(Y)$ cyclically finite and $\rank I^\#(Y)>|H_1(Y)|$. Then there is an irreducible representation $\pi_1(Y) \to SU(2)$.
\end{theorem}

\begin{proof}
Suppose $Y$ satisfies the hypotheses of Theorem \ref{thm:reducibles-morse-bott}. And suppose, for a contradiction, that every representation $\pi_1(Y) \to SU(2)$ is reducible. Then all representations have finite cyclic image, so up to conjugation they can be arranged to lie in a fixed $U(1)$ subgroup $\{e^{i\theta}\}$ of $SU(2)$.  This defines a surjection from the 1-dimensional characters $\chi: H_1(Y) \to U(1)$ to the conjugacy classes of representations $R(Y)$. Let $C_\chi$ denote the conjugacy class corresponding to $\chi$. It is straightforward to check that for two distinct characters $\chi,\chi'$, $C_\chi = C_{\chi'}$ iff $\chi' = \chi^{-1}$. The number of conjugacy classes in $R(Y)$ is therefore equal to \[ \#\{\chi\mid\chi=\chi^{-1}\}+\frac{1}{2}\cdot\#\{\chi\mid\chi\neq \chi^{-1}\}.\]
Since we are assuming that every representation is reducible, the connected components of $R(Y)$ are precisely these conjugacy classes, as in the proof of Proposition \ref{prop:cs-morse-bott}. Moreover, $C_\chi$ is a point if $\img(\chi)$ is central in $SU(2)$, which happens precisely when $\chi= \chi^{-1}$;  and $C_\chi$ is a 2-sphere otherwise. That is,
\[H_*(C_\chi) = \begin{cases}
\mathbb{Z}& \textrm{if } \chi=\chi^{-1},\\
\mathbb{Z}^2 & \textrm{if }\chi \neq \chi^{-1}.
\end{cases}
\]
It follows that $H_*(R(Y))$ is free abelian of rank equal to \[ \#\{\chi\mid\chi=\chi^{-1}\}+2\cdot\frac{1}{2}\cdot\#\{\chi\mid\chi\neq \chi^{-1}\},\] which is just the total number of characters of $H_1(Y)$, which is equal to $|H_1(Y)|$.

According to Proposition~\ref{prop:morse-bott-criterion} and the hypotheses of Theorem \ref{thm:reducibles-morse-bott}, the unperturbed Chern-Simons functional defining $I^\#(Y)$ is Morse-Bott with critical points given precisely by $R(Y)$.  We then have a Morse-Bott spectral sequence as in \cite{furuta-steer}, which has $E_2$ page $H_*(R(Y))$ and which converges to $I^\#(Y)$.  It follows that \[\rank I^\#(Y)\leq \rank H_*(R(Y)) = |H_1(Y)|.\] On the other hand, we also have the reverse inequality since $|H_1(Y)|$ is the Euler characteristic of $I^\#(Y)$. Thus, $\rank I^\#(Y) = |H_1(Y)|$, a contradiction.
\end{proof}

Theorem \ref{thm:mainqhs}, restated below for convenience, then follows immediately from Theorem \ref{thm:reducibles-morse-bott} and Corollary \ref{cor:rank-i-sharp}.
{
\renewcommand{\thetheorem}{\ref{thm:mainqhs}}
\begin{theorem} Suppose $Y$ is a rational homology sphere with $\pi_1(Y)$ cyclically finite. If $Y$ bounds a $4$-manifold $W$ which admits $n>|H_1(Y)|$ Stein structures whose Chern classes are distinct in $H^2(W;\R)$, then there is an irreducible representation $\pi_1(Y)\to SU(2)$. \qed
\end{theorem}
\addtocounter{theorem}{-1}
}

We conclude with some examples in which $\pi_1(Y)$ is known to be cyclically finite. The first two concern Dehn surgeries on knots in $S^3$.

\begin{proposition}[{\cite[Lemma~1.4]{boyer-nicas}}] \label{prop:cf-alexander-polynomial}
Suppose $K\subset S^3$ is a knot and fix some rational $p/q\neq 0$.  Then \[\pi_1(S^3_{p/q}(K))\]  is cyclically finite iff no zero of $\Delta_K(t^2)$ is a $p$th root of unity, where $\Delta_K(t)$ is the Alexander polynomial of $K$.
\end{proposition}

\begin{corollary} 
\label{cor:cycfinitesurg}Fix some rational $p/q\neq 0$ and suppose that \[\rank I^\#(S^3_{p/q}(K))>|p|=|H_1(Y)|.\] Then there is an irreducible representation \[\pi_1(S^3_{p/q}(K)) \to SU(2)\] if no zero of $\Delta_K(t^2)$ is a $p$th root of unity. 
\end{corollary}

\begin{proposition}[\cite{cohen}]
\label{prop:cohen}
If $H_1(Y)$ is cyclic of finite order $p^e$ for some prime $p$, then the universal abelian cover $\tilde{Y}$ is a rational homology sphere.
\end{proposition}

\begin{corollary} \label{cor:rank-at-most-5}
Suppose that $|H_1(Y)| \leq 5$ and $\rank(I^\#(Y)) > |H_1(Y)|$. There there is an irreducible representation $\pi_1(Y) \to SU(2).$
\end{corollary}

\begin{proof}
Suppose  $Y$ satisfies the hypotheses of Corollary \ref{cor:rank-at-most-5}. And suppose, for a contradiction, that every representation $\pi_1(Y)\to SU(2)$ is reducible.  If $H_1(Y)$ is cyclic then its order is a prime power, which implies that $\pi_1(Y)$ is cyclically finite by Proposition \ref{prop:cohen}.  If instead $H_1(Y)$ is not cyclic then it must be $\Z/2\Z \oplus \Z/2\Z$, which implies that every $SU(2)$ representation $\rho$ is reducible with image in $\{\pm 1\}$ and hence that $\ad(\rho)$ is trivial, so $\pi_1(Y)$ is again cyclically finite.  Theorem~\ref{thm:reducibles-morse-bott} then says that $I^\#(Y)$ has rank $|H_1(Y)| $, a contradiction.
\end{proof}


\subsection{Dehn surgery and L-spaces}
\label{ssec:lspaces}

We recall that if $Y$ is a rational homology sphere, then $I^\#(Y)$ has rank at least $|H_1(Y;\Z)|$ since the latter quantity is its Euler characteristic.  (If instead $b_1(Y) > 0$, then this is interpreted as $\chi(I^\#(Y)) = 0$.)  If in fact we have
\[ \rank I^\#(Y) = |H_1(Y;\Z)|, \]
then we will call $Y$ an \emph{instanton L-space}, in analogy with L-spaces in Heegaard Floer homology \cite{osz-lens}.  This class was shown in \cite{scaduto} to contain all rational homology spheres which are branched double covers of quasi-alternating links, which notably includes all lens spaces, as well as the Poincar\'e homology sphere $\Sigma(2,3,5)$.

\begin{remark} \label{rem:homology-sphere-l-space}
In the proof of Theorem~\ref{thm:main}, we saw as an application of  Corollary~\ref{cor:rank-i-sharp} that a homology sphere $Y$ is not an instanton L-space if it admits a Stein filling $(W,J)$ with $c_1(J) \neq 0$.
\end{remark}

In this subsection we will study when Dehn surgeries on knots in $S^3$ can produce instanton L-spaces, where we take the coefficients to be a field of characteristic zero.  The results in this subsection will be familiar to experts in Floer homology, since they are analogues of results about L-space surgeries in both monopole Floer homology \cite{kmos} and Heegaard Floer homology \cite{osz-lens}.

Our principal tool is the surgery exact triangle, due originally to Floer \cite{floer-surgery} but proved in the form we use here by Scaduto \cite{scaduto}.  If $L$ is a framed knot in a 3-manifold $Y$, and $\lambda \subset Y$ is a closed 1-manifold disjoint from $L$, then there is a surgery exact triangle
\[ \dots \to I^\#(Y;\lambda) \to I^\#(Y_0(L);\lambda \cup \mu) \to I^\#(Y_1(L);\lambda) \to \dots, \]
where $\mu$ is a core of the 0-surgery (taken with respect to the given framing) and in each group we have twisted $I^\#$ by using an $SO(3)$ bundle with $w_2$ Poincar\'e dual to the indicated curve.  If $H_1(Y \ssm L) \cong \Z$ and all three manifolds in the exact triangle are rational homology spheres then at most one of them has $H_1$ of even order, so we can choose $\lambda$ to make the corresponding bundle trivial, and then the other two manifolds are $\Z/2\Z$-homology spheres so the twisting has no effect.  This gives us an exact triangle
\begin{equation} \label{eq:surgery-exact-triangle}
\dots \to I^\#(Y) \to I^\#(Y_0(L)) \to I^\#(Y_1(L)) \to \dots
\end{equation}
involving only the ordinary (i.e.\ untwisted) $I^\#$ groups.  (Likewise, if $Y$ is a homology sphere and $L$ is 0-framed then both $H_1(Y)$ and $H_1(Y_1(L))$ have odd order, so we can again take all three groups to be untwisted.)  As explained in \cite[Section~7.5]{scaduto}, the maps in this triangle are induced by 2-handle cobordisms with appropriate choices of bundles.

\begin{proposition} \label{prop:l-space-triangle}
If $S^3_n(K)$ is an instanton L-space for some integer $n \geq 1$, then $S^3_r(K)$ is an instanton L-space for all rational $r \geq n$.
\end{proposition}

\begin{proof}
Using the surgery exact triangle, it follows exactly as in \cite[Proposition~2.1]{osz-lens} that if $Y$, $Y_0(L)$, and $Y_1(L)$ are all rational homology spheres, where $Y$ and $Y_0(L)$ are instanton L-spaces and $|H_1(Y)|+|H_1(Y_0(L))| = |H_1(Y_1(L))|$, then $Y_1(L)$ is also an instanton L-space.  Since $S^3$ and $S^3_n(K)$ are instanton L-spaces, then, so is $S^3_p(K)$ for every integer $p \geq n$.

More generally, we can express any rational number $r>n$ as a continued fraction
\[ r = a_0 - \frac{1}{a_1-\frac{1}{\dots - \frac{1}{a_k}}}, \]
with $a_0 \geq n+1$ and $a_1,\dots,a_k \geq 2$.  We can then identify $S^3_r(K)$ as the result $Y_{[a_0,\dots,a_k]}$ of surgery on a framed link in $S^3$, in which $K$ has framing $a_0$ and a chain of $k$ unknots with framings $a_1,\dots,a_k$ is attached so that the first one is a meridian of $K$.  We induct on $k$, having already established the case $k=0$: when $k \geq 1$, we know by hypothesis that $Y_{[a_0,\dots,a_{k-1}]}$ is an instanton L-space, and so is $Y_{[a_0,\dots,a_{k-1},0]} = Y_{[a_0,\dots,a_{k-2}]}$ (which we interpret as $S^3$ if $k=1$), so arguing as above for the last unknot in the chain shows that $Y_{[a_1,\dots,a_{k-1},a_k]}$ is an instanton L-space for all integers $a_k \geq 0$ as well.
\end{proof}

We wish to show something similar to Proposition~\ref{prop:l-space-triangle} in the case where $n$ is rational, namely an instanton version of \cite[Proposition~7.4]{kmos}.  This will require the following analogue of \cite[Lemma~7.1]{kmos}, which in this setting is a particular case of the adjunction inequality for Donaldson invariants \cite[Theorem~1.1]{km-gauge2} (see also \cite{mmr}):

\begin{lemma} \label{lem:zero-sphere}
Let $W$ be a cobordism from $Y$ to $Y'$ which contains a homologically essential sphere $S$ of self-intersection zero and a closed surface $F$ such that $F\cdot S > 0$.  Then the induced map $I^\#(Y) \to I^\#(Y')$ is identically zero.
\end{lemma}

\begin{remark} \label{rem:zero-sphere-cobordism}
The reason for including $F$ in the statement of Lemma~\ref{lem:zero-sphere} is that Kronheimer and Mrowka do not prove the adjunction inequality $\Sigma \cdot \Sigma \leq 2g(\Sigma) - 2$ directly for essential spheres of self-intersection zero; rather, in \cite[Section~6(ii)]{km-gauge2} they reduce the problem to the case where the genus is odd, by finding such a surface in the homology class $[F] + d[S]$ which also violates the adjunction inequality for $d \gg 0$, and so their proof requires the surface $F$ to exist.  In \cite{km-gauge2} they work with closed 4-manifolds with $b_2^+ \geq 3$, so such $F$ are easy to find, but this is not automatic for cobordisms $W$ as in Lemma~\ref{lem:zero-sphere}.

In the cases we are interested in, we construct $W$ by attaching a 2-handle $H$ to $Y \times [0,1]$ along some rationally nullhomologous knot $L\times\{1\}$ and then attaching a 0-framed 2-handle $H'$ along the boundary of the cocore of $H$.  The 2-sphere $S$ is the union of the cocore of $H$ and the core of $H'$.  For the surface $F$, we take $n\geq 1$ such that $n[L] = 0$ in $H_1(Y)$, and then we glue $n$ parallel cores of the handle $H$ to a Seifert surface for the union of their boundaries in $Y \times \{1\}$.  It is clear that $F \cdot S = n$, as desired.
\end{remark}

We now prove the desired instanton version of \cite[Proposition~7.4]{kmos}; our proof is similar in spirit but somewhat trickier because we do not have anything analogous to the three variants of monopole Floer homology used there or the exact triangle relating them.

\begin{proposition} \label{prop:rational-l-space-surgery}
If $S^3_r(K)$ is an instanton L-space for some rational $r > 0$, and $m = \max(\lfloor r \rfloor, 1),$ then $S^3_m(K)$ is also an instanton L-space.
\end{proposition}

\begin{proof}
If $r$ is an integer then there is nothing to show, so assume it is not and let $n = \lfloor r \rfloor$.  We define a sequence $r_i = \frac{p_i}{q_i} \in \Q \cup \{\infty\}$ for $i=0,\dots,k$ by 
\begin{align*}
\frac{p_0}{q_0} &= \frac{1}{0}, & \frac{p_1}{q_1} &= \frac{n}{1}, & \frac{p_2}{q_2} &= \frac{n+1}{1} = \frac{p_0+p_1}{q_0+q_1}
\end{align*}
and then for $i \geq 3$ we define $r_i = \frac{p_i}{q_i}$ as follows: if $r_{i-1} = \frac{p_{j_{i-1}}+p_{i-2}}{q_{j_{i-1}}+q_{i-2}}$ for some $j_{i-1} \leq i-3$, then $r$ lies between $r_{i-1}$ and either $r_{i-2}$ or $r_{j_{i-1}}$, so we let $j_i$ be either $i-2$ or $j_{i-1}$ respectively and then set $r_i = \frac{p_{j_i} + p_{i-1}}{q_{j_i} + q_{i-1}}$.  (Note that $j_2 = 0$ and $j_3 = 1$.)  At each step we have 
\begin{align} \label{eq:farey-neighbors}
|p_{j_i}q_{i-1} - q_{j_i}p_{i-1}|&=1, & |p_iq_{i-1} - q_ip_{i-1}|&=1, & |p_iq_{j_i}-q_ip_{j_i}|&=1;
\end{align}
and eventually we reach $r_k = r$, at which point the sequence ends.  (For example, if $r=\frac{10}{7}$ then the sequence $\{r_i\}$ is $\frac{1}{0}, \frac{1}{1}, \frac{2}{1}, \frac{3}{2}, \frac{4}{3}, \frac{7}{5}, \frac{10}{7}$, and since $r_6 = \frac{10}{7} = \frac{p_3+p_5}{q_3+q_5}$ we have $j_6 = 3$.)  This works by properties of the Farey sequence of order $m$, in which all nonnegative rational numbers with denominator at most $m$ are listed in increasing order: given adjacent terms $\frac{a}{b} < \frac{c}{d}$ we always have $bc-ad=1$, and the sequence is built inductively from the Farey sequence of order $m-1$ by inserting $\frac{a+c}{b+d}$ between adjacent $\frac{a}{b} < \frac{c}{d}$ whenever $b+d=m$.

Letting $Y_{a/b} = S^3_{a/b}(K)$ for convenience, there is a surgery exact triangle relating $I^\#(Y_{r_i})$, $I^\#(Y_{r_{i-1}})$, and $I^\#(Y_{r_{j_i}})$ for each $i \geq 2$.  This follows from a standard argument as in the proof of Proposition~\ref{prop:l-space-triangle}, once we know that \eqref{eq:farey-neighbors} is satisfied.  Namely, using the continued fraction expansions of these slopes, we can attach a chain of unknots to $K$ and then perform integer-framed surgeries on $K$ and all but the last unknot in the chain, so that the $r_i$--, $r_{i-1}$--, and $r_{j_i}$-surgeries on $K$ result from surgeries on the last unknot with framings $\infty,k,k+1$ in some order for some integer $k$.  The claimed exact triangle is then a special case of \eqref{eq:surgery-exact-triangle}.

For each $i \geq 3$, we now have a commutative diagram of one of the two following forms, in which both triangles are exact and $f_i\circ g_i = g_i \circ f_i = 0$ (assuming in each case that the composite is not a map $I^\#(Y_0) \to I^\#(Y_0)$) by Lemma~\ref{lem:zero-sphere} and Remark~\ref{rem:zero-sphere-cobordism}:
\[ \xymatrix{
I^\#(Y') \ar[r] \ar@<-0.5ex>[dr]_{f_i} & I^\#(Y_{r_i}) \ar[d] & &
I^\#(Y') \ar[d]_{g_{i-1}} \ar@<0.5ex>[dr]^{g_i} & I^\#(Y_{r_i}) \ar[l] \\
I^\#(Y'') \ar[u]^{g_{i-1}} & I^\#(Y_{r_{i-1}}) \ar[l] \ar@<-0.5ex>[ul]_{g_i} & &
I^\#(Y'') \ar[r] & I^\#(Y_{r_{i-1}}) \ar@<0.5ex>[ul]^{f_i} \ar[u]
} \]
The pair $(Y',Y'')$ is either $(Y_{r_{i-2}}, Y_{r_{j_{i-1}}})$ or $(Y_{r_{j_{i-1}}}, Y_{r_{i-2}})$ depending on whether $j_i=i-2$ or $j_i = j_{i-1}$.  In the left diagram, we have $\img(g_i) \subset \ker(f_i) = \img(g_{i-1})$, by the equation $f_i \circ g_i = 0$ and exactness.  In the right diagram, we have $\ker(g_{i-1}) = \img(f_i) \subset \ker(g_i)$, by exactness and $g_i \circ f_i = 0$.   (Note that if one of $Y'$ or $Y_{r_{i-1}}$ is $Y_0$, we must be in the case on the left with $Y'=Y_0$, $Y_{r_{i-1}} = Y_{1/q}$, and $Y_{r_i} = Y_{1/(q+1)}$; then the map $f_i\circ g_i$ which must vanish corresponds to the cobordism $Y_{1/q} \to Y_0 \to Y_{1/q}$, so Remark~\ref{rem:zero-sphere-cobordism} still applies.)  Thus in either case we have $\rank(g_i) \leq \rank(g_{i-1})$.

Repeating this argument for $3 \leq i \leq k$, we have $\rank(g_k) \leq \rank(g_2)$ where $g_2$ appears in the exact triangle
\[ \dots \to I^\#(S^3) \xrightarrow{g_2} I^\#(Y_{n}) \xrightarrow{f_3} I^\#(Y_{n+1}) \to \dots. \]
Since $I^\#(S^3)$ has rank 1, we conclude that $g_k$ has rank either 0 or 1.  Now $g_k$ is a map from $I^\#(Y_{r_{k-1}})$ to $I^\#(Y_{r_{j_k}})$ or vice versa, and its mapping cone is quasi-isomorphic to $I^\#(Y_{r_k}) = I^\#(Y_r)$, so we have
\begin{align*}
\rank(I^\#(Y_{r_k})) &= \rank \ker(g_k) + \rank \coker(g_k) \\
&= \rank(I^\#(Y_{r_{k-1}})) + \rank(I^\#(Y_{r_{j_k}})) - 2\rank(g_k).
\end{align*}
For any $\frac{p}{q} \geq 0$, we can write $\rank(I^\#(Y_{p/q})) = p + 2e_{p/q}$ for some integer $e_{p/q} \geq 0$, since $I^\#(Y_{p/q})$ has Euler characteristic $|H_1(Y_{p/q})| = p$, and then (assuming $p\neq 0$) the condition $e_{p/q} = 0$ is equivalent to $Y_{p/q}$ being an instanton L-space.  Since $p_k = p_{k-1}+p_{j_k}$ and $e_{r_k}=0$, the above equation simplifies to
\[ 2e_{r_{k-1}} + 2e_{r_{j_k}} = 2e_{r_k} + 2\rank(g_k) = 2\rank(g_k) \leq 2 \]
and so either $e_{r_{k-1}}$ or $e_{r_{j_k}}$ must be zero as well.

Since $k \geq 3$ and the sequence $\{j_i\}$ is nondecreasing with $j_3 = 1$, we have shown that $e_{r_i} = 0$ for some $i$ satisfying $1 \leq i < k$.  It follows by induction on $k$ that either $e_{r_1}=0$ or $e_{r_2}=0$, where we recall that $r_1 = n$ and $r_2 = n+1$.  We now consider the surgery exact triangle
\[ \dots \to I^\#(S^3) \xrightarrow{g_2} I^\#(Y_n) \xrightarrow{f_3} I^\#(Y_{n+1}) \to \dots. \]
If $n=0$, then either $e_1=0$, or we see that $e_0=0$ implies $e_1=0$, so we are done.  If $n \geq 1$, then we are likewise done except for the possibility that $e_{n+1}=0$ but $e_n \neq 0$.  In this case, the same exact triangle implies that $\rank(g_2) = 1$, hence $g_2$ is injective and $f_3$ is surjective.  Then since $n\geq 1$, the map $g_3$ in the triangle
\[ \dots \to I^\#(Y_{n+1}) \xrightarrow{g_3} I^\#(Y_n) \to I^\#(Y_{(2n+1)/2}) \to \dots \]
satisfies $g_3 \circ f_3 = 0$, so $g_3$ must be zero and hence $\rank(g_i) = 0$ for all $i \geq 3$.  But then we see as above that $e_{r_i} = e_{r_{i-1}} + e_{r_{j_i}}$ for $3 \leq i \leq k$.  In particular, if $e_{r_i} = 0$ and $i \geq 3$ then $e_{r_{i-1}} = e_{r_{j_i}} = 0$; since $e_{r_k} = 0$, we eventually have $e_{r_3} = 0$ and so $e_{r_{j_3}} = 0$ as well.  But $r_{j_3} = r_1 = n$, so $Y_n$ must have been an instanton L-space after all.
\end{proof}

The conclusion of Proposition~\ref{prop:rational-l-space-surgery} is not entirely satisfying if $0 < r < 1$, but we will see that this never happens unless $K$ is the unknot.

\begin{proposition} \label{prop:l-space-genus-2}
If $K \subset S^3$ is a knot of genus $g>1$, then $S^3_r(K)$ is not an instanton L-space for $0 < r < 2$.
\end{proposition}

\begin{proof}
By Proposition~\ref{prop:rational-l-space-surgery} it suffices to show that $S^3_1(K)$ is not an instanton L-space.  We use the surgery exact triangle
\[ \dots \to I^\#(S^3) \to I^\#(S^3_0(K); \mu) \to I^\#(S^3_1(K)) \to \dots \]
where we now use twisted coefficients for $Y_0 = S^3_0(K)$, equipping it with the nontrivial $SO(3)$ bundle $P \to Y_0$ where $w_2(P) = PD(\mu)$ and $\mu$ is the image of a meridian of $K$ inside $Y_0$.  Since $I^\#(S^3) \cong \Z$, it will suffice to show that $\rank(I^\#(Y_0;\mu)) > 2$.  We will work with coefficients in $\C$, though this holds more generally by the universal coefficient theorem.

The bundle $P \to Y_0$ is nontrivial and admissible, since $w_2(P)$ has nonzero pairing with a closed surface $\Sigma$ of genus $g$ built by capping off a Seifert surface for $K$ inside $Y_0$.  There is an  associated relatively $\Z/8\Z$-graded Floer homology group, denoted $I_*(P)$ in \cite{scaduto} and $I_*(Y_0)_w$ in \cite{km-excision} where $w \to Y_0$ is a Hermitian line bundle with $c_1(w) = PD(\mu)$.  It has an operator $u=4\mu(\mathrm{pt})$ of degree 4, and Fr{\o}yshov \cite[Theorem~9]{froyshov} showed that $(u^2-64)^{n_g} = 0$ for some constant $n_g \geq 1$ depending only on $g$.  In particular $u$ is invertible, so $I_*(Y_0)_w \cong I_{*+4}(Y_0)_w$, and Scaduto \cite[Theorem~1.3]{scaduto} showed that
\[ I^\#(Y_0;\mu) \cong \ker\left(\left.u^2-64\right|_{\bigoplus_{j=0}^3 I_{*+j}(Y_0)_w}\right) \otimes H_*(S^3). \]
This kernel has half the dimension of $\ker(u^2-64)$, so as ungraded groups we have $I^\#(Y_0;\mu) \cong \ker(u^2-64)$ where $u$ acts on all of $I_*(Y_0)_w$.

If $g \geq 1$, then Kronheimer and Mrowka \cite[Theorem~7.21]{km-excision} showed that the degree-2 operator $\mu(\Sigma)$ on $I_*(Y_0)_w$ has a nontrivial generalized $(2g-2)$-eigenspace, and remarked that it is isomorphic to each of the generalized $i^r(2g-2)$-eigenspaces for $r=0,1,2,3$; these are all distinct eigenspaces since $g \geq 2$.  Now $u = 4\mu(\mathrm{pt})$ commutes with $\mu(\Sigma)$, so $u^2-64$ acts nilpotently on each eigenspace, and in particular $\ker(u^2-64)$ has dimension at least 1 when restricted to each of these four eigenspaces.  The kernel of $u^2-64$ on all of $I_*(Y_0)_w$ must therefore have dimension at least 4, as desired.
\end{proof}

The case where $K \subset S^3$ has genus 1 requires some additional work.  In what follows, we will repeatedly use results of Gordon \cite[Corollary~7.3]{gordon} on Dehn surgeries on cables.  We use $C_{p,q}(K)$ to denote the $(p,q)$-cable of $K$, which has longitudinal winding $q$.  We will also use the K\"unneth formula for $I^\#$, which is a special case of \cite[Corollary~5.9]{km-khovanov}: it says that if either $I^\#(Y)$ or $I^\#(Y')$ is torsion-free, then $I^\#(Y\#Y') \cong I^\#(Y) \otimes I^\#(Y')$.  Notably, this applies when $Y'$ is a lens space, since $I^\#(L(p,q)) \cong \Z^p$ by \cite[Corollary~1.2]{scaduto}.

\begin{lemma} \label{lem:l-space-one-half}
If $K \subset S^3$ is not the unknot, then $S^3_{1/2}(K)$ is not an instanton L-space.
\end{lemma}

\begin{proof}
If $K$ has genus $g \geq 2$ then this is implied by Proposition~\ref{prop:l-space-genus-2}, so assume that $g=1$ and that $S^3_{1/2}(K)$ is an L-space.  By assumption, Proposition~\ref{prop:rational-l-space-surgery} says that $S^3_1(K)$ is an L-space, so we argue exactly as in the proof of Proposition~\ref{prop:l-space-genus-2} to conclude that $I^\#(S^3_0(K); \mu) \cong \ker(u^2-64)$ where $u$ acts on $I_*(S^3_0(K))_w$.  In fact, we have $(u^2-64)^{n_1} = 0$, and Fr{\o}yshov remarks just before \cite[Theorem~9]{froyshov} that $n_1=n_2=1$, so $\ker(u^2-64) = I_*(S^3_0(K))_w$ and hence $I_*(S^3_0(K))_w$ has rank at most 2.  But $I_*(S^3_0(K))_w$ is nonzero since $K$ is nontrivial \cite[Corollary~7.22]{km-excision} (cf.\ \cite{km-p}), and so its rank is exactly 2.

Since $u$ is invertible of degree 4, the homology $I_*(S^3_0(K))_w$ is supported in two gradings which agree mod 4.  Using Floer's surgery exact triangle (see \cite{braam-donaldson}), it follows that the same is true of the instanton homology $HF(S^3_1(K))$, whose Euler characteristic is twice the Casson invariant $\lambda(S^3_1(K)) = \frac{1}{2}\Delta''_K(1)$, so we have $\Delta''_K(1) = \pm 2$.  But since $K$ has genus 1 and $\Delta_K(1)=1$, we can write $\Delta_K(t) = at-(2a-1)+at^{-1}$ for some integer $a$ (possibly zero), and then $\Delta''_K(1) = \pm 2$ implies that $a = \pm 1$.

Now let $C=C_{1,2}(K)$ denote the $(1,2)$-cable of $K$.  Gordon \cite{gordon} showed that
\[ S^3_2(C) \cong S^3_{1/2}(K) \# \mathbb{RP}^3, \]
and $I^\#(\mathbb{RP}^3) \cong \Z^2$, so if $S^3_{1/2}(K)$ is an L-space then $I^\#(S^3_2(C))$ has rank 2 by the K\"unneth formula.  Using the surgery exact triangle for $I^\#$ with the triads $(S^3,S^3_1(C),S^3_2(C))$ and $(S^3,S^3_0(C),S^3_1(C))$, we conclude that the twisted homology $I^\#(S^3_0(C);\mu)$ has rank at most 4.  Since $C$ has Seifert genus $2g=2$ \cite{shibuya}, and $(u^2-64)^{n_2} = 0$ with $n_2=1$, we once again have $I^\#(S^3_0(C);\mu) \cong I_*(S^3_0(C))_w$ as ungraded groups, so the latter has rank at most 4.

Using Floer's exact triangle again, we see that $HF(S^3_1(C))$ has rank at most 4, hence computing its Euler characteristic gives $|\Delta''_C(1)| \leq 4$.  On the other hand, we know that
\[ \Delta_C(t) = \Delta_K(t^2) = at^2 - (2a-1) + at^{-2}, \]
so $\Delta''_C(1) = 8a = \pm 8$ and we have a contradiction.
\end{proof}

We note one interesting consequence of the proof of Lemma~\ref{lem:l-space-one-half}: if $S^3_1(K)$ is an instanton L-space and $K$ is not the unknot, then $K$ has genus 1 and its Alexander polynomial is either $t-1+t^{-1}$ or $-t+3-t^{-1}$.

\begin{lemma} \label{lem:surgeries-below-1}
Let $K \subset S^3$ be a knot for which some $S^3_r(K)$ is an instanton L-space, where $0 < r < 1$.  If $m = \max(\lfloor\frac{r}{1-r}\rfloor, 1)$, then $S^3_s(K)$ is an instanton L-space for all $s$ such that $\frac{m}{m+1} \leq s \leq 1$.
\end{lemma}

\begin{proof}
Write $r = \frac{a}{b}$, where $a$ and $b$ are relatively prime positive integers.  By Proposition~\ref{prop:rational-l-space-surgery}, we know that $\tilde{Y} = S^3_1(K)$ is also an instanton L-space; let $\tilde{K}$ be the core of this surgery.  We observe that $\tilde{Y}$ and $\tilde{Y}_{a/(b-a)}(\tilde{K}) = S^3_{a/b}(K)$ are both instanton L-spaces, and $\frac{a}{b-a} = \frac{r}{1-r} > 0$, so another application of Proposition~\ref{prop:rational-l-space-surgery} says that $\tilde{Y}_m(\tilde{K})$ is an L-space, hence so is $\tilde{Y}_s(\tilde{K})$ for all rational $s \geq m$ by Proposition~\ref{prop:l-space-triangle}.  (Both propositions are only stated for $S^3$, but their proofs still work with $S^3$ replaced by any homology sphere L-space, such as $\tilde{Y}$.)  Now if $s = \frac{c}{d}$ with $c,d > 0$ satisfies $\frac{m}{m+1} \leq s < 1$, then we have $\frac{c}{d-c} \geq m$ and so it follows that $\tilde{Y}_{c/(d-c)}(\tilde{K}) = S^3_{c/d}(K) = S^3_s(K)$ is an L-space, as desired.
\end{proof}

\begin{proposition} \label{prop:no-l-spaces-below-1}
If $K$ is not the unknot and $0 < r < 1$, then $S^3_r(K)$ is not an instanton L-space.
\end{proposition}

\begin{proof}
According to Lemma~\ref{lem:surgeries-below-1}, if some $r$-surgery is an instanton L-space and $0 < r < 1$, then so is $\frac{n}{n+1}$-surgery for some integer $n \geq 1$.  It therefore suffices to show that $\frac{n}{n+1}$-surgery on $K$ does not produce an L-space for any integer $n \geq 1$.  The case $n=1$ is Lemma~\ref{lem:l-space-one-half}, so we will proceed by induction.

If $S^3_{n/(n+1)}(K)$ is not an L-space for some $n \geq 1$, then $I^\#(S^3_{n/(n+1)}(K))$ has rank at least $n+2$ since its Euler characteristic is $n$.  If $C = C_{n,n+1}(K)$, then
\[ S^3_{n(n+1)}(C) \cong S^3_{n/(n+1)}(K) \# L(n+1,n) \]
by \cite{gordon}, and $I^\#(L(n+1,n)) \cong \Z^{n+1}$, so the K\"unneth formula says that $I^\#(S^3_{n^2+n}(C))$ has rank at least $(n+2)(n+1)$.  The surgery exact triangle 
\[ \dots \to I^\#(S^3) \to I^\#(S^3_{n^2+n}(C)) \to I^\#(S^3_{n^2+n+1}(C)) \to \dots \]
then implies that $I^\#(S^3_{n^2+n+1}(C))$ has rank at least $n^2+3n+1$, and $n \geq 1$, so $S^3_{n^2+n+1}(C)$ is not an L-space either.  But then we also know from \cite{gordon} that
\[ S^3_{n(n+1)+1}(C) \cong S^3_{(n(n+1)+1)/(n+1)^2}(K), \]
so $\frac{n^2+n+1}{(n+1)^2}$-surgery on $K$ does not produce an L-space.  We note that
\[ \frac{n^2+n+1}{(n+1)^2} = 1 - \frac{n}{(n+1)^2} > 1 - \frac{1}{n+2} = \frac{n+1}{n+2}, \]
and Lemma~\ref{lem:surgeries-below-1} says that if $\frac{n+1}{n+2}$-surgery on $K$ produces an instanton L-space then so does $s$-surgery whenever $\frac{n+1}{n+2} \leq s \leq 1$, so it follows that $S^3_{(n+1)/(n+2)}(K)$ cannot be an L-space either and we are done.
\end{proof}

We finish by combining Propositions~\ref{prop:l-space-triangle}, \ref{prop:rational-l-space-surgery}, \ref{prop:l-space-genus-2}, and \ref{prop:no-l-spaces-below-1} into a single result.

\begin{theorem} \label{thm:l-space-slopes}
Suppose that $K \subset S^3$ is not the unknot and that $S^3_r(K)$ is an instanton L-space for some rational $r > 0$.  Then $r \geq 1$ and $S^3_s(K)$ is an instanton L-space for every rational $s \geq \lfloor r \rfloor$.  Moreover, if $r < 2$ then $K$ has genus 1 and the same Alexander polynomial as either the trefoil or the figure eight.
\end{theorem}

We remark that if $K$ is either the left-handed trefoil or the figure eight, then $S^3_1(K)$ is not actually an instanton L-space: in each case it is a Seifert fibered homology sphere other than the Poincar\'e sphere, so this will follow from Corollary~\ref{cor:sfs-l-space}.

%
%
%
%

\section{Examples} \label{sec:examples}

In this section, we use Theorem~\ref{thm:main} to deduce the existence of nontrivial $SU(2)$ representations for integer homology spheres in a variety of examples, and similarly we use Corollaries~\ref{cor:rank-i-sharp} and \ref{cor:cycfinitesurg} to certify that many rational homology spheres have irreducible $SU(2)$ representations.


\subsection{Seifert fibered homology spheres}

In this subsection, we use Theorem~\ref{thm:main} to give a quick proof of the following, originally due to Fintushel-Stern \cite{fs-seifert}.
\begin{theorem}
\label{thm:sfs-nontrivial}
If $Y$ is a nontrivial Seifert fibered homology sphere, then there is a nontrivial representation $\pi_1(Y) \to SU(2)$.
\end{theorem}

A Seifert fibered integer homology sphere has base $S^2$, and can be described as
\[ Y = M(e; r_1,r_2,\dots,r_k), \]
where $e$ is an integer and $r_1 \leq r_2 \leq \dots \leq r_k$ all lie in $(0,1) \cap \Q$.  This manifold has a surgery diagram consisting of an $e$-framed unknot and $k$ of its meridians, each of which has surgery coefficient $-\frac{1}{r_i} < -1$.  This is illustrated on the left side of Figure~\ref{fig:sfs-example} for the Brieskorn sphere \[M(-2;\frac{1}{2},\frac{2}{3},\frac{9}{11}) = \Sigma(2,3,11).\] 
It is not hard to verify that \[-Y = M(-e-k; 1-r_k,\dots,1-r_1)\] in general.

\begin{figure}
\labellist
\small
\pinlabel $-2$ at 260 191
\pinlabel $-2$ at 290 157
\pinlabel $-\frac{3}{2}$ at 315 152
\pinlabel $-\frac{11}{9}$ at 340 155
\endlabellist
\centering
\includegraphics{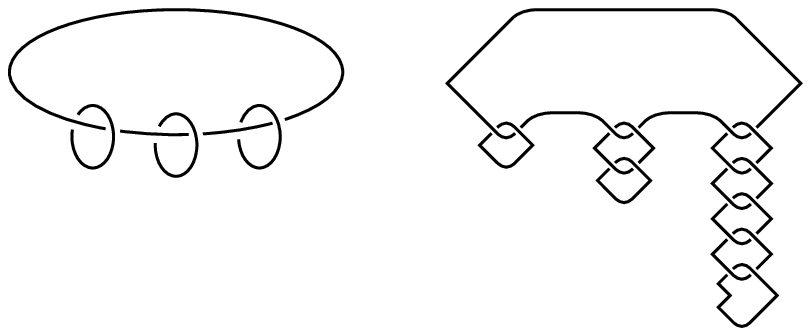}
\caption{A surgery diagram for the Seifert fibered homology sphere $Y=M(-2;\frac{1}{2},\frac{2}{3},\frac{9}{11})$, and a Legendrian link on which Legendrian surgery produces $Y$.  
}
\label{fig:sfs-example}
\end{figure}

\begin{proposition} \label{prop:sfs-fillings}
Let $Y$ be a Seifert fibered integer homology sphere.  If $Y$ is not $S^3$, then some contact structure on either $Y$ or $-Y$ admits a Stein filling $(W,J)$ with $b_2(W) > 0$.
\end{proposition}

\begin{proof}
Let $Y= M(e; r_1,\dots,r_k)$.  We write
\[ -\frac{1}{r_i} = a_{i,1} - \frac{1}{a_{i,2} - \frac{1}{\ldots - \frac{1}{a_{i,n_i}}}} \]
for some integers $a_{i,j} \leq -2$, and then replace each $-\frac{1}{r_i}$-framed meridian with a chain of $n_i$ unknots with framings $a_{i,j}$.  If $e \leq -2$ as well then the corresponding star-shaped diagram can be realized as a Legendrian link where every component $K$ has framing $tb(K)-1$, as shown on the right side of Figure~\ref{fig:sfs-example}, so attaching Weinstein 2-handles to $S^3 = \partial B^4$ along this link produces a Stein domain $(W,J)$ bounded by $Y$.  Since there are no 1-handles, we have \[b_2(W) = 1 + \sum_i n_i > 0\] and we are done.  On the other hand, if $e \geq -1$ then we note that $k \geq 3$ since $Y$ is not a lens space, and so $-e-k \leq -2$.  Applying the same argument to \[-Y = M(-e-k;1-r_k,\dots,1-r_1)\] then produces a Stein domain with boundary $-Y$ and $b_2$ positive, as desired.
\end{proof}

\begin{proof}[Proof of Theorem \ref{thm:sfs-nontrivial}]
Observe that Theorem \ref{thm:sfs-nontrivial} follows immediately from Theorem \ref{thm:main} together with Proposition \ref{prop:sfs-fillings}.
\end{proof}

Suppose $Y$ is a homology sphere as in the proof of Proposition \ref{prop:sfs-fillings} above, with $e<-2$ or some $a_{i,j} < -2$.  Then $Y$ is the boundary of a Stein domain $(W,J)$ given by Weinstein 2-handle attachment along a Legendrian link in which one of the components has positive rotation number. But  in this case,  we have that $c_1(J) \neq 0$ by Theorem \ref{thm:gompf}. It then follows, as noted in Remark~\ref{rem:homology-sphere-l-space}, that $Y$ is not an instanton L-space.  Now, we can determine when $Y$ is a homology sphere using the formula
\[ |H_1(Y)| = p_1p_2\dots p_k\left|e + \sum_{i=1}^k \frac{q_i}{p_i}\right| = p_1p_2\dots p_k\left|e+\sum_{i=1}^k r_i\right|, \]
where $-\frac{1}{r_i} = -\frac{p_i}{q_i}$ and $p_i > q_i \geq 1$ are relatively prime integers.  And it is not hard to verify that the only such nontrivial homology spheres for which the above procedure does not give a Stein filling of either $Y$ or $-Y$ with $c_1 \neq 0$ are $M(-2;\frac{1}{2},\frac{2}{3},\frac{4}{5})$ and $M(-2;\frac{1}{2},\frac{2}{3},\frac{6}{7})$.  These are $\Sigma(2,3,5)$ and $\Sigma(2,3,7)$ up to orientation reversal, but $-\Sigma(2,3,7)$ also results from Legendrian surgery on a right-handed trefoil with $tb=0$ and rotation number $\pm 1$.  This and the fact \cite{scaduto} that $\Sigma(2,3,5)$ is an instanton L-space then allows us to conclude the following, which may be of independent interest. 

\begin{corollary} \label{cor:sfs-l-space}
Let $Y$ be a nontrivial Seifert fibered integer homology sphere.  Then $Y$ is an instanton L-space if and only if $Y = \pm \Sigma(2,3,5)$. \qed
\end{corollary}

In particular, this agrees with the classification of Heegaard Floer L-spaces among nontrivial Seifert fibered integer homology spheres.


\subsection{Hyperbolic manifolds which are not branched double covers}

In this subsection we construct infinitely many examples for which we can use Stein fillings (via Theorem \ref{thm:main}), but not other available results, to show the existence of nontrivial $SU(2)$ representations. In particular, we prove Theorem \ref{thm:intro-hyperbolic-examples}, restated below for convenience.

{
\renewcommand{\thetheorem}{\ref{thm:intro-hyperbolic-examples}}
\begin{theorem} 
There are infinitely many hyperbolic integer homology spheres $Y_n$ such that
\begin{itemize}
\item $Y_n$ has Casson invariant zero;
\item $Y_n$ is not a branched double cover of a knot in $S^3$;
\item $Y_n$ bounds a negative definite Stein manifold which is not a homology ball.
\end{itemize}
The last property implies there is a nontrivial homomorphism $\pi_1(Y_n) \to SU(2)$ by Theorem \ref{thm:main}.\end{theorem}
\addtocounter{theorem}{-1}
}

Figure~\ref{fig:link-821-11a20} shows a Legendrian link $L$ whose components are knots of type $\overline{8_{21}}$ and $\overline{11a_{20}}$.  Using SnapPy \cite{SnapPy} within Sage \cite{sage}, we can verify that $L$ is a hyperbolic link:
\begin{verbatim}
sage: import snappy
sage: K0 = '(70,-58,-44,22,24,16,36,20,10,14,18,8,6,-68,-48,64,56,42,12)'
sage: K1 = '(-34,52,-38,26,-60,-32,40,62,-50,-2,28,66,54,30,46,4)'
sage: M = snappy.Manifold('DT:[%s,%s]'%(K0,K1))
sage: V = M.verify_hyperbolicity(bits_prec=100); V[0]
True
\end{verbatim}

\begin{figure}
\labellist
\small
\pinlabel $-\frac{1}{m}$ at 250 453
\pinlabel $-\frac{1}{n}$ at 240 328
\endlabellist
\centering
\includegraphics[scale=0.75]{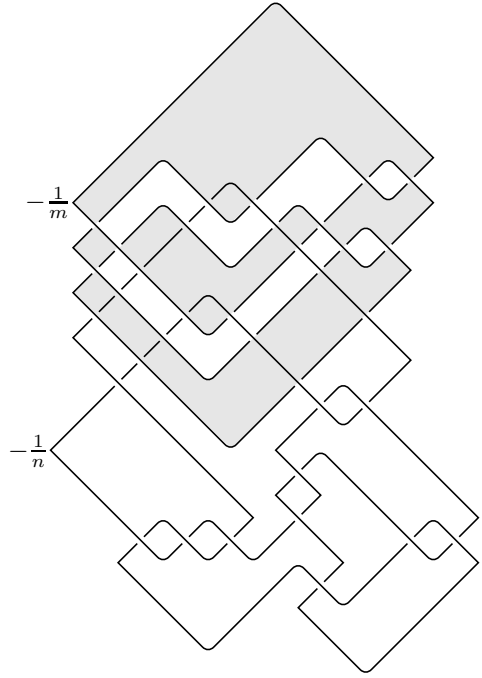}
\caption{A hyperbolic two-component link $L \subset S^3$, with a Seifert surface for the $\overline{8_{21}}$ component shaded to make the components easier to distinguish.  The closed $3$-manifold $Y(m,n)$ is constructed by Dehn surgery on $L$ with the indicated framings.}
\label{fig:link-821-11a20}
\end{figure}

Since the components of $L$ have linking number zero, we can perform $-\frac{1}{m}$-surgery on the $\overline{8_{21}}$ component and $-\frac{1}{n}$-surgery on the $\overline{11a_{20}}$ component to get an integer homology sphere $Y(m,n)$.  The Gromov-Thurston $2\pi$-theorem \cite{bleiler-hodgson} says that there is some finite list of slopes for each component such that any Dehn surgery on $L$ which avoids these slopes is hyperbolic, so there is an integer $N \geq 1$ such that $Y(m,n)$ is hyperbolic whenever $m,n \geq N$.

\begin{proposition} \label{prop:hyperbolic-casson}
The manifold $Y(m,n)$ has Casson invariant zero for all $m,n \geq 1$.
\end{proposition}

\begin{proof}
Since the components $K_1 = \overline{8_{21}}$ and $K_2 = \overline{11a_{20}}$ of $L$ have linking number zero, a formula of Hoste \cite{hoste} computes the Casson invariant as
\[ \lambda(Y(m,n)) = -m\varphi_1(K_1) - n\varphi_1(K_2) + mn\varphi_1(L), \]
where $\varphi_1$ of a $k$-component link is the $z^{k+1}$-coefficient of its Conway polynomial $\nabla(z)$.  The two components of the link $L$ have Conway polynomials
\begin{align*}
\nabla_{\overline{8_{21}}}(z) &= 1-z^4, &
\nabla_{\overline{11a_{20}}}(z) &= 1 - 5z^4 - 3z^6
\end{align*}
and so $\varphi_1(K_1) = \varphi_1(K_2) = 0$.  We must show that $\varphi_1(L) = 0$ as well.

In order to determine $\varphi_1(L)$, we recall that $\nabla_L(z)$ is defined by the relations
\begin{align*}
\nabla(\skeinpositive) &= \nabla(\skeinnegative) + z\nabla(\skeinresolve), &
\nabla(\unknot) &= 1.
\end{align*}
The skein relation implies that $\nabla(z)$ vanishes for split links, and hence it also implies that
\begin{align*}
\nabla(\positivemeridian) &= z\nabla(\upstrand), &
\nabla(\negativemeridian) &= -z\nabla(\upstrand)
\end{align*}
and that $\nabla_{K\#K'}(z) = \nabla_K(z)\nabla_{K'}(z)$.  As observed in \cite{hoste}, if $L$ has $k$ components then \[\nabla_L(z) = z^{k-1}(a_0+a_1z^2+\dots+a_mz^{2m}),\] where the $a_i$ are integers, and we have $\varphi_1(L) = a_1$.
\begin{figure}
\labellist
\small
\pinlabel $L$ at 105 18
\pinlabel $\to$ at 220 181
\pinlabel $=$ at 342 181
\pinlabel $L_1$ at 377 150
\pinlabel $\nearrow$ at 410 191
\pinlabel $\searrow$ at 410 171
\pinlabel $\leadsto$ at 481 205
\pinlabel $\nearrow$ at 481 183
\pinlabel $\to$ at 481 157
\pinlabel $\overline{11a_{20}}$ at 517 237
\pinlabel $3_1\#\overline{5_2^\circlearrowright}$ at 517 126
\pinlabel $\to$ at 220 65
\pinlabel $L_2$ at 282 10
\pinlabel $L_3$ at 387 10
\pinlabel $L_4$ at 491 10
\tiny
\pinlabel $-z$ at 481 214
\endlabellist
\centering
\includegraphics[scale=0.75]{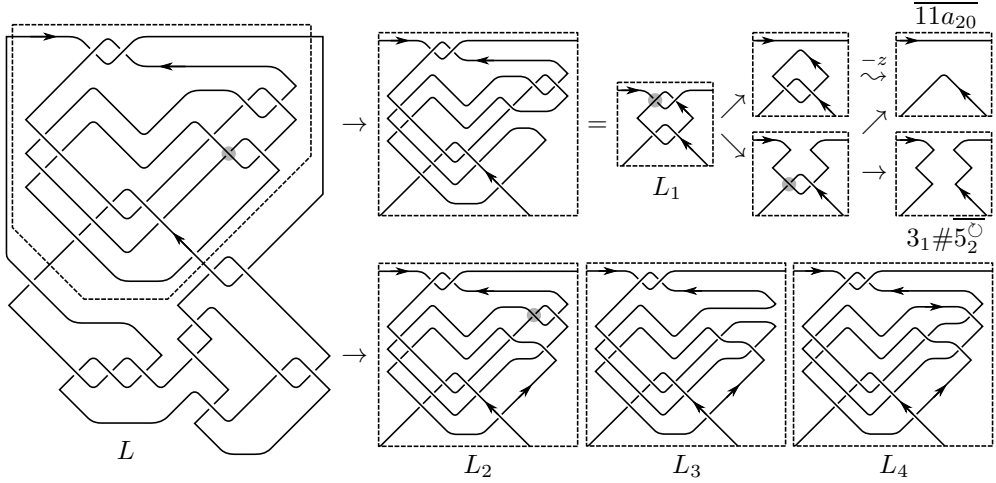}
\caption{We compute $\nabla_L(z)$ by applying the skein relation to $L$ at the indicated crossings.  Here the ``$\stackrel{-z}{\leadsto}$'' means that the Conway polynomial of the target is $-z$ times that of the source, and $\overline{5_2^\circlearrowright}$ denotes a $\overline{5_2}$ knot linked with a negatively oriented meridian.}
\label{fig:link-skein}
\end{figure}
Now we apply the skein relation to $L$ as indicated in Figure~\ref{fig:link-skein} to get $\nabla_L(z) = \nabla_{L_1}(z) + z\nabla_{L_2}(z)$, where
\begin{align*}
\nabla_{L_1}(z) &= \big(-z\nabla_{\overline{11a_{20}}}(z)\big) + z\big(\nabla_{\overline{11a_{20}}}(z) - z\nabla_{3_1\#\overline{5_2^\circlearrowright}}(z)\big) \\
&= z^3 \nabla_{3_1}(z)\nabla_{\overline{5_2}}(z) \\
\nabla_{L_2}(z) &= \nabla_{L_3}(z) + z\nabla_{L_4}(z) \\
&= -z^2 \nabla_{\overline{11a_{20}}}(z) + z\nabla_{L_4}(z);
\end{align*}
the last equation holds because the $L_3$ tangle is isotopic to the $\overline{11a_{20}}$ tangle with one meridian added around each strand.  Combining these, we have
\begin{align*}
\nabla_L(z) &= \big(z^3(1+z^2)(1+2z^2)\big) + z\big(-z^2(1-5z^4-3z^6) + z\nabla_{L_4}(z)\big) \\
&= 3z^5 + 7z^7 + 3z^9 + z^2\nabla_{L_4}(z).
\end{align*} 
But $L_4$ has four components, so $\nabla_{L_4}(z)$ is a multiple of $z^3$ and hence $\nabla_L(z)$ has $z^3$-coefficient equal to 0.  In other words, $\varphi_1(L) = 0$, and so $\lambda(Y(m,n))=0$ as desired.
\end{proof}

\begin{proposition} \label{prop:hyperbolic-stein}
For each $m,n \geq 1$, the manifold $Y(m,n)$ bounds a negative definite Stein manifold $(W,J)$ such that $c_1(J)\neq 0$.  In particular, we have $b_2(W) > 0$, and $Y(m,n)$ is not an instanton L-space.
\end{proposition}

\begin{proof}
The $\overline{8_{21}}$ and $\overline{11a_{20}}$ components of $L$ both have Thurston-Bennequin number $1$, so we may stabilize them each once so that they have Thurston-Bennequin number $0$ and nonzero rotation number.  If we attach a chain of unknots with Thurston-Bennequin number $-1$ to each component, of lengths $m-1$ and $n-1$ respectively, as shown in Figure~\ref{fig:surgery-821-11a20}, then Legendrian surgery along the resulting link $L'$ produces a Stein manifold $(W,J)$ whose boundary is topologically a $-1$-surgery along each component of $K$ and $-2$-surgery along each unknot.  Removing the chains by a series of slam dunks shows that $\partial W = Y(m,n)$, and we have $c_1(J) \neq 0$ by Theorem \ref{thm:gompf} since the components of $L \subset L'$ had nonzero rotation number.

\begin{figure}
\labellist
\small
\pinlabel $-1$ at 100 280
\pinlabel $-1$ at 111 57
\tiny
\pinlabel $-2$ at 40 234
\pinlabel $-2$ at 57 234
\pinlabel $-2$ at 74 234
\pinlabel $-2$ at 91 234
\pinlabel $-2$ at 47 132
\pinlabel $-2$ at 64 132
\pinlabel $-2$ at 81 132
\endlabellist
\centering
\includegraphics[scale=0.6]{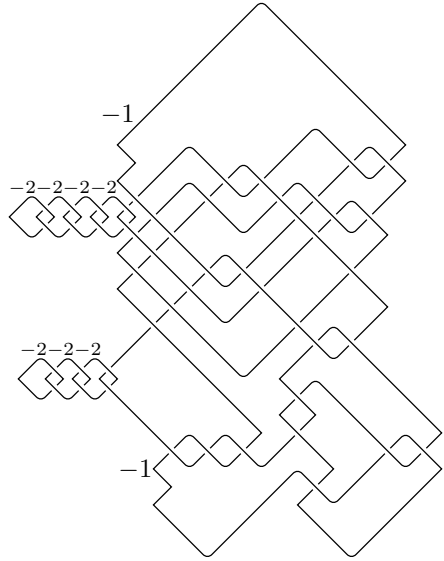}
\caption{A Legendrian link $L'$ on which Legendrian surgery produces $Y(5,4)$.  The surgery coefficient of each component is given with respect to the Seifert framing, and is one less than the Thurston-Bennequin invariant.}
\label{fig:surgery-821-11a20}
\end{figure}

The intersection form on $H_2(W)$ is specified by the linking matrix of $L'$.  Since the components of $L$ have linking number zero, this matrix has block form
$\left(\begin{smallmatrix} B_m & 0 \\ 0 & B_n \end{smallmatrix}\right)$, where $B_k$ is the $k\times k$ tridiagonal matrix
\[ \left(\begin{array}{rrrcr} -1 & 1 & 0 & \dots & 0 \\ 1 & -2 & 1 &  \dots & 0 \\ 0 & 1 & -2 & \dots & 0 \\ \vdots & \vdots & \vdots & \ddots & \vdots \\ 0 & 0 & 0 & \dots & -2 \end{array}\right) \]
whose corresponding quadratic form $Q(x) = \langle x, B_k x\rangle$ is equal to
\[ -(x_1 - x_2)^2 - (x_2 - x_3)^2 - \ldots - (x_{k-1}^2 - x_k)^2 - x_k^2. \]
It follows that each $B_k$ is negative definite, so $W$ is negative definite as well.
\end{proof}

\begin{proposition} \label{prop:hyperbolic-not-branched}
If $n$ is sufficiently large then $Y(1,n)$ is hyperbolic and is not a branched double cover of any link in $S^3$.
\end{proposition}

\begin{proof}
We continue the above Sage/SnapPy session to certify that after performing $-1$-surgery on the $\overline{8_{21}}$ component of $L$, the complement $M$ of the $\overline{11a_{20}}$ component remains hyperbolic and has no nontrivial isometries:
\begin{verbatim}
sage: M.dehn_fill((1,-1),1)
sage: V = M.verify_hyperbolicity(bits_prec=128); V[0]
True
sage: K = M.canonical_retriangulation(verified=True)
sage: len(K.isomorphisms_to(K)) == 1
True
\end{verbatim}
Readers interested in carrying out these computations should be warned that computing the canonical retriangulation may take about a minute.

The ``exceptional symmetry theorem'' of \cite[Theorem~5.2]{auckly} says that if we also perform Dehn surgery on the $\overline{11a_{20}}$ component, then all but finitely many slopes will produce closed hyperbolic $3$-manifolds with isometry group a subgroup of $\operatorname{Isom}(M)$, which is trivial.  It follows as in \cite{auckly} that only finitely many of the hyperbolic homology spheres $Y(1,n)$ are branched double covers of links in $S^3$, or indeed in any closed 3-manifold.
\end{proof}

\begin{remark}
Analogous computations prove that the result of Dehn filling along the $\overline{8_{21}}$ component of $S^3 \ssm L$ with slope $-\frac{1}{m}$ is an asymmetric hyperbolic manifold for $1 \leq m \leq 25$, so that $Y(m,n)$ is not a branched double cover for any such $m$ if $n$ is sufficiently large.
\end{remark}

\begin{proof}[Proof of Theorem \ref{thm:intro-hyperbolic-examples}]
According to Propositions~\ref{prop:hyperbolic-casson}, \ref{prop:hyperbolic-stein} and \ref{prop:hyperbolic-not-branched}, we can take the manifolds $Y_n$ to be the manifolds $Y(1,n)$ for $n$ sufficiently large. 
\end{proof}


\subsection{Surgeries on knots} In this subsection, we  study the existence of irreducible $SU(2)$ representations for rational homology spheres obtained as surgeries on knots in $S^3$. In particular, we prove Theorem \ref{thm:intro-irreps-sl-positive}, restated in a  slightly strengthened form below.

{
\renewcommand{\thetheorem}{\ref{thm:intro-irreps-sl-positive}}

\begin{theorem} 
Suppose $\maxsl (\overline K)\geq 0$ and fix a rational number $r = p/q > 0$. Then $S^3_r(K)$ is not an instanton L-space. Moreover, there is an irreducible representation \[\pi_1(S^3_{r}(K)) \to SU(2)\] if no zero of $\Delta_K(t^2)$ is a $p$th  root of unity, where $\Delta_K(t)$ is  the Alexander polynomial of $K$.
\end{theorem}
\addtocounter{theorem}{-1}
}

Recall that $\maxsl(K)$ denotes the maximal self-linking number over all transverse representatives of $K$ in the standard contact $S^3$, or, equivalently, the maximum value of $tb(\Lambda)-r(\Lambda)$ over all Legendrian representatives $\Lambda$ of $K$.  As mentioned in the introduction, the knots $K$ for which $\maxsl(K)\geq 0$ includes all nontrivial strongly quasipositive knots; these even have maximal Thurston-Bennequin number satisfying $\maxtb(K) \geq 0$ \cite{rudolph}.

\begin{proof}[Proof of Theorem \ref{thm:intro-irreps-sl-positive}]
Suppose $\maxsl (\overline K)\geq 0$ and $r = p/q > 0$. We will first show that $S^3_r(K)$ is not an instanton L-space. By Theorem~\ref{thm:l-space-slopes}, if  $S^3_{r}(K)$ is an instanton L-space then so is $S^3_n(K)$ for all integers $n \geq \lfloor r \rfloor$. So,  it will suffice to prove that \[\rank(I^\#(S^3_{n}(K))) \geq n+1\] for all sufficiently large integers $n > 0$.

Take a Legendrian representative $\Lambda$ of $\overline K$ whose transverse pushoff achieves the maximum self-linking number, so that $tb(\Lambda) - r(\Lambda) = \maxsl(\overline K)$, and hence, by assumption, $tb(\Lambda) \geq r(\Lambda)$.  Note that reversing orientation changes the sign of $r(\Lambda)$, so we must have that $r(\Lambda) \leq 0$ or else a transverse pushoff of $-\Lambda$ would have strictly larger self-linking number.  Also, since self-linking number is always odd, we actually have $\maxsl(\overline K) \geq 1$.

Now, if we take $n \geq \max(1-tb(\Lambda),0)$ and stabilize $\Lambda$ exactly $n+tb(\Lambda)-1$ times, with $k$ positive stabilizations and $n-k+tb(\Lambda)-1$ negative stabilizations, we get a Legendrian representative of $k$ with
\[ (tb,r) = (-n+1, -\maxsl(\overline K) - n + 1 + 2k) \]
for any integer $k = 0,\dots,n+tb(\Lambda)-1$.  Thus, we can find such a representative with rotation number equal to any integer of the same parity as $-\maxsl(\overline K)-n+1$ between $-\maxsl(\overline K)-n+1$ and
\[ -\maxsl(\overline K)+n+2tb(\Lambda)-1 = n + (tb(\Lambda) + r(\Lambda) - 1). \]
Note that the lower bound is negative while the upper bound is positive if $n$ is sufficiently large, so in this case we can also reverse the orientations of these representatives to achieve every integer of this parity between $-\maxsl(\overline K)-n+1$ and $\maxsl(\overline K)+n-1$.  We conclude that for all sufficiently large integers $n$, there are Legendrian representatives of $\overline K$ with $tb = -n+1$ and at least $\maxsl(\overline K)+n \geq n+1$ different rotation numbers.

The traces of Legendrian surgeries on each of these $n+1$ representatives of $\overline K$, obtained topologically by attaching a $-n$-framed handle to $B^4$ along $\overline K$, have Stein structures with distinct Chern classes since their rotation numbers are all distinct, by Theorem \ref{thm:gompf}.  Corollary \ref{cor:rank-i-sharp} therefore implies that \[\rank I^\#(S^3_{n}(K))=\rank I^\#(S^3_{-n}(\overline K)) \geq n+1\] for all sufficiently large integers $n$, as desired. Thus, $S^3_r(K)$ is not an instanton L-space.  We may now appeal to Corollary \ref{cor:cycfinitesurg} to conclude that there exists an irreducible $SU(2)$ representation of $\pi_1(S^3_r(K))$, as desired.
\end{proof}

We saw in the introduction that $K=5_2$ satisfies the hypotheses of Theorem \ref{thm:intro-irreps-sl-positive}.

One can sometimes do slightly better for knots whose mirrors have sufficiently large $\maxtb$, as illustrated in the proof of the proposition below.

\begin{proposition} \label{prop:positive-knots}
Suppose $\overline K$ is a positive knot of genus $g \geq 1$ and fix a nonzero rational number $r=p/q$ with $r>-g$. Then $S^3_r(K)$ is not an instanton L-space. Moreover, there exists an irreducible representation \[\pi_1(S^3_r(K)) \to SU(2)\] if no zero of $\Delta_K(t^2)$ is a $p$th  root of unity.
\end{proposition}

\begin{remark}
Proposition \ref{prop:positive-knots} is sharp when $K$ is the left-handed trefoil, on which $-1$-surgery produces an instanton L-space.
\end{remark}

\begin{proof}[Proof of Proposition \ref{prop:positive-knots}]
Tanaka \cite{tanaka} showed that $\maxtb(\overline K) = 2g-1$ since $\overline K$ is positive.  (To be precise, he showed that $\maxtb(\overline K) = 2g'-1$, where $g'$ is the genus of a surface obtained by applying Seifert's algorithm to a positive diagram of $\overline K$. The Bennequin inequality \[\maxtb(\overline K) \leq 2g-1\] then implies that $g'=g$.)  Suppose we stabilize a $tb$-maximizing Legendrian representative of $\overline K$ $g-1$ times, with $k$ of these positive and $g-1-k$ negative.  As $k$ varies between 0 and $g-1$, we obtain Legendrian representatives  with $tb=g$ and $g$ different rotation numbers.  The traces of Legendrian surgeries on each of these  $g$ representatives of $\overline K$, obtained topologically by attaching a $(g-1)$-framed handle to $B^4$ along $\overline K$, have Stein structures with distinct Chern classes since their rotation numbers are all distinct, by Theorem \ref{thm:gompf}. It then follows from Corollary~\ref{cor:rank-i-sharp} as before that \[\rank I^\#(S^3_{g-1}(\overline K)) \geq g.\] Theorem~\ref{thm:l-space-slopes} then implies that \[S^3_{r}(K) = -S^3_{-r}(\overline K)\] is not an instanton L-space in the case $-g<r<0$. And an application of Corollary~\ref{cor:cycfinitesurg} then tells us that there exists an irreducible $SU(2)$ representation of $\pi_1(S^3_r(K))$ in this case.
The $r >0$ case follows already from Theorem \ref{thm:intro-irreps-sl-positive}, since $\maxtb(\overline K)\geq 0$ implies that $\maxsl(\overline K)\geq 0$.
\end{proof}

For example, given positive integers $k,p,q$ with $p,q$ odd, one can see from its standard diagram that the pretzel knot $P(-2k,p,q)$ is positive with Seifert genus $g = \frac{p+q}{2}$.  It is known that several negative surgeries on \[K = \overline{P(-2,3,7)}\] do not admit irreducible $SU(2)$ representations. For example, the $-18$- and $-19$-surgeries on $K$ do not because they are lens spaces \cite{fs-lens}. In fact, neither does the $-\frac{37}{2}$-surgery  even though its fundamental group is noncyclic (attributed to Dunfield in \cite{km-dehn}). However, since the Alexander polynomial of $K$ has no roots of unity among its zeroes, we may conclude (appealing to \cite{km-dehn} in the case $r=0$) that for all rational $r>-5$ there exist irreducible representations $\pi_1(S^3_r(K)) \to SU(2)$.


%
%
%
%

\newcommand{\appendixtitle}{Appendix: Some remarks on \cite{sivek-donaldson}}
\bgroup
\newcommand{\nocontentsline}[3]{}
\let\addcontentsline=\nocontentsline
\section*{\appendixtitle} \label{sec:appendix}
\egroup
\addcontentsline{toc}{section}{\appendixtitle}

The proofs of several of the main results in \cite{sivek-donaldson} made use of \cite[Theorem~1.4]{stipsicz}, which asserts that a homologically essential fiber of a Lefschetz fibration $X \to S^2$ represents a primitive class in $H_2(X)$, but Baykur \cite[Appendix]{baykur} pointed out that the result is not proved correctly in the case where $\pi_1(X) \neq 1$. It still holds when $\pi_1(X)=1$ or when $X\to S^2$ has a section.  This does not affect any of the results of \cite{sivek-donaldson}, since its use in the more general cases can be avoided.  We indicate how to do so here, since we invoke some of these results in the proof of Proposition~\ref{prop:modifyZ}. All numbered results and sections in the following discussion refer to \cite{sivek-donaldson}.

In the statement of Lemma~3.1, the class $[\Sigma]$ need not be primitive, but it is nontorsion since \[K_X\cdot\Sigma = 2g(\Sigma)-2 \neq 0.\]  As a result, in Proposition~3.2, we can only assert that $c_0 \in \Q$ rather than $c_0 \in \Z$; and otherwise the rest of Section~3, which establishes Theorem~1.3 among other things, remains the same.

In proving Theorem~4.1, the initial choice of $Z \to S^2$ should have vanishing cycles which generate not just $H_1(\Sigma)$ but $\pi_1(\Sigma)$, so that $\pi_1(Z) = 1$ and the fiber of $Z$ is thus homologically primitive.  The application of Theorem~5.2 then continues to work verbatim (notably, the class $w_Z$ exists), so the argument of Section~5 still proves Theorem~1.2 in the case when $g\geq 8$.  Then the proof of Theorem~1.1 still works as written (and likewise for Corollary~8.6), since the Lefschetz fibrations $\tilde{X}\to S^2$ used in its proof admit sections and hence a class $w$ dual to the fiber $\tilde{\Sigma}$ exists; and at the end of Section~7, Theorem~1.1 can still be used to complete the proof of Theorem~1.2 as long as we ensure once again that $\pi_1(Z)=1$.

\bibliographystyle{alpha}
\bibliography{References}

\end{document}